\documentclass[review,3p,10pt]{elsarticle}
\biboptions{sort&compress}

\usepackage{amsmath,amssymb,amsthm}
\usepackage{bm,color}
\usepackage{stmaryrd}
\usepackage[T1]{fontenc}
\usepackage{float, caption, subcaption}
\usepackage{booktabs}
\usepackage{multirow}
\usepackage[bookmarks=true,
  bookmarksnumbered=true,
  bookmarksopen=true,
  colorlinks,
  pdfborder=001,
  linkcolor=black]{hyperref}

\allowdisplaybreaks
\newtheorem{definition}{Definition}[section]

\newtheorem{proposition}{Proposition}[section]
\newtheorem{remark}{Remark}[section]
\newtheorem{example}{Example}[section]

\numberwithin{equation}{section}
\numberwithin{figure}{section}
\numberwithin{table}{section}
\newcommand\bbR{\mathbb{R}}

\newcommand\bx{\bm{x}}
\newcommand\bV{\bm{V}}
\newcommand\bU{\bm{U}}
\newcommand\bF{\bm{F}}
\newcommand\bT{\bm{T}}

\newcommand\pd[2]{\dfrac{\partial {#1}}{\partial {#2}}}

\newcommand\mean[1]{\{\!\{ #1 \}\!\}}
\newcommand\jump[1]{\llbracket #1 \rrbracket}
\newcommand\jumpangle[1]{\langle\!\langle #1 \rangle\!\rangle}

\journal{Journal of Computational Physics}

\begin{document}

\begin{frontmatter}



\title{High-order accurate well-balanced energy stable adaptive moving mesh finite difference schemes for the shallow water equations with non-flat bottom topography}


\author[inst1]{Zhihao Zhang}
\affiliation[inst1]{organization={Center for Applied Physics and Technology, HEDPS and LMAM,
	School of Mathematical Sciences, Peking University},
            city={Beijing},
            postcode={100871},
            country={P.R. China}}
\ead{zhihaozhang@pku.edu.cn}

\author[inst2]{Junming Duan\corref{cor1}}
\affiliation[inst2]{organization={Chair of Computational Mathematics and Simulation Science, \'Ecole Polytechnique F\'ed\'erale de Lausanne},
            city={Lausanne},
            postcode={1015},
            country={Switzerland}}
\cortext[cor1]{Corresponding author}
\ead{junming.duan@epfl.ch}

\author[inst3,inst1]{Huazhong Tang}
\affiliation[inst3]{organization={Nanchang Hangkong University},
            city={Nanchang},
            postcode={330000},
            state={Jiangxi Province},
            country={P.R. China}}
\ead{hztang@math.pku.edu.cn}

\begin{abstract}
This paper proposes high-order accurate well-balanced (WB) energy stable (ES) adaptive moving mesh finite difference schemes for the shallow water equations (SWEs) with non-flat bottom topography.
To enable the construction of the ES schemes on moving meshes, a reformulation of the SWEs is introduced, with the bottom topography as an additional conservative variable that evolves in time.
The corresponding energy inequality is derived based on a modified energy function,
then the reformulated SWEs and energy inequality are transformed into curvilinear coordinates.
A two-point energy conservative (EC) flux is constructed,
and high-order EC schemes based on such a flux are proved to be WB that they preserve the lake at rest.
Then high-order ES schemes are derived by adding suitable dissipation terms to the EC schemes,
which are newly designed to maintain the WB and ES properties simultaneously.
The adaptive moving mesh strategy is performed by iteratively solving the Euler-Lagrangian equations of a mesh adaptation functional.
The fully-discrete schemes are obtained by using the explicit strong-stability preserving third-order Runge-Kutta method.
Several numerical tests are conducted to validate the accuracy, WB and ES properties, shock-capturing ability, and high efficiency of the schemes.

\end{abstract}


\end{frontmatter}


\section{Introduction}\label{section:Intro}
Shallow water equations (SWEs) describe a thin layer of free surface flow under the influence of gravity and bottom topology,
which have been widely used in the studies of atmospheric, river, and coastal flows, tsunamis, etc.
The two-dimensional (2D) SWEs with non-flat bottom topography defined in a time-space physical domain $(t,\bx)\in\bbR^+\times\Omega_p,~\Omega_p\subset \bbR^2,~\bx=(x_1,x_2)$ can be written as the following hyperbolic balance laws
\begin{equation}\label{eq:SWE}
  \left\{
    \begin{aligned}
      &~\pd{h}{t}+\pd{\left(hv_1\right)}{x_1}+\pd{\left(hv_2\right)}{x_2}=0, \\
      &\pd{\left(hv_1\right)}{t}+\pd{\left(h v_1^2+\frac{1}{2} g h^2\right)}{x_1}+\pd{\left(h v_1 v_2\right)}{x_2}=-g h \pd{b}{x_1}, \\
      &\pd{\left(hv_2\right)}{t}+\pd{\left(h v_1 v_2\right)}{x_1}+\pd{\left(h v_2^2+\frac{1}{2} g h^2\right)}{x_2}=-g h \pd{b}{x_2},
    \end{aligned}
  \right.
\end{equation}
where $h(t,\bx)$ is the water depth, $v_\ell(t,\bx)$ is the $x_\ell$-component of the fluid velocity, $\ell=1,2$,
and $g$ is the gravitational acceleration constant.
The source terms in \eqref{eq:SWE} involve the time-independent bottom topography $b(\bx)$.
Numerical simulation is important in the applications of SWEs,
and their numerical methods have been extensively studied in the literature,
e.g. \cite{AuDusse2004fast,Bermudez1994Upwind,Kuang2017Runge,Noelle2007High,Tang2004Solution,Tang2004Gas,Wu2016Newton,Xing2017Numerical,Xing2005High,Zhao2022Well}
and the references therein.

The SWEs \eqref{eq:SWE} admit nontrivial steady states due to the appearance of the source terms,
e.g. the lake at rest, where the flux gradients are balanced by the source terms.
Many physical phenomena such as waves on a lake or tsunamis in the deep ocean can be seen as small perturbations of the steady states,
thus it is of great significance to capture those small perturbations in the numerical simulations.
This raises the task to design the so-called well-balanced (WB) schemes
that can preserve the steady states in the discrete sense,
which then allows capturing the small perturbations well, even on coarse meshes.
Standard numerical methods are generally not WB,
so that one needs to adopt an extremely fine mesh to capture the perturbations accurately,
leading to prohibitive computational costs.
The WB property or the ``C-property'' was first illustrated in \cite{Bermudez1994Upwind}.
After that, various WB numerical methods for the SWEs were studied,
e.g. finite difference methods \cite{Li2020High,Vukovic2002ENO,Xing2005High}, finite volume methods \cite{AuDusse2004fast,Li2012Hybrid,Noelle2006Well,Noelle2007High}, and discontinuous Galerkin (DG) methods \cite{Xing2014Exactly,Xing2013Positivity,Zhang2021High}.
The readers are also referred to the review article \cite{Kurganov2018Finite} and the references therein.



The localized interesting structures in the solutions to the SWEs, such as shocks and sharp transitions,
usually need fine meshes to resolve.
In such cases, adaptive moving mesh methods become an effective way to improve the efficiency and quality of numerical solutions,
and have played an important role in solving partial differential equations, e.g.
\cite{Brackbill1993An,Brackbill1982Adaptive,CAO1999221,CENICEROS2001609,Davis1982,Miller1981,Ren2000An,Stockie2001,Tang2003Adaptive,Wang2004A,Winslow1967Numerical}.
For the numerical simulations of the SWEs,
the adaptive moving mesh kinetic flux-vector splitting method was proposed in \cite{Tang2004Solution}.
The work \cite{Lamby2005Solution} presented a fully adaptive multiscale finite volume method for the SWEs with source terms,
where the mesh generation is combined with B-spline methods.
A high-order indirect positivity-preserving WB adaptive moving mesh DG method was given in \cite{Zhang2021High},
in which the flow variables and bottom topography were interpolated from the old mesh to the new mesh at each time step using the same scheme,
and a direct moving mesh WB DG method based on hydrostatic reconstruction technique was studied in \cite{Zhang2022AWell}.

In the numerical simulations,
it is usually reasonable to require that the numerical schemes satisfy semi-discrete or discrete stability conditions,
often analogue to the stability of the solutions at the continuous level,
which helps to improve the stability of the schemes.
One important class of such methods is the energy stable (ES) numerical methods,
which have been extensively studied for elliptic and parabolic equations.
For the SWEs, the second-order WB ES schemes were constructed in \cite{Fjordholm2011Well} satisfying a semi-discrete energy inequality,
where the energy is also an entropy function of the system.

This paper focuses on the construction of high-order WB ES schemes for the SWEs on moving meshes,
which is based on the techniques in the construction of entropy stable schemes in curvilinear coordinates \cite{DUAN2021109949,Duan2022High},
with extra efforts to incorporate the WB property.
As this work focuses on the schemes on moving meshes,
the SWEs are considered to be transformed into curvilinear coordinates,
and the bottom topography should be updated due to mesh movement.
We choose to evolve the bottom topography as another conservative variable in time,
so that the two-point energy conservative (EC) flux can be constructed similarly to the entropy conservative schemes in \cite{Duan2021SWMHD,Duan2022High}.
Based on those considerations,
a reformulation of the SWEs is first introduced by adding the bottom topography as an additional conservative variable,
and corresponding energy inequality is derived based on a modified energy function.
For the modified system,
the specific expressions of a two-point EC flux are derived.
Furthermore, the high-order EC fluxes are constructed by using the linear combinations of the two-point case,
and the schemes are proved to be WB that they preserve the lake at rest.
With the help of the WENO reconstruction \cite{Borges2008An}, the high-order ES schemes are obtained by adding suitable dissipation terms,
consisting of two parts.
One is based on the entropy variables of the original SWEs \eqref{eq:SWE},
but with an additional zero as the last component.
It can be proved to achieve the WB and ES properties simultaneously,
but oscillations may appear when the bottom topography is discontinuous,
since the bottom topography is evolved by a transport equation
but no dissipation is added for the lake at rest.
Thus a second part is proposed to suppress the oscillations in such cases,
and also preserves the WB and ES properties at the same time.
The mesh adaptation is performed by iteratively solving the Euler-Lagrangian equations of a mesh adaptation functional \cite{Duan2022High,Li2022High},
and the monitor function is chosen to concentrate the mesh points near those interesting features.
Finally, the explicit strong-stability preserving (SSP) third-order Runge-Kutta (RK3) method is used to obtain the fully-discrete schemes.

The outline of this paper is as follows.
Section \ref{section:EntropyCondition}
presents a reformulation of the SWEs and corresponding energy inequality both in the Cartesian and curvilinear coordinates.
In Section \ref{section:2DHOScheme}, the high-order WB EC schemes are first constructed based on the two-point EC fluxes, then proved to be WB.
Further, the high-order WB ES schemes are developed by adding suitable dissipation terms to the EC schemes.
Some numerical tests are conducted in Section \ref{section:Result} to verify the high-order accuracy, WB and ES properties, shock-capturing ability, and efficiency of our schemes.
The conclusions are given in Section \ref{section:Conc}.

\section{Energy inequality for the SWEs}\label{section:EntropyCondition}
In this paper, the water depth $h$ is always assumed to be positive, i.e., dry area is not considered.

\subsection{Reformulation of the SWEs and corresponding energy inequality}
This paper focuses on the ES schemes on moving meshes,
so that the bottom topography $b$ should be updated in the time discretization due to mesh movement.
Generally speaking, there are two ways to do that:
$b$ is obtained by using the analytical expression,
or can be viewed as a time-dependent variable to be evolved simultaneously with the original SWEs.
Based on the first way, it is challenging to design a consistent two-point EC flux as one will see in Remark \ref{rmk:original_phi}.
That is why we propose to reformulate the SWEs by adding the bottom topography $b$ as an additional component in the conservative variables.

The SWEs \eqref{eq:SWE} can be cast in the equivalent form as
\begin{equation}\label{eq:SWE1}
	\frac{\partial \bm{U}}{\partial t} + \sum_{\ell=1}^2\frac{\partial \bm{F}_\ell(\bm{U})}{\partial{x}_\ell}
    = -gh\sum_{\ell=1}^{2}\frac{\partial \bm{B}_\ell}{\partial x_\ell}.
\end{equation}
Here the new conservative variables, physical fluxes, and source terms are
\begin{equation*}
    \begin{aligned}
	\bm{U} &= (h,hv_1,hv_2,b)^{\mathrm{T}},\\
	\bm{F}_1 &= (hv_1,hv_1^2+\frac{g}{2}h^2, hv_1v_2,0)^{\mathrm{T}},\\
	\bm{F}_2 &= (hv_2,hv_1v_2,hv_2^2+\frac{g}{2}h^2,0)^{\mathrm{T}},\\
	\bm{B}_1 &= (0,b,0,0)^{\mathrm{T}},\\
	\bm{B}_2 &= (0,0,b,0)^{\mathrm{T}},
    \end{aligned}
\end{equation*}
where the extra components are zeros for the physical fluxes and source terms.
For the system \eqref{eq:SWE1}, define the function pair $(\eta, q_\ell)$ as
\begin{equation}\label{eq:EntropyPair}
    \begin{aligned}
	\eta(\bU)&:=\frac{1}{2} h (v_1^2+v_2^2) + \frac{1}{2} g h^2 + ghb + \gamma gb^2,\\
	q_\ell(\bU)&:=\frac{1}{2}hv_\ell\left(v_1^2+v_2^2\right) + gh^2 v_\ell + ghbv_\ell,\quad \ell=1,2,
    \end{aligned}
\end{equation}
and
\begin{equation}\label{eq:EnergyVar}
    \bm{V}(\bU):= \left(\frac{\partial \eta(\bU)} {\partial \bm{U}}\right)^\mathrm{T}
    = \left(g(h+b) - \frac{v_1^2+v_2^2}{2}, v_1,v_2, gh+2\gamma gb\right)^{\mathrm{T}}.
\end{equation}
One can verify that
\begin{equation}\label{eq:energy_consistent}
     \pd{q_{\ell}}{\bU} - \bV^{\mathrm{T}} \pd{\bF_{\ell}}{\bU} = (0,0,0,ghv_\ell), ~ \ell=1,2.
\end{equation}
For smooth solutions, The energy identity can be obtained by taking the dot product of the system \eqref{eq:SWE1} and $\bV$, i.e.
\begin{equation*}
    \bV^{\mathrm{T}}\frac{\partial \bm{U}}{\partial t} + \sum_{\ell=1}^2\bV^{\mathrm{T}}\frac{\partial \bm{F}_\ell(\bm{U})}{\partial{x}_\ell}
    = -gh\sum_{\ell=1}^{2}\bV^{\mathrm{T}}\frac{\partial \bm{B}_\ell}{\partial x_\ell},
\end{equation*}
which is simplified by using \eqref{eq:energy_consistent} as
\begin{equation*}
   \pd{\eta}{t}+\sum_{\ell=1}^{2}\left(\pd{q_{\ell}}{\bU}-(0,0,0,ghv_\ell)\right)\pd{\bU}{x_{\ell}} = -\sum_{\ell=1}^{2} ghv_{\ell}\pd{b}{x_{\ell}}.
\end{equation*}
Due to the fact $(0,0,0,ghv_\ell)\pd{\bU}{x_{\ell}} = ghv_{\ell}\pd{b}{x_{\ell}}$,
it reduces to
\begin{equation*}
    \pd{\eta}{t} + \sum_{\ell=1}^2 \pd{q_\ell}{x_\ell} = 0.
\end{equation*}
When the solutions contain discontinuities,
the identity becomes the following energy inequality
\begin{equation}\label{eq:EntropyIneq2D}
    \pd{\eta}{t} + \sum_{\ell=1}^2 \pd{q_\ell}{x_\ell} \leqslant 0,
\end{equation}
in the sense of distribution.
The energy inequality endows stability in the system \eqref{eq:SWE1}.

\begin{remark}\rm
    To ensure the transformation between $\bU$ and $\bV$ is bijective,
    the Jacobian matrix $\partial\bV/\partial\bU$ should be invertible.
    In other words, the Hessian matrix
    \begin{equation}\label{eq:Hessian}
        \dfrac{\partial^2\eta(\bU)}{\partial\bU^2} = \dfrac{1}{h}\begin{bmatrix}
	   {gh+v_1^2+v_2^2}& {-v_1} & {-v_2} & gh \\
	   {-v_1} &  {1} &  0 & 0\\
	   {-v_2} & 0 & {1}  & 0\\
          gh & 0&0 & 2 \gamma gh
        \end{bmatrix}
    \end{equation}
    is invertible,
    i.e., $\gamma \not= 1/2$,
\end{remark}

\begin{remark}\rm\label{rmk:EntropyForOriSWEs}
    This work is concerned with the ES schemes for the system \eqref{eq:SWE1},
    but one may notice that the function pair $(\eta, q_\ell)$ in \eqref{eq:EntropyPair} is closely related to the entropy pair for the original SWEs \eqref{eq:SWE}.
    Indeed, express the original SWEs \eqref{eq:SWE} as
    \begin{equation}\label{eq:SWE0}
        \frac{\partial\widehat{\bm{U}}}{\partial t}+\sum_{\ell=1}^2 \frac{\partial \widehat{\bm{F}}_{\ell}(\widehat{\bm{U}})}{\partial x_{\ell}}
       =-g h\sum_{\ell=1}^2  \frac{\partial \widehat{\bm{B}}_{\ell}}{\partial x_{\ell}},
    \end{equation}
    where $\widehat{\bm{U}}, \widehat{\bm{F}}_\ell, \widehat{\bm{B}}_\ell$ are the first three components of $\bU, \bF_\ell, \bm{B}_\ell$,
    then an entropy pair is
    \begin{equation}\label{eq:Ori_EntropyPair}
       \begin{aligned}
    	\widehat{\eta}(\widehat{\bm{U}})&:=\frac{1}{2} h (v_1^2+v_2^2) + \frac{1}{2} g h^2 + ghb,\\
    	\widehat{q}_\ell(\widehat{\bm{U}})&:=\frac{1}{2}hv_\ell\left(v_1^2+v_2^2\right) + gh^2 v_\ell + ghbv_\ell,\quad \ell=1,2,
      \end{aligned}
    \end{equation}
    where $\widehat{\eta}(\widehat{\bm{U}})$ is the convex entropy function or the total energy,
    with the entropy variables
    \begin{equation}\label{eq:original_V}
     \widehat{\bV} =    (\partial\widehat{\eta}/\partial\widehat{\bm{U}})^{\mathrm{T}} 
      = \left(g(h+b) - \dfrac{v_1^2+v_2^2}{2}, v_1, v_2\right)^{\mathrm{T}},
    \end{equation}
    as the first three components of $\bV$.
    The dissipation terms in the high-order ES schemes in Section \ref{section:ES} will be constructed with the help of the entropy pair \eqref{eq:Ori_EntropyPair}.
    It is seen that $\eta$ in \eqref{eq:EntropyPair} differs from $\widehat{\eta}$ by $\gamma gb^2$,
    thus $\eta$ is a modified energy function in this paper.
\end{remark}

\begin{remark}\rm
    The construction of the schemes in this paper is based on the entropy stable schemes in the literature,
    but the function pair in \eqref{eq:EntropyPair} is not an entropy pair for the reformulated SWEs \eqref{eq:SWE1}.
    To be specific, although $\eta(\bU)$ is convex as the Hessian matrix \eqref{eq:Hessian}
    is positive-definite for $\gamma > 1/2$,
    $(\eta, q_\ell)$ does not satisfy the consistent condition due to \eqref{eq:energy_consistent}.
    And the system \eqref{eq:SWE1} under the change of variables $\bU\rightarrow\bV$
    \begin{equation*}
        \pd{\bU}{\bV}\pd{\bV}{t}
        + \left(\pd{\bF_1}{\bV}
        + (0,gh,0,0)^\mathrm{T}\pd{b}{\bV} \right)\pd{\bV}{x_1}
        + \left(\pd{\bF_2}{\bV}
        + (0,0,gh,0)^\mathrm{T}\pd{b}{\bV} \right)\pd{\bV}{x_2} = 0,
    \end{equation*}
    is not symmetric,
    since
    \begin{equation*}
    \begin{aligned}
       \pd{\bF_1}{\bV}
       + &(0,gh,0,0)^\mathrm{T}\pd{b}{\bV}
       =
       \\
       &\dfrac{\gamma}{g(2\gamma-1)}
       \begin{bmatrix}
    	2 v_1&  \frac{2\gamma v_1^2 - g h + 2\gamma g h}{\gamma}& 2 v_1 v_2 & -\frac{v_1}{\gamma}
    \\
	\frac{2\gamma v_1^2 - g h + 2\gamma g h}{\gamma}
    &  \frac{v_1(2\gamma v_1^2 - 3g h + 6 \gamma g h)}{\gamma}
    & \frac{v_2(2\gamma v_1^2 - g h + 2 \gamma g h)}{\gamma}
    & -\frac{v_1^2}{\gamma}
    \\
    	2 v_1 v_2
    & \frac{v_2(2\gamma v_1^2 - g h + 2 \gamma g h)}{\gamma}
    &  \frac{v_1(2\gamma v_2^2 - g h + 2 \gamma g h)}{\gamma}  &    -\frac{v_1 v_2}{\gamma}\\
           0& 0&0 &0 \\
        \end{bmatrix}
        \end{aligned}
    \end{equation*}
    is not symmetric.
\end{remark}

\begin{remark}\rm
    The system \eqref{eq:SWE1} cannot be symmetrized by adding suitable source terms similar to the MHD \cite{Godunov1972,Powell1994} or RMHD \cite{Wu2020Entropy,Duan2020RMHD} cases,
    because the source terms added there vanish when the divergence-free conditions are satisfied.
\end{remark}

Two auxiliary variables $\phi$ and $\psi_\ell$ are introduced as
\begin{equation}\label{eq:EntropyPotential}
    \begin{aligned}
	\phi &:= \bm{V}^{\mathrm{T}}\bm{U} - \eta(\bm{U}) = \frac{1}{2}gh^2 + ghb+\gamma gb^2,\\
     \psi_\ell &:= \bm{V}^{T}\bm{F}_\ell(\bm{U}) - q_\ell(\bm{U}) + ghv_\ell b = \frac{1}{2}gh^2v_\ell + ghv_\ell b,\quad \ell = 1,2,
\end{aligned}
\end{equation}
which are essential in constructing the two-point EC flux for our schemes, see Definition \ref{def:sufficient_condition}.

\subsection{Reformulated SWEs and the energy inequality in curvilinear coordinates}
To derive the energy inequality in curvilinear coordinates,
define $\Omega_c$ as the computational domain, which will also be used in the mesh redistribution.
Assume that there is a time-dependent, differentiable coordinate transformation $t = \tau, \bx = \bx(\tau, \bm{\xi})$ from the computational domain $\Omega_{c}$ to the physical domain $\Omega_{p}$,
then the adaptive mesh can be generated from the reference mesh in $\Omega_c$ based on such a transformation.
The determinant of the Jacobian matrix is
\begin{equation}
    J = \det\left(\frac{\partial (t,\bx)}{\partial (\tau,\bm{\xi})}\right) =
    \left|\begin{matrix}
		1 & 0 &0 \\
		\dfrac{\partial x_1}{\partial \tau} & \dfrac{\partial x_1}{\partial \xi_1} &  \dfrac{\partial x_1}{\partial \xi_2}\\
		\dfrac{\partial x_2}{\partial \tau} & \dfrac{\partial x_2}{\partial \xi_1} &  \dfrac{\partial x_2}{\partial \xi_2}\\
	\end{matrix}\right|,\nonumber
\end{equation}
and the mesh metrics introduced by the transformation satisfy the geometric conservation laws (GCLs) consisting of the volume conservation law (VCL) and the surface conservation laws (SCLs)
\begin{equation}\label{eq:GCL_2D}
    \begin{aligned}
	&\text { VCL: } \quad \frac{\partial J}{\partial \tau}+ \sum_{\ell=1}^{2}\frac{\partial}{\partial \xi_\ell}\left(J \frac{\partial \xi_\ell}{\partial t}\right)=0,\\
    &\text { SCLs: } \quad \sum_{\ell = 1}^2\frac{\partial}{\partial \xi_\ell}\left(J \frac{\partial \xi_\ell}{\partial x_k}\right)=0,~k = 1,2.
    \end{aligned}
\end{equation}
With the help of the GCLs, the SWEs \eqref{eq:SWE1} and the energy inequality \eqref{eq:EntropyIneq2D} can be written in curvilinear coordinates as follows
\begin{align}
&\pd{\bm{\mathcal{U}}}{\tau}
+\sum_{\ell=1}^2\pd{\bm{\mathcal{F}}_{\ell}}{\xi_\ell}
=-gh
\sum_{\ell=1}^2\pd{\bm{\mathcal{B}}_{\ell}}{\xi_\ell},\label{eq:SWECurv}\\
&\dfrac{\partial\mathcal{E}}{\partial \tau}
+\sum_{\ell=1}^2\dfrac{\partial \mathcal{Q}_{\ell}}{\partial \xi_\ell}\leqslant0,\nonumber
\end{align}
where the energy identity only holds for the smooth solutions,
and the notations are defined as
\begin{align*}
    &\bm{\mathcal{U}}=J \bm{U},~
    \bm{\mathcal{F}}_{\ell}
    =\left(J \dfrac{\partial \xi_\ell}{\partial t} \right)\bm{U}
    +\sum_{k=1}^2\left(J \dfrac{\partial \xi_\ell}{\partial x_k}\right) \bm{F}_{k}, ~\bm{\mathcal{B}}_{\ell}=\sum_{k=1}^2\left(J \dfrac{\partial \xi_\ell}{\partial x_k}\right) \bm{B}_{k},\\
    &\mathcal{E}=J\eta,~
    \mathcal{Q}_{\ell}=\left(J \dfrac{\partial \xi_\ell}{\partial t} \right)\eta
    +\sum_{k=1}^2\left(J \dfrac{\partial \xi_\ell}{\partial x_k}\right) q_{k},~\ell=1,2.
\end{align*}

\section{Numerical schemes}\label{section:2DHOScheme}
This section constructs the high-order accurate WB EC and ES schemes for the 2D SWEs in curvilinear coordinates \eqref{eq:SWECurv},
based on the entropy conservative and entropy stable schemes \cite{Duan2022High}.
The specific expressions of the two-point EC fluxes in curvilinear coordinates are derived
and the high-order accurate EC schemes built on such fluxes are proved to be WB.
To further obtain the high-order WB ES schemes,
high-order dissipation terms are proposed to maintain the WB and ES properties at the same time and can suppress oscillations for discontinuous bottom topography on moving meshes.

\subsection{Two-point EC fluxes}
To construct high-order accurate EC schemes,
one major effort is to derive the two-point EC flux.
\begin{definition}\rm\label{def:sufficient_condition}
  For the SWEs in curvilinear coordinates \eqref{eq:SWECurv},
  a consistent two-point numerical flux $\widetilde{\bm{\mathcal{F}}}_{\ell}\left(\bm{U}_{L},\bm{U}_{R},\left(J\dfrac{\partial \xi_{\ell}}{\partial \zeta}\right)_{L},\left(J\dfrac{\partial \xi_{\ell}}{\partial \zeta}\right)_{R}\right)$ with $\zeta = t,x_1,x_2$ satisfying
  \begin{align}
    \left(\bm{V}\left(\bU_{R}\right)-\bm{V}\left(\bU_{L}\right)\right)^{\mathrm{T}}\widetilde{\bm{\mathcal{F}}}_{\ell}
    = \ &~\frac{1}{2}\left(\left(J \frac{\partial \xi_\ell}{\partial t}\right)_{L}+\left(J \frac{\partial \xi_\ell}{\partial t}\right)_{R}\right)\left({\phi}\left(\bU_{R}\right)-{\phi}\left(\bU_{L}\right)\right)\nonumber\\
    &+ \sum_{k=1}^2\frac{1}{2}\left(\left(J \frac{\partial \xi_\ell}{\partial x_{k}}\right)_L+\left(J \frac{\partial \xi_\ell}{\partial x_{k}}\right)_R\right)\left(\psi_{k}\left(\bU_R\right)-\psi_{k}\left(\bU_R\right)\right)\nonumber\\
    &
    - \sum_{k=1}^2\frac{g}{4}\left(\left(J \frac{\partial \xi_\ell}{\partial x_{k}}\right)_L+\left(J \frac{\partial \xi_\ell}{\partial x_{k}}\right)_R\right)\left(\left(hv_{k}\right)_{R}-\left(hv_{k}\right)_{L}\right)\left(b_{L}+b_{R}\right),
    \label{eq:EC_condition_1}
  \end{align}
  is called the two-point EC flux.
\end{definition}

One can choose the following two-point flux which satisfies the condition \eqref{eq:EC_condition_1}
\begin{align*}
  \widetilde{\bm{\mathcal{F}}}_{\ell}
  = \ &\frac{1}{2}\left(\left(J\frac{\partial \xi_\ell}{\partial t}\right)_{L}+\left(J\frac{\partial \xi_\ell}{\partial t}\right)_{R}\right)\widetilde{\bm{U}}+ \sum_{k=1}^{2} \frac{1}{2}\left(\left(J\frac{\partial \xi_\ell}{\partial x_k}\right)_{L}+\left(J\frac{\partial \xi_\ell}{\partial x_k}\right)_{R}\right)\bm{\widetilde{F}}_{k},
\end{align*}
where $\widetilde{\bm{U}}$ is the temporal two-point EC flux consistent with $\bU$ and satisfies
\begin{equation}\label{eq:tilde_U_2D}
  \left(\bm{V}_{R}-\bm{V}_{L}\right)^{\mathrm{T}} \widetilde{\bm{U}}=\phi_{R}-\phi_{L},
\end{equation}
and $\bm{\widetilde{F}}_{k}$ is the two-point EC flux in the Cartesian coordinates consistent with $\bF_k$ and satisfies
\begin{align}
  \left(\bm{V}_{R}-\bm{V}_{L}\right)^{\mathrm{T}} \bm{\widetilde{F}}_{k}=\big(\left(\psi_{k}\right)_{R}-\left(\psi_{k}\right)_{L}\big)  -\frac{g}{2}\big(\left(hv_{k}\right)_{R}-\left(hv_{k}\big)_{L}\right)\left(b_{L}+b_{R}\right),
  \label{eq:tilde_F_1}
\end{align}
with $\phi$ and $\psi_k$ given in \eqref{eq:EntropyPotential}.
In this paper, the notations $\mean{a}= (a_{L}+a_{R})/2$
and $\jump{a} = (a_{R}-a_{L})$
represent the average and jump of $a$, respectively.

\begin{proposition}\rm
  The two-point EC fluxes $\bm{\widetilde{{U}}}$, $\bm{\widetilde{F}}_1$ and $\bm{\widetilde{F}}_2$ can be chosen as
  \begin{equation*}
    \widetilde{\bm{U}}=\left(\begin{array}{c}
        \mean{h} \\
        \mean{h}\mean{v_1} \\
        \mean{h}\mean{v_2} \\
        \mean{b}
    \end{array}\right),
  \end{equation*}

  \begin{equation*}
    \widetilde{\bm{F}}_1=\left(\begin{array}{c}
        \mean{h}\mean{v_1} \\
        \mean{h}\mean{v_1}^2 + \frac{g}{2}\mean{h^2}+g\left(\mean{hb} - \mean{h}\mean{b}\right) \\
        \mean{h}\mean{v_1}\mean{v_2}\\
        0
    \end{array}\right),
  \end{equation*}

  \begin{equation*}
    \widetilde{\bm{F}}_2=\left(\begin{array}{c}
        \mean{h}\mean{v_2} \\
        \mean{h}\mean{v_1}\mean{v_2}\\
        \mean{h}\mean{v_2}^2 + \frac{g}{2}\mean{h^2}+g\left(\mean{hb} - \mean{h}\mean{b}\right) \\

        0
    \end{array}\right).
  \end{equation*}
\end{proposition}

\begin{proof}
  Utilizing the fact $\jump{ab} = \jump{a}\mean{b}+\mean{a}\jump{b}$ one can rewrite the jump of $\bm{V}$ in \eqref{eq:EnergyVar} as
  \begin{equation*}
    \jump{\bm{V}} = \begin{pmatrix}
      \jump{g(h+b)-\frac{v_1^2+v_2^2}{2}} = g\jump{h} + g\jump{b} - \mean{v_1}\jump{v_1} - \mean{v_2}\jump{v_2}\\
      \jump{v_1}\\
      \jump{v_2}\\
      \jump{gh+2\gamma gb} = g\jump{h}+2\gamma g\jump{b}
    \end{pmatrix},
  \end{equation*}
  while the jump of $\phi$ can be expressed as
  \begin{equation*}
    \jump{\phi} = \jump{\frac{gh^2}{2}+ghb+\gamma gb^2}
    = g\mean{h}\jump{h}+g\jump{hb}+\gamma g\jump{b^2}
    = g\left(\mean{h} + \mean{b}\right)\jump{h}+g\mean{h}\jump{b}+2\gamma g\mean{b}\jump{b}.
  \end{equation*}
  Substituting them into \eqref{eq:tilde_U_2D} and matching the coefficients of the same jump terms on both sides yields
  \begin{equation*}
    \widetilde{\bm{U}}=\left(\begin{array}{c}
        \mean{h} \\
        \mean{h}\mean{v_1} \\
        \mean{h}\mean{v_2} \\
        \mean{b}
    \end{array}\right).
  \end{equation*}
  For the construction of $\widetilde{\bm{F}}_1$,
  based on the condition \eqref{eq:tilde_F_1},
  one can proceed as follows,
  \begin{align*}
    &\jump{\psi_1} = \jump{ghv_1b}+\jump{\frac{gh^2v_1}{2}}= \jump{ghv_1b}+\frac{g}{2}\mean{h^2}\jump{v_1}+g\mean{h}\mean{v_1}\jump{h},
    \\
    &g\jump{hv_1b}
    = g\mean{hb}\jump{v_1}
    + g\jump{hb}\mean{v_1}
    = g\mean{hb}\jump{v_1}
    +g\mean{b}\mean{v_1}\jump{h}
    +g\mean{h}\mean{v_1}\jump{b},
    \\
    &
    g\jump{hv_1}\mean{b}
    = g\mean{h}\mean{b}\jump{v_1}
    +g\mean{b}\mean{v_1}\jump{h},
  \end{align*}
  and obtain the flux
  \begin{equation*}
    \widetilde{\bm{F}}_1=\left(\begin{array}{c}
        \mean{h}\mean{v_1} \\
        \mean{h}\mean{v_1}^2 + \frac{g}{2}\mean{h^2}+g\left(\mean{hb} - \mean{h}\mean{b}\right) \\
        \mean{h}\mean{v_1}\mean{v_2}\\
        0
    \end{array}\right).
  \end{equation*}
  The expressions of $\widetilde{\bm{F}}_2$ can be derived similarly.
  It is easy to check the consistency of $\widetilde{\bm{U}}, \widetilde{\bm{F}}_k$.
\end{proof}

\begin{remark}\rm\label{rmk:original_phi}
  The first three components of $\widetilde{\bm{F}}_k$ are the same as those in \cite{Duan2021SWMHD} with zero magnetic fields in the Cartesian coordinates,
  where $\widetilde{\bU}$ does not appear,
  which is newly derived as this paper considers the moving mesh schemes.
  If considering the original SWEs \eqref{eq:SWE0} and entropy pair \eqref{eq:Ori_EntropyPair},
  the expressions of $\widehat{\bV}$ \eqref{eq:original_V} depend on $b$
  but $\widehat{\phi}=\widehat{\bV}^\mathrm{T}\widehat{\bU} - \widehat{\eta} = {gh^2}/{2}$ does not depend on $b$,
  so that it is impossible to construct a consistent two-point EC flux $\widetilde{\bU}$ satisfying the condition \eqref{eq:tilde_U_2D},
  which motivates us to reformulate the SWEs.
\end{remark}

\subsection{High-order WB EC schemes}
Assume that a uniform Cartesian mesh is taken as the computational mesh
$(\xi_1)_i = a_1 + i\Delta\xi_1,~i=0,1,\cdots,N_1-1,~\Delta\xi_1 = (b_1-a_1)/(N_1-1)$,
$(\xi_2)_j = a_2 + j\Delta\xi_2,~j=0,1,\cdots,N_2-1,~\Delta\xi_2 = (b_2-a_2)/(N_2-1)$.
The 2D SWEs in curvilinear coordinates \eqref{eq:SWECurv} and the GCLs \eqref{eq:GCL_2D} are discretized by using the  $2p$th-order semi-discrete conservative finite difference schemes
\begin{align}
  \dfrac{\mathrm{d}}{\mathrm{d} t} \bm{\mathcal{U}}_{i, j}=
  & -\dfrac{1}{\Delta \xi_1}\left(\left(\bm{\widetilde{\mathcal{F}}}_{1}\right)_{i+\frac{1}{2}, j}^{\tt 2pth}-\left(\bm{\widetilde{\mathcal{F}}}_{1}\right)_{i-\frac{1}{2}, j}^{\tt 2pth}\right)-\frac{1}{\Delta \xi_2}\left(\left(\bm{\widetilde{\mathcal{F}}}_{2}\right)_{i, j+\frac{1}{2}}^{\tt 2pth}-\left(\bm{\widetilde{\mathcal{F}}}_{2}\right)_{i, j-\frac{1}{2}}^{\tt 2pth}\right)\nonumber
  \\
  &- \frac{gh_{i,j}}{\Delta \xi_1}\left(\left(\bm{\widetilde{\mathcal{B}}}_{1}\right)_{i+\frac{1}{2}, j}^{{\tt 2pth}}-\left(\bm{\widetilde{\mathcal{B}}}_{1}\right)_{i-\frac{1}{2}, j}^{{\tt 2pth}}\right)
  -\frac{gh_{i,j}}{\Delta \xi_2}\left(\left(\bm{\widetilde{\mathcal{B}}}_{2}\right)_{i, j+\frac{1}{2}}^{{\tt 2pth}}-\left(\bm{\widetilde{\mathcal{B}}}_{2}\right)_{i, j-\frac{1}{2}}^{{\tt 2pth}}\right),\label{eq:2D_HighOrder_EC_Discrete}
\end{align}
\begin{align}
  &\frac{\mathrm{d}}{\mathrm{d} t} J_{i, j}=-\frac{1}{\Delta \xi_1}\left(\left(\widetilde{J \frac{\partial \xi_1}{\partial t}}\right)_{i+\frac{1}{2}, j}^{\tt 2pth}-\left(\widetilde{J \frac{\partial \xi_1}{\partial t}}\right)_{i-\frac{1}{2}, j}^{\tt 2pth}\right)-\frac{1}{\Delta \xi_2}\left(\left(\widetilde{J \frac{\partial \xi_2}{\partial t}}\right)_{i, j+\frac{1}{2}}^{2 p \mathrm{th}}-\left(\widetilde{J \frac{\partial \xi_2}{\partial t}}\right)_{i, j-\frac{1}{2}}^{\tt 2pth}\right),\label{eq:2D_VCL_HighOrder}
  \\
  &\frac{1}{\Delta \xi_1}\left(\left(\widetilde{J \frac{\partial \xi_1}{\partial x_1}}\right)_{i+\frac{1}{2},j}^{{\tt 2pth}}-\left(\widetilde{J \frac{\partial \xi_1}{\partial x_1}}\right)_{i-\frac{1}{2},j}^{{\tt 2pth}}\right) + \frac{1}{\Delta \xi_2}\left(\left(\widetilde{J \frac{\partial \xi_2}{\partial x_1}}\right)_{i, j+\frac{1}{2}}^{{\tt 2pth}}-\left(\widetilde{J \frac{\partial \xi_2}{\partial x_1}}\right)_{i,j-\frac{1}{2}}^{{\tt 2pth}}\right)=0,\label{eq:2D_SCL_HighOrder_1}
  \\
  &\frac{1}{\Delta \xi_1}\left(\left(\widetilde{J \frac{\partial \xi_1}{\partial x_2}}\right)_{i+\frac{1}{2},j}^{{\tt 2pth}}-\left(\widetilde{J \frac{\partial \xi_1}{\partial x_2}}\right)_{i-\frac{1}{2},j}^{{\tt 2pth}}\right) + \frac{1}{\Delta \xi_2}\left(\left(\widetilde{J \frac{\partial \xi_2}{\partial x_2}}\right)_{i, j+\frac{1}{2}}^{{\tt 2pth}}-\left(\widetilde{J \frac{\partial \xi_2}{\partial x_2}}\right)_{i,j-\frac{1}{2}}^{{\tt 2pth}}\right)=0,\label{eq:2D_SCL_HighOrder_2}
\end{align}
where $\bm{\mathcal{U}}_{i,j}=J_{i,j}\bU_{i,j}$ and $J_{i,j}$ are the approximations of the point values of $J\bU$ and $J$ at $((\xi_1)_i, (\xi_2)_j)$, respectively,
and the numerical fluxes are defined as
\begin{equation}\label{eq:HighOrder_EC_Flux}
  \begin{aligned}
    \left(\bm{\widetilde{\mathcal{F}}}_{1}\right)_{i+\frac{1}{2}, j}^{{\tt 2pth}}=\sum_{m=1}^{p} \alpha_{p,m} \sum_{s=0}^{m-1} \Bigg[ &\frac{1}{2}\left(\left(J \frac{\partial \xi_1}{\partial t}\right)_{i-s, j}+\left(J \frac{\partial \xi_1}{\partial t}\right)_{i-s+m, j}\right) \widetilde{\bm{U}}\left(\bm{U}_{i-s, j}, \bm{U}_{i-s+m, j}\right) \\
      +& \frac{1}{2}\left(\left(J \frac{\partial \xi_1}{\partial x_1}\right)_{i-s, j}+\left(J \frac{\partial \xi_1}{\partial x_1}\right)_{i-s+m, j}\right) \widetilde{\bm{F}}_{1}\left(\bm{U}_{i-s, j}, \bm{U}_{i-s+m, j}\right) \\
    +& \frac{1}{2}\left(\left(J \frac{\partial \xi_1}{\partial x_2}\right)_{i-s, j}+\left(J \frac{\partial \xi_1}{\partial x_2}\right)_{i-s+m, j}\right) \widetilde{\bm{F}}_{2}\left(\bm{U}_{i-s, j}, \bm{U}_{i-s+m, j}\right)\Bigg],
    \\
    \left(\bm{\widetilde{\mathcal{F}}}_{2}\right)_{i, j+\frac{1}{2}}^{{\tt 2pth}}=\sum_{m=1}^{p} \alpha_{p,m} \sum_{s=0}^{m-1}  \Bigg[&\frac{1}{2}\left(\left(J \frac{\partial \xi_2}{\partial t}\right)_{i, j-s}+\left(J \frac{\partial \xi_2}{\partial t}\right)_{i, j-s+m}\right) \widetilde{\bm{U}}\left(\bm{U}_{i, j-s}, \bm{U}_{i, j-s+m}\right) \\
      +& \frac{1}{2}\left(\left(J \frac{\partial \xi_2}{\partial x_1}\right)_{i, j-s}+\left(J \frac{\partial \xi_2}{\partial x_1}\right)_{i, j-s+m}\right) \widetilde{\bm{F}}_{1}\left(\bm{U}_{i, j-s}, \bm{U}_{i, j-s+m}\right) \\
    +&\frac{1}{2}\left(\left(J \frac{\partial \xi_2}{\partial x_2}\right)_{i, j-s}+\left(J \frac{\partial \xi_2}{\partial x_2}\right)_{i, j-s+m}\right) \widetilde{\bm{F}}_{2}\left(\bm{U}_{i, j-s}, \bm{U}_{i, j-s+m}\right)\Bigg],
  \end{aligned}
\end{equation}
\begin{equation}\label{eq:HighOrder_EC_Flux_B}
  \begin{aligned}
    \left(\bm{\widetilde{\mathcal{B}}}_{1}\right)_{i+\frac{1}{2},j}^{{\tt 2pth}}=\sum_{m=1}^{p} \alpha_{p,m} \sum_{s=0}^{m-1} \Bigg[&\frac{1}{4}\left(\left(J \frac{\partial \xi_1}{\partial x_1}\right)_{i-s,j}+\left(J \frac{\partial \xi_1}{\partial x_1}\right)_{i-s+m,j}\right)\left(\left(\bm{B}_{1}\right)_{i-s,j}+\left(\bm{B}_{1}\right)_{i-s+m,j}\right) \\
    +&\frac{1}{4}\left(\left(J \frac{\partial \xi_1}{\partial x_2}\right)_{i-s,j}+\left(J \frac{\partial \xi_1}{\partial x_2}\right)_{i-s+m,j}\right)\left(\left(\bm{B}_{2}\right)_{i-s,j}+\left(\bm{B}_{2}\right)_{i-s+m,j}\right)\Bigg],
    \\
    \left(\bm{\widetilde{\mathcal{B}}}_{2}\right)_{i, j+\frac{1}{2}}^{{\tt 2pth}}=\sum_{m=1}^{p} \alpha_{p,m} \sum_{s=0}^{m-1} \Bigg[&\frac{1}{4}\left(\left(J \frac{\partial \xi_2}{\partial x_1}\right)_{i, j-s}+\left(J \frac{\partial \xi_2}{\partial x_1}\right)_{i, j-s+m}\right)\left(\left(\bm{B}_{1}\right)_{i, j-s}+\left(\bm{B}_{1}\right)_{i, j-s+m}\right) \\
    +&\frac{1}{4}\left(\left(J \frac{\partial \xi_2}{\partial x_2}\right)_{i, j-s}+\left(J \frac{\partial \xi_2}{\partial x_2}\right)_{i, j-s+m}\right)\left(\left(\bm{B}_{2}\right)_{i, j-s}+\left(\bm{B}_{2}\right)_{i, j-s+m}\right)\Bigg].
  \end{aligned}
\end{equation}
The discrete mesh metrics are given by
\begin{align}
  & \left(\widetilde{J \frac{\partial \xi_1}{\partial t}}\right)_{i+\frac{1}{2}, j}^{\tt 2pth}=\sum_{m=1}^p \alpha_{p,m} \sum_{s=0}^{m-1} \frac{1}{2}\left(\left(J \frac{\partial \xi_1}{\partial t}\right)_{i-s, j}+\left(J \frac{\partial \xi_1}{\partial t}\right)_{i-s+m, j}\right),
  \label{eq:2D_VCL_flux_1}\\
  & \left(\widetilde{J \frac{\partial \xi_2}{\partial t}}\right)_{i, j+\frac{1}{2}}^{\tt 2pth}=\sum_{m=1}^p \alpha_{p,m} \sum_{s=0}^{m-1} \frac{1}{2}\left(\left(J \frac{\partial \xi_2}{\partial t}\right)_{i, j-s}+\left(J \frac{\partial \xi_2}{\partial t}\right)_{i, j-s+m}\right),
  \label{eq:2D_VCL_flux_2}\\
  &\left(\widetilde{J \frac{\partial \xi_1}{\partial x_k}}\right)_{i+\frac{1}{2},j}^{{\tt 2pth}} =
  \sum_{m=1}^{p} \alpha_{p,m} \sum_{s=0}^{m-1} \frac{1}{2}\left(\left(J \frac{\partial \xi_1}{\partial x_k}\right)_{i-s,j}+\left(J \frac{\partial \xi_1}{\partial x_k}\right)_{i-s+m,j}\right),
  \label{eq:2D_SCL_flux_1}\\
  &\left(\widetilde{J \frac{\partial \xi_2}{\partial x_k}}\right)_{i,j+\frac{1}{2}}^{{\tt 2pth}} =
  \sum_{m=1}^{p} \alpha_{p,m} \sum_{s=0}^{m-1} \frac{1}{2}\left(\left(J \frac{\partial \xi_2}{\partial x_k}\right)_{i,j-s}+\left(J \frac{\partial \xi_2}{\partial x_k}\right)_{i,j-s+m}\right)\label{eq:2D_SCL_flux_2},
\end{align}
where
\begin{equation}
  \begin{aligned}\label{eq:Grid_Matrix_Coeff}
    & \left(J \frac{\partial \xi_1}{\partial t}\right)_{i, j}=-(\dot{x}_1)_{i, j}\left(J \frac{\partial \xi_1}{\partial x_1}\right)_{i, j}-(\dot{x}_2)_{i, j}\left(J \frac{\partial \xi_1}{\partial x_2}\right)_{i, j},
    \\
    & \left(J \frac{\partial \xi_2}{\partial t}\right)_{i, j}=-(\dot{x}_1)_{i, j}\left(J \frac{\partial \xi_2}{\partial x_1}\right)_{i, j}-(\dot{x}_2)_{i, j}\left(J \frac{\partial \xi_2}{\partial x_2}\right)_{i, j},
    \\
    & \left(J \frac{\partial \xi_1}{\partial x_1}\right)_{i, j}=+\sum_{m=1}^p \frac{\alpha_{p,m}}{2
    \Delta\xi_2}
    \left(\left(x_2\right)_{i, j+m}-\left(x_2\right)_{i, j-m}\right),
    \\
    & \left(J \frac{\partial \xi_1}{\partial x_2}\right)_{i, j}=-\sum_{m=1}^p \frac{\alpha_{p,m}}{2
    \Delta\xi_2}
    \left(\left(x_1\right)_{i, j+m}-\left(x_1\right)_{i, j-m}\right),
    \\
    & \left(J \frac{\partial \xi_2}{\partial x_1}\right)_{i, j}=-\sum_{m=1}^p \frac{\alpha_{p,m}}{2
    \Delta\xi_1}
    \left(\left(x_2\right)_{i+m, j}-\left(x_2\right)_{i-m, j}\right),
    \\
    & \left(J \frac{\partial \xi_2}{\partial x_2}\right)_{i, j}=+\sum_{m=1}^p \frac{\alpha_{p,m}}{2
    \Delta\xi_1}
    \left(\left(x_1\right)_{i+m, j}-\left(x_1\right)_{i-m, j}\right).
  \end{aligned}
\end{equation}
Here the coefficients $\alpha_{p,m}$ are introduced in \cite{Lefloch2002Fully} and satisfy
\begin{equation*}
  \sum_{m=1}^{p} m \alpha_{p,m} = 1,~\sum_{m=1}^{p} m^{2s-1}\alpha_{p,m} = 0,~s = 2,\dots,p,
\end{equation*}
and for $p=2,3$, they are
\begin{equation*}
  \begin{aligned}
    &\alpha_{2,1} = \dfrac{4}{3},~\alpha_{2,2} = -\dfrac{1}{6},\\
    &\alpha_{3,1} = \dfrac{3}{2},~\alpha_{3,2}=-\frac{3}{10},~\alpha_{3,3}=\frac{1}{30}.
  \end{aligned}
\end{equation*}
The choice of the mesh velocities $\dot{x}_1,~\dot{x}_2$ follows the strategy given in \cite{Duan2022High,Li2022High}.
The proof of the EC property is similar to that in \cite{Duan2022High}.
For completeness, it is presented in the following proposition.

\begin{proposition}\rm\label{prop:2D_HO_sufficient_condition}
  The schemes \eqref{eq:2D_HighOrder_EC_Discrete}-\eqref{eq:2D_SCL_HighOrder_2} with the numerical fluxes \eqref{eq:HighOrder_EC_Flux}-\eqref{eq:2D_SCL_flux_2} are EC, in the sense that,
  the numerical solutions satisfy the semi-discrete energy identities
  \begin{equation*}
    \frac{\mathrm{d}}{\mathrm{d} t}\mathcal{E}_{i, j}+\frac{1}{\Delta \xi_{1}}\left(\left({\widetilde{{\mathcal{Q}}}}_1\right)_{i+\frac{1}{2},j}^{\tt 2pth}-\left({\widetilde{{\mathcal{Q}}}}_1\right)_{i-\frac{1}{2},j}^{\tt 2pth}\right)+\frac{1}{\Delta \xi_{2}}\left(\left({\widetilde{{\mathcal{Q}}}}_2\right)_{i,j+\frac{1}{2}}^{\tt 2pth}-\left({\widetilde{{\mathcal{Q}}}}_2\right)_{i,j-\frac{1}{2}}^{\tt 2pth}\right) = 0,
  \end{equation*}
  with the numerical energy $\mathcal{E}_{i,j}=J_{i,j}\eta(\bU_{i,j})$,
  and the numerical energy fluxes
  \begin{align*}
    \left(\widetilde{\mathcal{Q}}_1\right)_{i+\frac{1}{2},j}^{\tt 2pth }=\sum_{m=1}^p \alpha_{p, m} \sum_{s=0}^{m-1} \widetilde{\mathcal{Q}}_1\left(\bm{U}_{i-s,j}, \bm{U}_{i-s+m,j},\left(J \frac{\partial \xi_1}{\partial \zeta}\right)_{i-s,j},\left(J \frac{\partial \xi_1}{\partial \zeta}\right)_{i-s+m,j}\right),
    \\
    \left(\widetilde{\mathcal{Q}}_2\right)_{i,j+\frac{1}{2}}^{\tt 2pth }=\sum_{m=1}^p \alpha_{p, m} \sum_{s=0}^{m-1} \widetilde{\mathcal{Q}}_2\left(\bm{U}_{i,j-s}, \bm{U}_{i,j-s+m},\left(J \frac{\partial \xi_2}{\partial \zeta}\right)_{i,j-s},\left(J \frac{\partial \xi_2}{\partial \zeta}\right)_{i,j-s+m}\right),
  \end{align*}
  where
  \begin{align*}
    \widetilde{\mathcal{Q}}_\ell =
    \ &\frac{1}{2}\left(\bm{V}(\bU_{L})+\bm{V}(\bU_{R})\right)^{\mathrm{T}}\widetilde{\bm{\mathcal{F}}}_{\ell}
    - \frac{1}{2}\left(\left(J \frac{\partial \xi_{\ell}}{\partial t}\right)_L+\left(J \frac{\partial \xi_{\ell}}{\partial t}\right)_R\right)\left(\phi\left(\bU_{L}\right)+\phi\left(\bU_{R}\right)\right) \nonumber \nonumber\\
    &- \sum_{k=1}^{2}\frac{1}{4}\left(\left(J \frac{\partial \xi_{\ell}}{\partial x_k}\right)_L+\left(J \frac{\partial \xi_{\ell}}{\partial x_k}\right)_R\right)\left(\psi_{k}\left(\bU_{L}\right)+\psi_{k}\left(\bU_{R}\right)\right) \nonumber\\
    & +\sum_{k=1}^{2} \frac{g}{8} \left(\left(J \frac{\partial \xi_\ell}{\partial x_k}\right)_L+\left(J \frac{\partial \xi_\ell}{\partial x_k}\right)_R\right)\left(\left(hv_k\right)_{L}+\left(hv_k\right)_{R}\right)\left(b_{L}+b_{R}\right).
  \end{align*}
\end{proposition}

\begin{proof}\rm
  Taking the dot product of the schemes \eqref{eq:2D_HighOrder_EC_Discrete} with $\bm{V}_{i,j}:=\bV(\bU_{i,j})$
  and using the discrete VCL \eqref{eq:2D_VCL_HighOrder},
  one has
  \begin{align*}
    \frac{\mathrm{d}}{\mathrm{d} t}\mathcal{E}_{i,j}
    =
    & -\dfrac{1}{\Delta \xi_1}\bm{V}_{i,j}^{\mathrm{T}}
    \left(
    \left(\bm{\widetilde{\mathcal{F}}}_{1}\right)_{i+\frac{1}{2}, j}^{\tt 2pth}-\left(\bm{\widetilde{\mathcal{F}}}_{1}\right)_{i-\frac{1}{2}, j}^{\tt 2pth}\right)
    -\frac{1}{\Delta \xi_2}\bm{V}_{i,j}^{\mathrm{T}}
    \left(\left(\bm{\widetilde{\mathcal{F}}}_{2}\right)_{i, j+\frac{1}{2}}^{\tt 2pth}
    -\left(\bm{\widetilde{\mathcal{F}}}_{2}\right)_{i, j-\frac{1}{2}}^{\tt 2pth}\right)\nonumber
    \\
    &- \frac{gh_{i,j}}{\Delta \xi_1}
    \bm{V}_{i,j}^{\mathrm{T}}
    \left(\left(\bm{\widetilde{\mathcal{B}}}_{1}\right)_{i+\frac{1}{2}, j}^{{\tt 2pth}}-\left(\bm{\widetilde{\mathcal{B}}}_{1}\right)_{i-\frac{1}{2}, j}^{{\tt 2pth}}\right)
    -\frac{gh_{i,j}}{\Delta \xi_2}\bm{V}_{i,j}^{\mathrm{T}}\left(\left(\bm{\widetilde{\mathcal{B}}}_{2}\right)_{i, j+\frac{1}{2}}^{{\tt 2pth}}
    -\left(\bm{\widetilde{\mathcal{B}}}_{2}\right)_{i, j-\frac{1}{2}}^{{\tt 2pth}}\right)
    \\
    &+  \phi_{i,j}\left[\frac{1}{\Delta \xi_1}
      \left(\left(\widetilde{J \frac{\partial \xi_1}{\partial t}}\right)_{i+\frac{1}{2}, j}^{\tt 2pth}
      -\left(\widetilde{J \frac{\partial \xi_1}{\partial t}}\right)_{i-\frac{1}{2}, j}^{\tt 2pth}\right)
      +\frac{1}{\Delta \xi_2}
    \left(\left(\widetilde{J \frac{\partial \xi_2}{\partial t}}\right)_{i, j+\frac{1}{2}}^{2 p \mathrm{th}}-\left(\widetilde{J \frac{\partial \xi_2}{\partial t}}\right)_{i, j-\frac{1}{2}}^{\tt 2pth}\right)\right].
  \end{align*}
  Further combining it with the discrete SCLs \eqref{eq:2D_SCL_HighOrder_1}-\eqref{eq:2D_SCL_HighOrder_2} gives
  \begin{align*}
    \frac{\mathrm{d}}{\mathrm{d} t}\mathcal{E}_{i,j}
    =
    & -\dfrac{1}{\Delta \xi_1}\bm{V}_{i,j}^{\mathrm{T}}
    \left(
    \left(\bm{\widetilde{\mathcal{F}}}_{1}\right)_{i+\frac{1}{2}, j}^{\tt 2pth}-\left(\bm{\widetilde{\mathcal{F}}}_{1}\right)_{i-\frac{1}{2}, j}^{\tt 2pth}\right)
    -\frac{1}{\Delta \xi_2}\bm{V}_{i,j}^{\mathrm{T}}
    \left(\left(\bm{\widetilde{\mathcal{F}}}_{2}\right)_{i, j+\frac{1}{2}}^{\tt 2pth}
    -\left(\bm{\widetilde{\mathcal{F}}}_{2}\right)_{i, j-\frac{1}{2}}^{\tt 2pth}\right)\nonumber
    \\
    &- \frac{gh_{i,j}}{\Delta \xi_1}
    \bm{V}_{i,j}^{\mathrm{T}}
    \left(\left(\bm{\widetilde{\mathcal{B}}}_{1}\right)_{i+\frac{1}{2}, j}^{{\tt 2pth}}-\left(\bm{\widetilde{\mathcal{B}}}_{1}\right)_{i-\frac{1}{2}, j}^{{\tt 2pth}}\right)
    -\frac{gh_{i,j}}{\Delta \xi_2}\bm{V}_{i,j}^{\mathrm{T}}\left(\left(\bm{\widetilde{\mathcal{B}}}_{2}\right)_{i, j+\frac{1}{2}}^{{\tt 2pth}}
    -\left(\bm{\widetilde{\mathcal{B}}}_{2}\right)_{i, j-\frac{1}{2}}^{{\tt 2pth}}\right)
    \\
    &+  \phi_{i,j}\left[\frac{1}{\Delta \xi_1}
      \left(\left(\widetilde{J \frac{\partial \xi_1}{\partial t}}\right)_{i+\frac{1}{2}, j}^{\tt 2pth}
      -\left(\widetilde{J \frac{\partial \xi_1}{\partial t}}\right)_{i-\frac{1}{2}, j}^{\tt 2pth}\right)
      +\frac{1}{\Delta \xi_2}
    \left(\left(\widetilde{J \frac{\partial \xi_2}{\partial t}}\right)_{i, j+\frac{1}{2}}^{2 p \mathrm{th}}-\left(\widetilde{J \frac{\partial \xi_2}{\partial t}}\right)_{i, j-\frac{1}{2}}^{\tt 2pth}\right)\right]\\
    &+\sum_{k=1}^{2}\left(\psi_{k}\right)_{i,j}\left[\frac{1}{\Delta \xi_1}\left(\left(\widetilde{J \frac{\partial \xi_1}{\partial x_k}}\right)_{i+\frac{1}{2},j}^{{\tt 2pth}}-\left(\widetilde{J \frac{\partial \xi_1}{\partial x_k}}\right)_{i-\frac{1}{2},j}^{{\tt 2pth}}\right) + \frac{1}{\Delta \xi_2}\left(\left(\widetilde{J \frac{\partial \xi_2}{\partial x_k}}\right)_{i, j+\frac{1}{2}}^{{\tt 2pth}}-\left(\widetilde{J \frac{\partial \xi_2}{\partial x_k}}\right)_{i,j-\frac{1}{2}}^{{\tt 2pth}}\right)\right].
  \end{align*}
  For ease of exposition, the right-hand side of the above equation can be split as
  \begin{equation*}
    \frac{\mathrm{d}}{\mathrm{d} t}\mathcal{E}_{i,j}
    = -\frac{1}{\Delta \xi_1}\sum_{m=1}^{p}\alpha_{p,m}\left(I_1+I_2\right)-\frac{1}{\Delta \xi_2}\sum_{m=1}^{p}\alpha_{p,m}\left(I_3+I_4\right),
  \end{equation*}
  where
  \begin{align*}
    &I_1 = \bm{V}_{i,j}^{\mathrm{T}}
    \left[\bm{\widetilde{\mathcal{F}}}_1
      \left(\bm{U}_{i,j},\bm{U}_{i+m,j},\left(J\frac{\partial \xi_1}{\partial \zeta}\right)_{i,j},\left(J\frac{\partial \xi_1}{\partial \zeta}\right)_{i+m,j}\right)
    -\bm{\widetilde{\mathcal{F}}}_1\left(\bm{U}_{i,j},\bm{U}_{i-m,j},\left(J\frac{\partial \xi_1}{\partial \zeta}\right)_{i,j},\left(J\frac{\partial \xi_1}{\partial \zeta}\right)_{i-m,j}\right)\right]\nonumber
    \\
    &\qquad-\phi_{i,j}\Bigg[\frac{1}{2}\left(\left(J\frac{\partial \xi_1}{\partial t}\right)_{i,j}
      +\left(J\frac{\partial \xi_1}{\partial t}\right)_{i+m,j}\right)
      - \frac{1}{2}\left(\left(J\frac{\partial \xi_1}{\partial t}\right)_{i,j}
    +\left(J\frac{\partial \xi_1}{\partial t}\right)_{i-m,j}\right) \Bigg]
    \\
    &\qquad-\sum_{k=1}^{2}\left(\psi_{k}\right)_{i,j}\left[\frac{1}{2}\left(\left(J\frac{\partial \xi_1}{\partial x_k}\right)_{i,j}+\left(J\frac{\partial \xi_1}{\partial x_k}\right)_{i+m,j}\right)-\frac{1}{2}\left(\left(J\frac{\partial \xi_1}{\partial x_k}\right)_{i,j}+\left(J\frac{\partial \xi_1}{\partial x_k}\right)_{i-m,j}\right)\right],
    \\
    &I_2 =\sum_{k=1}^2\frac{ g\left(hv_k\right)_{i,j}}{4}
    \left[\left(\left(J \frac{\partial \xi_1}{\partial x_k}\right)_{i,j}+\left(J \frac{\partial \xi_1}{\partial x_k}\right)_{i+m,j}\right)\left({b}_{i,j}+b_{i+m,j}\right)
    \right.\\
    & \left.\qquad
    -\left(\left(J \frac{\partial \xi_1}{\partial x_k}\right)_{i,j}+\left(J \frac{\partial \xi_1}{\partial x_k}\right)_{i-m,j}\right)\left({b}_{i,j}+b_{i-m,j}\right)\right],\\
    &I_3 =\bm{V}_{i,j}^{\mathrm{T}} \left[\bm{\widetilde{\mathcal{F}}}_2\left(\bm{U}_{i,j},\bm{U}_{i,j+m},\left(J\frac{\partial \xi_2}{\partial \zeta}\right)_{i,j},\left(J\frac{\partial \xi_2}{\partial \zeta}\right)_{i,j+m}\right)-\bm{\widetilde{\mathcal{F}}}_2\left(\bm{U}_{i,j},\bm{U}_{i,j-m},\left(J\frac{\partial \xi_2}{\partial \zeta}\right)_{i,j},\left(J\frac{\partial \xi_2}{\partial \zeta}\right)_{i,j-m}\right)\right]\\
    &\qquad-\phi_{i,j}\Bigg[\frac{1}{2}\left(\left(J\frac{\partial \xi_2}{\partial t}\right)_{i,j}
      +\left(J\frac{\partial \xi_2}{\partial t}\right)_{i,j+m}\right)
      - \frac{1}{2}\left(\left(J\frac{\partial \xi_2}{\partial t}\right)_{i,j}
    +\left(J\frac{\partial \xi_2}{\partial t}\right)_{i,j-m}\right) \Bigg]
    \\
    &\qquad-\sum_{k=1}^{2}\left(\psi_{k}\right)_{i,j}\left[\frac{1}{2}\left(\left(J\frac{\partial \xi_2}{\partial x_k}\right)_{i,j}+\left(J\frac{\partial \xi_2}{\partial x_k}\right)_{i,j+m}\right)
    -\frac{1}{2}\left(\left(J\frac{\partial \xi_2}{\partial x_k}\right)_{i,j}+\left(J\frac{\partial \xi_1}{\partial x_k}\right)_{i,j-m}\right)\right],
    \\
    &I_4 =\sum_{k=1}^2\frac{ g\left(hv_k\right)_{i,j}}{4}\left[\left(\left(J \frac{\partial \xi_2}{\partial x_k}\right)_{i,j}+\left(J \frac{\partial \xi_2}{\partial x_k}\right)_{i,j+m}\right)\left({b}_{i,j}+b_{i,j+m}\right)
         \right.\\
    & \left.\qquad
    -\left(\left(J \frac{\partial \xi_2}{\partial x_k}\right)_{i,j}+\left(J \frac{\partial \xi_2}{\partial x_k}\right)_{i,j-m}\right)\left({b}_{i,j}+b_{i,j-m}\right)\right],
  \end{align*}
  with $\zeta = t,x_1,x_2$.
  Using $a_{i,j}= \frac{1}{2}\left(a_{i,j}+a_{i+m,j}\right)-\frac{1}{2}\left(a_{i+m,j}-a_{i,j}\right)$ and $a_{i,j} = \frac{1}{2}\left(a_{i,j}+a_{i-m,j}\right)+\frac{1}{2}\left(a_{i,j}-a_{i-m,j}\right)$,
  the term $I_1$ can be simplified as follows
  \begin{align*}
    I_1 = &+\frac{1}{2}\left(\bm{V}_{i,j}+\bm{V}_{i+m,j}\right)^{\mathrm{T}}
    \bm{\widetilde{\mathcal{F}}}_1\left(\bm{U}_{i,j},\bm{U}_{i+m,j},\left(J\frac{\partial \xi_1}{\partial \zeta}\right)_{i,j},\left(J\frac{\partial \xi_1}{\partial \zeta}\right)_{i+m,j}\right)\\
    &-\frac{1}{2}\left(\bm{V}_{i+m,j}-\bm{V}_{i,j}\right)^{\mathrm{T}}
    \bm{\widetilde{\mathcal{F}}}_1\left(\bm{U}_{i,j},\bm{U}_{i+m,j},\left(J\frac{\partial \xi_1}{\partial \zeta}\right)_{i,j},\left(J\frac{\partial \xi_1}{\partial \zeta}\right)_{i+m,j}\right)\\
    &-\frac{1}{2}\left(\bm{V}_{i,j}+\bm{V}_{i-m,j}\right)^{\mathrm{T}}\bm{\widetilde{\mathcal{F}}}_1\left(\bm{U}_{i,j},\bm{U}_{i-m,j},\left(J\frac{\partial \xi_1}{\partial \zeta}\right)_{i,j},\left(J\frac{\partial \xi_1}{\partial \zeta}\right)_{i-m,j}\right)\\
    & -\frac{1}{2}\left(\bm{V}_{i,j}-\bm{V}_{i-m,j}\right)^{\mathrm{T}}\bm{\widetilde{\mathcal{F}}}_1\left(\bm{U}_{i,j},\bm{U}_{i-m,j},\left(J\frac{\partial \xi_1}{\partial \zeta}\right)_{i,j},\left(J\frac{\partial \xi_1}{\partial \zeta}\right)_{i-m,j}\right)\\
    &-\frac{1}{4}\left(\phi_{i,j}+\phi_{i+m,j}\right)\left(\left(J\frac{\partial \xi_1}{\partial t}\right)_{i,j}
    +\left(J\frac{\partial \xi_1}{\partial t}\right)_{i+m,j}\right)
    +\frac{1}{4}\left(\phi_{i+m,j}-\phi_{i,j}\right)
    \left(\left(J\frac{\partial \xi_1}{\partial t}\right)_{i,j}
    +\left(J\frac{\partial \xi_1}{\partial t}\right)_{i+m,j}\right)\\
    &+ \frac{1}{4}\left(\phi_{i,j}+\phi_{i-m,j}\right)
    \left(\left(J\frac{\partial \xi_1}{\partial t}\right)_{i,j}
    +\left(J\frac{\partial \xi_1}{\partial t}\right)_{i-m,j}\right)
    +\frac{1}{4}\left(\phi_{i,j}-\phi_{i-m,j}\right)
    \left(\left(J\frac{\partial \xi_1}{\partial t}\right)_{i,j}
    +\left(J\frac{\partial \xi_1}{\partial t}\right)_{i-m,j}\right)
    \\
    &-\sum_{k=1}^{2}\Bigg[\frac{1}{4}\left(\left(\psi_k\right)_{i,j}+\left(\psi_k\right)_{i+m,j}\right)\left(\left(J\frac{\partial \xi_1}{\partial x_1}\right)_{i,j}
      +\left(J\frac{\partial \xi_1}{\partial x_1}\right)_{i+m,j}\right)
      \\
      &-\frac{1}{4}\left(\left(\psi_k\right)_{i+m,j}-\left(\psi_k\right)_{i,j}\right)
      \left(\left(J\frac{\partial \xi_1}{\partial x_1}\right)_{i,j}
      +\left(J\frac{\partial \xi_1}{\partial x_1}\right)_{i+m,j}\right)
      \\
      &- \frac{1}{4}\left(\left(\psi_k\right)_{i,j}+\left(\psi_k\right)_{i-m,j}\right)
      \left(\left(J\frac{\partial \xi_1}{\partial x_1}\right)_{i,j}
      +\left(J\frac{\partial \xi_1}{\partial x_1}\right)_{i-m,j}\right)
      \\
      &-\frac{1}{4}\left(\left(\psi_k\right)_{i,j}-\left(\psi_k\right)_{i-m,j}\right)
      \left(\left(J\frac{\partial \xi_1}{\partial x_1}\right)_{i,j}
    +\left(J\frac{\partial \xi_1}{\partial x_1}\right)_{i-m,j}\right)\Bigg].
  \end{align*}
  Similarly, $I_2$ becomes
  \begin{align*}
    I_2 =&+\sum_{k=1}^2 g\Bigg[\frac{1}{8}\left(\left(hv_k\right)_{i+m,j}+\left(hv_k\right)_{i,j}\right)
      \left(\left(J \frac{\partial \xi_1}{\partial x_k}\right)_{i,j}
      +\left(J \frac{\partial \xi_1}{\partial x_k}\right)_{i+m,j}\right)\left({b}_{i,j}+b_{i+m,j}\right)\\
      &-\frac{1}{8}\left(\left(hv_k\right)_{i+m,j}-\left(hv_k\right)_{i,j}\right)
      \left(\left(J \frac{\partial \xi_1}{\partial x_k}\right)_{i,j}
      +\left(J \frac{\partial \xi_1}{\partial x_k}\right)_{i+m,j}\right)\left({b}_{i,j}+b_{i+m,j}\right)\\
      &-\frac{1}{8}\left(\left(hv_k\right)_{i-m,j}+\left(hv_k\right)_{i,j}\right)
      \left(\left(J \frac{\partial \xi_1}{\partial x_k}\right)_{i,j}
      +\left(J \frac{\partial \xi_1}{\partial x_k}\right)_{i-m,j}\right)\left({b}_{i,j}+b_{i-m,j}\right)\\
      &-\frac{1}{8}\left(\left(hv_k\right)_{i,j}-\left(hv_k\right)_{i-m,j}\right)
      \left(\left(J \frac{\partial \xi_1}{\partial x_k}\right)_{i,j}
      +\left(J \frac{\partial \xi_1}{\partial x_k}\right)_{i-m,j}\right)\left({b}_{i,j}+b_{i-m,j}\right)
    \Bigg].
  \end{align*}
  Using the condition \eqref{eq:EC_condition_1} yields
  \begin{equation*}
    I_1+I_2 = \widetilde{\mathcal{Q}}_1\left(\bm{U}_{i,j}, \bm{U}_{i+m,j},\left(J \frac{\partial \xi_1}{\partial \zeta}\right)_{i,j},\left(J \frac{\partial \xi_1}{\partial \zeta}\right)_{i+m,j}\right)-\widetilde{\mathcal{Q}}_1\left(\bm{U}_{i,j}, \bm{U}_{i-m,j},\left(J \frac{\partial \xi_1}{\partial \zeta}\right)_{i,j},\left(J \frac{\partial \xi_1}{\partial \zeta}\right)_{i-m,j}\right).
  \end{equation*}
  Simplifying the terms $I_3,I_4$ in the same way completes the proof.
\end{proof}

\begin{proposition}\rm
  The fully-discrete schemes based on the semi-discrete EC schemes \eqref{eq:2D_HighOrder_EC_Discrete}-\eqref{eq:2D_SCL_HighOrder_2} with the numerical fluxes \eqref{eq:HighOrder_EC_Flux}-\eqref{eq:2D_SCL_flux_2} under the forward Euler or explicit SSP RK discretizations are WB on the moving meshes, in the sense that,
  if the mesh does not interleave during the computation,
  and the water depth is always positive,
  then for the given initial data satisfying the lake at rest
  \begin{equation*}
    (h+b)_{i,j}^{0} \equiv C,
    ~(v_1)_{i,j}^{0} = (v_2)_{i,j}^{0} = 0,
    ~\forall i,j,
  \end{equation*}
  the numerical solutions satisfy
  \begin{equation}\label{eq:2pnd_WB_tn}
    (h+b)_{i,j}^{n} \equiv C,
    ~(v_1)_{i,j}^{n} = (v_2)_{i,j}^{n} = 0,
    ~\forall i,j,
  \end{equation}
  at $t^n$, where $C$ is a given constant.
\end{proposition}

\begin{proof}
  By induction, assuming that the conditions \eqref{eq:2pnd_WB_tn} are satisfied at $t^{n}$,
  it is enough to prove they still hold at $t^{n+1}$.
  Moreover, we only need to prove in the case of the forward Euler time discretization
  since the explicit SSP RK schemes can be rewritten as its convex combination,
  and the superscript $n$ will be dropped in the right-hand sides for simplicity in the proof.
  For the fully-discrete schemes based on the semi-discrete schemes \eqref{eq:2D_HighOrder_EC_Discrete} with the EC fluxes \eqref{eq:HighOrder_EC_Flux}-\eqref{eq:HighOrder_EC_Flux_B} and the forward Euler time discretization,
  the first and last components can be simplified as
  \begin{align}
    (Jh)_{i,j}^{n+1} - (Jh)_{i,j}^{n} = &
    -\frac{\Delta t}{4\Delta {\xi_1}}\sum_{m=1}^{p}\alpha_{p,m}
    \Bigg[\left(\left(J\frac{\partial \xi_1}{\partial t}\right)_{i,j}
      +\left(J\frac{\partial \xi_1}{\partial t}\right)_{i+m,j}\right)
      (h_{i+m,j}+h_{i,j})
      \nonumber\\
      &- \left(\left(J\frac{\partial \xi_1}{\partial t}\right)_{i,j}
      +\left(J\frac{\partial \xi_1}{\partial t}\right)_{i-m,j}\right)
    (h_{i-m,j}+h_{i,j}) \Bigg]\nonumber\\
    &
    -\frac{\Delta t}{4\Delta {\xi_2}}\sum_{m=1}^{p}\alpha_{p,m}
    \Bigg[\left(\left(J\frac{\partial \xi_2}{\partial t}\right)_{i,j}
      +\left(J\frac{\partial \xi_2}{\partial t}\right)_{i,j+m}\right)
      (h_{i,j+m}+h_{i,j})
      \nonumber\\
      &- \left(\left(J\frac{\partial \xi_2}{\partial t}\right)_{i,j}
      +\left(J\frac{\partial \xi_2}{\partial t}\right)_{i,j-m}\right)
    (h_{i,j-m}+h_{i,j}) \Bigg],\label{eq:2D_h_discre}\\
    (Jb)_{i,j}^{n+1} - (Jb)_{i,j}^{n} = &
    -\frac{\Delta t}{4\Delta {\xi_1}}\sum_{m=1}^{p}\alpha_{p,m}
    \Bigg[\left(\left(J\frac{\partial \xi_1}{\partial t}\right)_{i,j}
      +\left(J\frac{\partial \xi_1}{\partial t}\right)_{i+m,j}\right)
      (b_{i+m,j}+b_{i,j})
      \nonumber\\
      &- \left(\left(J\frac{\partial \xi_1}{\partial t}\right)_{i,j}
      +\left(J\frac{\partial \xi_1}{\partial t}\right)_{i-m,j}\right)
    (b_{i-m,j}+b_{i,j}) \Bigg]\nonumber\\
    &
    -\frac{\Delta t}{4\Delta {\xi_2}}\sum_{m=1}^{p}\alpha_{p,m}
    \Bigg[\left(\left(J\frac{\partial \xi_2}{\partial t}\right)_{i,j}
      +\left(J\frac{\partial \xi_2}{\partial t}\right)_{i,j+m}\right)
      (b_{i,j+m}+b_{i,j})
      \nonumber\\
      &- \left(\left(J\frac{\partial \xi_2}{\partial t}\right)_{i,j}
      +\left(J\frac{\partial \xi_2}{\partial t}\right)_{i,j-m}\right)
    (b_{i,j-m}+b_{i,j}) \Bigg].\label{eq:2D_b_discre}
  \end{align}
  The summation of \eqref{eq:2D_h_discre} and \eqref{eq:2D_b_discre} is
  \begin{align*}
    (J(h+b))_{i,j}^{n+1} - J_{i,j}^{n}C =&
    -C\frac{\Delta t}{2\Delta {\xi_1}}\sum_{m=1}^{p}\alpha_{p,m}
    \Bigg[\left(\left(J\frac{\partial \xi_1}{\partial t}\right)_{i,j}
      +\left(J\frac{\partial \xi_1}{\partial t}\right)_{i+m,j}\right)
      \nonumber
      \\
      &- \left(\left(J\frac{\partial \xi_1}{\partial t}\right)_{i,j}
    +\left(J\frac{\partial \xi_1}{\partial t}\right)_{i-m,j}\right) \Bigg]\nonumber
    \\
    &
    -C\frac{\Delta t}{2\Delta {\xi_2}}\sum_{m=1}^{p}\alpha_{p,m}
    \Bigg[\left(\left(J\frac{\partial \xi_2}{\partial t}\right)_{i,j}
      +\left(J\frac{\partial \xi_2}{\partial t}\right)_{i,j+m}\right)
      \nonumber
      \\
      &- \left(\left(J\frac{\partial \xi_2}{\partial t}\right)_{i,j}
    +\left(J\frac{\partial \xi_2}{\partial t}\right)_{i,j-m}\right) \Bigg].
  \end{align*}
  Combining it with the fully-discrete VCL under the forward Euler time discretization
  \begin{align*}
    J_{i,j}^{n+1} - J_{i,j}^{n} = &
    -\frac{\Delta t}{2\Delta {\xi_1}}\sum_{m=1}^{p}\alpha_{p,m}
    \Bigg[\left(\left(J\frac{\partial \xi_1}{\partial t}\right)_{i,j}
      +\left(J\frac{\partial \xi_1}{\partial t}\right)_{i+m,j}\right)
      - \left(\left(J\frac{\partial \xi_1}{\partial t}\right)_{i,j}
    +\left(J\frac{\partial \xi_1}{\partial t}\right)_{i-m,j}\right) \Bigg]\nonumber
    \\
    &
    -\frac{\Delta t}{2\Delta {\xi_2}}\sum_{m=1}^{p}\alpha_{p,m}
    \Bigg[\left(\left(J\frac{\partial \xi_2}{\partial t}\right)_{i,j}
      +\left(J\frac{\partial \xi_2}{\partial t}\right)_{i,j+m}\right)
      - \left(\left(J\frac{\partial \xi_2}{\partial t}\right)_{i,j}
    +\left(J\frac{\partial \xi_2}{\partial t}\right)_{i,j-m}\right) \Bigg]
  \end{align*}
  gives
  \begin{equation*}
    J_{i,j}^{n+1}\left((h+b)_{i,j}^{n+1} - C\right) = 0.
  \end{equation*}
  Under the assumption $J_{i,j}^{n+1} > 0$,
  one has $(h+b)_{i,j}^{n+1} \equiv C$.

  The $2$nd component of the fully-discrete schemes under the forward Euler time discretization can be written as
  \begin{align}\label{eq:JHU_original}
    (Jhv_1)^{n+1}_{i,j} - (Jhv_1)^n_{i,j} =
    -\frac{\Delta t}{\Delta \xi_1}\sum_{m=1}^{p}\alpha_{p,m}\left(H_1-H_2\right)
    -\frac{\Delta t}{\Delta \xi_2}\sum_{m=1}^{p}\alpha_{p,m}\left(H_3-H_4\right),
  \end{align}
  where
  \begin{align}
    H_1 = \ & \frac{g}{2}\left(\left(J\frac{\partial \xi_1}{\partial x_1}\right)_{i,j}+\left(J\frac{\partial \xi_1}{\partial x_1}\right)_{i+m,j}\right)\Bigg[\frac{1}{4}\left(h^2_{i,j}+h^2_{i+m,j}\right) + \frac{1}{2}\left((hb)_{i,j}+(hb)_{i+m,j}\right)\nonumber
      \\
    & - \frac{1}{4}\left({h}_{i,j}+{h}_{i+m,j}\right)\left({b}_{i,j}+{b}_{i+m,j}\right) + \frac{1}{2}h_{i,j} \left({b}_{i,j}+{b}_{i+m,j}\right) \Bigg]\nonumber,
    \\
    H_2 = \ & \frac{g}{2}\left(\left(J\frac{\partial \xi_1}{\partial x_1}\right)_{i,j}+\left(J\frac{\partial \xi_1}{\partial x_1}\right)_{i-m,j}\right)\Bigg[\frac{1}{4}\left(h^2_{i,j}+h^2_{i-m,j}\right)  + \frac{1}{2}\left((hb)_{i,j}+(hb)_{i-m,j}\right)\nonumber
      \\
    & - \frac{1}{4}\left({h}_{i,j}+{h}_{i-m,j}\right)\left({b}_{i,j}+{b}_{i-m,j}\right) + \frac{1}{2}h_{i,j} \left({b}_{i,j}+{b}_{i-m,j}\right) \Bigg]\nonumber,
    \\
    H_3 = \ & \frac{g}{2}\left(\left(J\frac{\partial \xi_2}{\partial x_1}\right)_{i,j}+\left(J\frac{\partial \xi_2}{\partial x_1}\right)_{i,j+m}\right)\Bigg[\frac{1}{4}\left(h^2_{i,j}+h^2_{i,j+m}\right)  + \frac{1}{2}\left((hb)_{i,j}+(hb)_{i,j+m}\right)\nonumber
      \\
    & - \frac{1}{4}\left({h}_{i,j}+{h}_{i,j+m}\right)\left({b}_{i,j}+{b}_{i,j+m}\right) + \frac{1}{2}h_{i,j} \left({b}_{i,j}+{b}_{i,j+m}\right) \Bigg]\nonumber,
    \\
    H_4 = \ & \frac{g}{2}\left(\left(J\frac{\partial \xi_2}{\partial x_1}\right)_{i,j}+\left(J\frac{\partial \xi_2}{\partial x_1}\right)_{i,j-m}\right)\Bigg[\frac{1}{4}\left(h^2_{i,j}+h^2_{i,j-m}\right) + \frac{1}{2}\left((hb)_{i,j}+(hb)_{i,j-m}\right)\nonumber \\
    & - \frac{1}{4}\left({h}_{i,j}+{h}_{i,j-m}\right)\left({b}_{i,j}+{b}_{i,j-m}\right) + \frac{1}{2}h_{i,j} \left({b}_{i,j}+{b}_{i,j-m}\right) \Bigg]\nonumber.
  \end{align}
  Based on the discrete SCLs \eqref{eq:2D_SCL_HighOrder_1}-\eqref{eq:2D_SCL_HighOrder_2}, the following two identities hold
  \begin{align*}
    &\frac{g}{2}h^2_{i,j}\Bigg\{\sum_{m=1}^{p}\alpha_{p,m}\frac{1}{\Delta \xi_1}\left[\frac{1}{2}\left(\left(J\frac{\partial \xi_1}{\partial x_1}\right)_{i,j}+\left(J\frac{\partial \xi_1}{\partial x_1}\right)_{i+m,j}\right)-\frac{1}{2}\left(\left(J\frac{\partial \xi_1}{\partial x_1}\right)_{i,j}+\left(J\frac{\partial \xi_1}{\partial x_1}\right)_{i-m,j}\right)\right]\\
    &+\sum_{m=1}^{p}\alpha_{p,m}\frac{1}{\Delta \xi_2}\left[\frac{1}{2}\left(\left(J\frac{\partial \xi_2}{\partial x_1}\right)_{i,j}+\left(J\frac{\partial \xi_2}{\partial x_1}\right)_{i,j+m}\right)-\frac{1}{2}\left(\left(J\frac{\partial \xi_2}{\partial x_1}\right)_{i,j}+\left(J\frac{\partial \xi_2}{\partial x_1}\right)_{i,j-m}\right)\right]\Bigg\} = 0,\\
    &{g}(hb)_{i,j}\Bigg\{\sum_{m=1}^{p}\alpha_{p,m}\frac{1}{\Delta \xi_1}\left[\frac{1}{2}\left(\left(J\frac{\partial \xi_1}{\partial x_1}\right)_{i,j}+\left(J\frac{\partial \xi_1}{\partial x_1}\right)_{i+m,j}\right)-\frac{1}{2}\left(\left(J\frac{\partial \xi_1}{\partial x_1}\right)_{i,j}+\left(J\frac{\partial \xi_1}{\partial x_1}\right)_{i-m,j}\right)\right]
      \\
    &+\sum_{m=1}^{p}\alpha_{p,m}\frac{1}{\Delta \xi_2}\left[\frac{1}{2}\left(\left(J\frac{\partial \xi_2}{\partial x_1}\right)_{i,j}+\left(J\frac{\partial \xi_2}{\partial x_1}\right)_{i,j+m}\right)-\frac{1}{2}\left(\left(J\frac{\partial \xi_2}{\partial x_1}\right)_{i,j}+\left(J\frac{\partial \xi_2}{\partial x_1}\right)_{i,j-m}\right)\right]\Bigg\} = 0.
  \end{align*}
  One can combine the above two identities to rewrite \eqref{eq:JHU_original} as
  \begin{align*}
    (Jhv_1)^{n+1}_{i,j} - (Jhv_1)^n_{i,j} =
    -\frac{\Delta t}{\Delta \xi_1}\sum_{m=1}^{p}\alpha_{p,m}\left(\widetilde{H}_1-\widetilde{H}_2\right)
    -\frac{\Delta t}{\Delta \xi_2}\sum_{m=1}^{p}\alpha_{p,m}\left(\widetilde{H}_3-\widetilde{H}_4\right),
  \end{align*}
  where
  \begin{align*}
    \widetilde{H}_1 = \ &\frac{g}{2}\left(\left(J\frac{\partial \xi_1}{\partial x_1}\right)_{i,j}+\left(J\frac{\partial \xi_1}{\partial x_1}\right)_{i+m,j}\right)\Bigg[\frac{1}{4}\left(h^2_{i,j}+h^2_{i+m,j}\right)  - \frac{1}{2}h^2_{i,j}+ \frac{1}{2}\left((hb)_{i,j}+(hb)_{i+m,j}\right)
      \\
    & - \frac{1}{4}\left({h}_{i,j}+{h}_{i+m,j}\right)\left({b}_{i,j}+{b}_{i+m,j}\right) + \frac{1}{2}h_{i,j} \left({b}_{i,j}+{b}_{i+m,j}\right) - (hb)_{i,j}\Bigg],
    \\
    \widetilde{H}_2 = \ &\frac{g}{2}\left(\left(J\frac{\partial \xi_1}{\partial x_1}\right)_{i,j}+\left(J\frac{\partial \xi_1}{\partial x_1}\right)_{i-m,j}\right)\Bigg[\frac{1}{4}\left(h^2_{i,j}+h^2_{i-m,j}\right)  - \frac{1}{2}h^2_{i,j}+ \frac{1}{2}\left((hb)_{i,j}+(hb)_{i-m,j}\right)
      \\
    & - \frac{1}{4}\left({h}_{i,j}+{h}_{i-m,j}\right)\left({b}_{i,j}+{b}_{i-m,j}\right) + \frac{1}{2}h_{i,j} \left({b}_{i,j}+{b}_{i-m,j}\right) - (hb)_{i,j}\Bigg],
    \\
    \widetilde{H}_3 = \ &\frac{g}{2}\left(\left(J\frac{\partial \xi_2}{\partial x_1}\right)_{i,j}+\left(J\frac{\partial \xi_2}{\partial x_1}\right)_{i,j+m}\right)\Bigg[\frac{1}{4}\left(h^2_{i,j}+h^2_{i,j+m}\right)  - \frac{1}{2}h^2_{i,j}+ \frac{1}{2}\left((hb)_{i,j}+(hb)_{i,j+m}\right)
      \\
    & - \frac{1}{4}\left({h}_{i,j}+{h}_{i,j+m}\right)\left({b}_{i,j}+{b}_{i,j+m}\right) + \frac{1}{2}h_{i,j} \left({b}_{i,j}+{b}_{i,j+m}\right) - (hb)_{i,j}\Bigg],
    \\
    \widetilde{H}_4 = \ &\frac{g}{2}\left(\left(J\frac{\partial \xi_2}{\partial x_1}\right)_{i,j}+\left(J\frac{\partial \xi_2}{\partial x_1}\right)_{i,j-m}\right)\Bigg[\frac{1}{4}\left(h^2_{i,j}+h^2_{i,j-m}\right)  - \frac{1}{2}h^2_{i,j}+ \frac{1}{2}\left((hb)_{i,j}+(hb)_{i,j-m}\right)
      \\
    & - \frac{1}{4}\left({h}_{i,j}+{h}_{i,j-m}\right)\left({b}_{i,j}+{b}_{i,j-m}\right) + \frac{1}{2}h_{i,j} \left({b}_{i,j}+{b}_{i,j-m}\right) - (hb)_{i,j}\Bigg].
  \end{align*}
  The two parts in the bracket in $\widetilde{H}_1$ can be simplified as
  \begin{align}
    &\frac{1}{4}\left({h}_{i,j}^2+{h}_{i+m,j}^2\right)  - \frac{1}{2}h^2_{i,j} = \frac{1}{4}(h_{i,j}+h_{i+m,j})(h_{i+m,j}-h_{i,j}),\label{eq:h2_express}
    \\
    &\frac{1}{2}\left((hb)_{i,j}+(hb)_{i+m,j}\right)- \frac{1}{4}\left({h}_{i,j}+{h}_{i+m,j}\right)\left({b}_{i,j}+{b}_{i+m,j}\right) +  \frac{1}{2}h_{i,j}\left({b}_{i,j}+{b}_{i+m,j}\right) - (hb)_{i,j}\nonumber
    \\
    & = \frac{1}{4}\left(h_{i,j}+h_{i+m,j}\right)\left(b_{i+m,j}-b_{i,j}\right).\label{eq:h_b_express}
  \end{align}
  Substituting \eqref{eq:h2_express}-\eqref{eq:h_b_express} into $\widetilde{H}_1$ gives
  \begin{align*}
    \widetilde{H}_1= \frac{g}{8}\left(\left(J\frac{\partial \xi_1}{\partial x_1}\right)_{i,j}+\left(J\frac{\partial \xi_1}{\partial x_1}\right)_{i+m,j}\right)\Bigg\{(h_{i+m,j}+h_{i,j})\left[(h+b)_{i+m,j} - (h+b)_{i,j}\right]\Bigg\}=0.
  \end{align*}
  Similarly, $\widetilde{H}_2,~\widetilde{H}_3,~\widetilde{H}_4$ can be simplified as
  \begin{align*}
    \widetilde{H}_2=\ & \frac{g}{8}\left(\left(J\frac{\partial \xi_1}{\partial x_1}\right)_{i,j}+\left(J\frac{\partial \xi_1}{\partial x_1}\right)_{i-m,j}\right)\Bigg\{(h_{i-m,j}+h_{i,j})\left[(h+b)_{i-m,j} - (h+b)_{i,j}\right]\Bigg\}=0,\\
    \widetilde{H}_3=\ & \frac{g}{8}\left(\left(J\frac{\partial \xi_2}{\partial x_1}\right)_{i,j}+\left(J\frac{\partial \xi_2}{\partial x_1}\right)_{i,j+m}\right)\Bigg\{(h_{i,j+m}+h_{i,j})\left[(h+b)_{i,j+m} - (h+b)_{i,j}\right]\Bigg\}=0,
    \\
    \widetilde{H}_4=\ & \frac{g}{8}\left(\left(J\frac{\partial \xi_2}{\partial x_1}\right)_{i,j}+\left(J\frac{\partial \xi_2}{\partial x_1}\right)_{i,j-m}\right)\Bigg\{(h_{i,j-m}+h_{i,j})\left[(h+b)_{i,j-m} - (h+b)_{i,j}\right]\Bigg\}=0,
  \end{align*}
  so that $(Jhv_1)^{n+1}_{i,j} - (Jhv_1)^n_{i,j} = 0$.
 One can further conclude that $\left(v_1\right)_{i,j}^{n+1} = 0$,
  and $(v_2)^{n+1}_{i,j}=0$ can be proved in the same way.
\end{proof}

\subsection{High-order WB ES schemes}\label{section:ES}
Since the energy identity is available only if the solution is smooth,
the EC schemes may produce unphysical oscillations near discontinuities,
which leads to the development of ES schemes,
obtained by adding suitable dissipation terms to the EC schemes to suppress oscillations.
Meanwhile, the choice of the dissipation terms is required to make the schemes maintain the ES and WB properties simultaneously.

Considering the dissipation terms in the ES schemes for the shallow water magnetohydrodynamics \cite{Duan2021SWMHD} with zero magnetic fields,
they are based on the jump of $\widehat{\bV} = (g(h+b)-(v_1^2+v_2^2)/2, v_1, v_2)$,
which vanishes for the lake at rest.
So the ES schemes are WB as no dissipation is added in this case.
It should be noticed that the last component of the reformulated SWEs \eqref{eq:SWE1} is a transport equation of the bottom topography
\begin{equation*}
    \pd{}{t}\left(Jb\right) + \pd{}{\xi_1}\left(J\pd{\xi_1}{t}b\right) + \pd{}{\xi_2}\left(J\pd{\xi_2}{t}b\right) = 0.
\end{equation*}
If no dissipation is added for the lake at rest,
oscillations may appear near the discontinuities in $b$
as high-order EC schemes are used to solve that equation,
which is also observed in the numerical tests,
see Figures \ref{1D_Well_balance_Compare} and \ref{1D_ES_Pertubation_WB_bottom}.

Therefore, this paper proposes to add additional dissipation terms to stabilize the schemes for discontinuous $b$ and maintain the ES and WB properties at the same time.
Take the $\xi_1$-direction as example,
\begin{align}\label{eq:HO_ES_flux}
  &\left(\bm{\widetilde{\mathcal{F}}}_1\right)_{i+\frac{1}{2},j}^{\tt ES}
  = \left(\bm{\widetilde{\mathcal{F}}}_{1}\right)_{i+\frac{1}{2}, j}^{\tt 2pth}
  - \begin{pmatrix}
    \widehat{\bm{D}}_{i+\frac12,j} \\ 0 \\
  \end{pmatrix}
  - \mathring{\bm{D}}_{i+\frac12,j},\\
  &\widehat{\bm{D}}_{i+\frac12,j} = \frac{1}{2}\widehat{\bm{S}}_{i+\frac{1}{2},j}\widehat{\bm{Y}}_{{i+\frac{1}{2},j} }\jumpangle{\bm{\widetilde{V}}}_{i+\frac{1}{2},j}^{\tt WENO},\label{eq:diss1}\\
  &\mathring{\bm{D}}_{i+\frac12,j} = \frac{1}{2}\mathring{\bm{S}}_{i+\frac{1}{2},j}\mathring{\bm{Y}}_{{i+\frac{1}{2},j} }\jumpangle{\bm{U}}_{i+\frac{1}{2},j}^{\tt WENO}.\label{eq:diss2}
\end{align}
To preserve the lake at rest,
i.e., $h+b\equiv C$ and $v_\ell\equiv 0$ are not influenced by the dissipation terms,
the summation of the first and last components,
and the other two components in the dissipation terms should be zero.
The first three components $\widehat{\bm{D}}_{i+\frac{1}{2},j}\in\bbR^3$ in the first dissipation term are designed in the same manner of high-order dissipation terms in curvilinear coordinates as \cite{Duan2022High},
which are based on the jump of reconstructed values of $\widetilde{\bV}$, scaled version of $\widehat{\bV}$ to be defined later,
so that the first dissipation term preserves the WB property similar to \cite{Duan2021SWMHD}.
To be specific,
the first three components depend on the jumps of the reconstructed values of $h+b$ and $v_{\ell}$,
which vanish for the lake at rest.
The second dissipation term $\mathring{\bm{D}}_{i+\frac{1}{2},j}\in\bbR^4$ is proposed to suppress possible oscillations due to the transportation of the discontinuous bottom topography on moving meshes,
which is based on the jump of reconstructed values of $\bU$.
Note the first and last components of $\mathring{\bm{D}}_{i+\frac{1}{2},j}$ depend on the jump of reconstructed values of $h$ and $b$, respectively,
so that to make the summation of those two components vanish for the lake at rest,
they should be added at the same time,
to be detailed in the discussions below.

In \eqref{eq:diss1}, $\widehat{\bm{S}}_{i+\frac{1}{2},j}$ is defined as
\begin{equation}\label{eq:diss1_S1}
  \widehat{\bm{S}}_{i+\frac{1}{2},j} = \left(\alpha_1\right)_{i+\frac12,j}\left(\bm{T}^{-1} \bm{R}(\bm{T} \widehat{\bm{U}})\right)_{i+\frac{1}{2},j},
\end{equation}
with the rotational matrix
\begin{equation*}
  \bm{T}=\begin{pmatrix}
    1 & 0 & 0  \\
    0 &  \cos \varphi &  \sin \varphi  \\
    0 & -\sin \varphi & \cos \varphi \\
  \end{pmatrix},~
  \varphi=\arctan \left(\left(J \frac{\partial \xi_1}{\partial x_2}\right) \Big/\left(J \frac{\partial \xi_1}{\partial x_1}\right)\right).
\end{equation*}
The matrix $\bm{R}$ satisfies the following decomposition based on the convexity of $\widehat{\eta}$ with respect to $\widehat{\bU}$,
\begin{equation*}
  \pd{\widehat{\bm{U}}}{\bm{\widehat{V}}}=\bm{R} \bm{R}^{\mathrm{T}},~ \pd{\bm{\widehat{F}}_{1}}{\bm{\widehat{U}}}=\bm{R} \bm{\Lambda} \bm{R}^{-1},~ \bm{\Lambda}=\operatorname{diag}\left\{\lambda_{1}, \ldots, \lambda_{3}\right\},
\end{equation*}
where $\bm{\widehat{U}}$, $\bm{\widehat{V}}$, and $\bm{\widehat{F}}_{1}$ are the conservative variables, entropy variables, and physical flux of the original SWEs in Remark \ref{rmk:EntropyForOriSWEs}, respectively.
Indeed, the columns of $\bm{R}$ are the scaling eigenvectors of ${\partial\bm{\widehat{F}}_{1}}/{\partial\bm{\widehat{U}}}$,
with $\bm{\Lambda}$ the diagonal matrix consisting of the corresponding eigenvalues.
The specific expressions are
\begin{equation*}
  \bm{R} = \begin{pmatrix}
    1 &1& 0 \\
    v_1+c &v_1-c &0\\
    v_2  & v_2& 1\\
  \end{pmatrix}
  \begin{pmatrix}
    \frac{1}{\sqrt{2g}} & 0 & 0\\
    0&\frac{1}{\sqrt{2g}}&0\\
    0& 0&\sqrt{h}\\
  \end{pmatrix},~
  \lambda_1=v_1+c,~\lambda_2=v_1-c,
  ~\lambda_3=v_1,
\end{equation*}
with $c = \sqrt{gh}$.
The maximal wave speed $\left(\alpha_1\right)_{i+\frac12,j}$ in \eqref{eq:diss1_S1} is evaluated at  $\left(\left(\xi_1\right)_{i+\frac12},\left(\xi_2\right)_{j}\right)$ based on the spectral radius
\begin{equation}\label{eq:eigens}
  \left(\alpha_1\right)_{i+\frac12,j}:=\left(\varrho_1\right)_{i+ \frac12,j},~
  \varrho_1 = \max\limits_{w}\left|J\frac{\partial \xi_1}{\partial t}+\widetilde{L}_1{\lambda}_{w}(\bm{T} \widehat{\bm{U}})\right|,~w=1,2,3,
\end{equation}
with
$\widetilde{L}_1=\sqrt{\sum\limits_{k=1}^{2}\left(J \pd{\xi_1}{x_{k}}\right)^{2}}$.
To construct high-order accurate ES fluxes, the high-order WENO reconstruction \cite{Borges2008An} is employed in the scaled entropy variables
$\left\{\widetilde{\bm{V}}_{i+r,j}:=\left(\bm{R}^{\mathrm{T}}(\bT\widehat{\bU})\bT\right)_{i+\frac12,j}\bm{\widehat{V}}_{i+r,j}\right\}$
to obtain the left and right limit values,
denoted by $\widetilde{\bm{V}}_{i+\frac{1}{2},j}^{{\tt WENO},-}$ (using stencil $r=-p+1,\cdots,p-1$) and $\widetilde{\bm{V}}_{i+\frac{1}{2},j}^{{\tt WENO},+}$ (using stencil $r=-p+2,\cdots,p$),
then the high-order jump of the scaled entropy variables can be defined as
\begin{equation*}
  \jumpangle{\widetilde{\bm{V}}}_{i+\frac{1}{2},j}^{{\tt WENO}} = \widetilde{\bm{V}}_{i+\frac{1}{2},j}^{{\tt WENO},+} - \widetilde{\bm{V}}_{i+\frac{1}{2},j}^{{\tt WENO},-}.
\end{equation*}
The diagonal matrix $\widehat{\bm{Y}}_{i+\frac{1}{2},j}$ is used to preserve the ``sign'' property \cite{Biswas2018Low},
with the diagonal elements
\begin{equation*}
  \left(\widehat{\bm{Y}}_{i+\frac{1}{2},j}\right)_{\iota,\iota} = \begin{cases}
    1,\quad
    &\text{if}\quad \text{sign}\left(\jumpangle{\widetilde{\bm{V}}_\iota}^{{\tt WENO}}_{{{i+\frac{1}{2},j} }}\right) ~\text{sign}\left(\jump{\bm{\widetilde{V}}_\iota}_{{{i+\frac{1}{2},j} }}\right)\geqslant 0, \\
    0, &\text{otherwise}, \\
  \end{cases}
\end{equation*}
where $\iota = 1,2,3$.

In the second dissipation term \eqref{eq:diss2},
$\mathring{\bm{S}}_{i+\frac{1}{2},j}$ and $\mathring{\bm{Y}}_{i+\frac{1}{2},j}$ are both diagonal matrices
\begin{align*}
  &\mathring{\bm{S}}_{i+\frac{1}{2},j} = \left|J\frac{\partial \xi_1}{\partial t}\right|_{i+\frac{1}{2},j}\bm{I}_{4},\\
  &(\mathring{\bm{Y}}_{i+\frac{1}{2},j})_{\nu,\nu} =
  \begin{cases}
    1,~ &\text{if}~ \nu =1,4, ~        \text{sign}\left(\jumpangle{\bm{U}_1}_{i+\frac{1}{2},j}^{{\tt WENO}}\right)\text{sign}\left(\jump{\bm{V}_1}_{i+\frac{1}{2},j} \right)\geqslant 0,
    \text{sign}\left(\jumpangle{\bm{U}_4}_{i+\frac{1}{2},j}^{{\tt WENO}}\right)\text{sign}\left(\jump{\bm{V}_4}_{i+\frac{1}{2},j} \right)\geqslant 0,\\
    1,~ &\text{if}~ \nu =2,3, ~       \text{sign}\left(\jumpangle{\bm{U}_\nu}_{i+\frac{1}{2},j}^{{\tt WENO}}\right)\text{sign}\left(\jump{\bm{V}_\nu}_{i+\frac{1}{2},j} \right)\geqslant 0,\\
    0,&\text{otherwise}.
  \end{cases}
\end{align*}
As mentioned before,
the summation of the first and last components in the second dissipation term corresponding to $h$ and $b$ must be zero for the lake at rest.
In this way, the coefficients in the switch function $(\mathring{\bm{Y}}_{i+\frac{1}{2},j})_{1,1}$ and $(\mathring{\bm{Y}}_{i+\frac{1}{2},j})_{4,4}$ should be the same,
thus one can see that only when the jumps of the reconstructed first and last components both satisfy the sign property,
the corresponding components for $h$ and $b$ are added at the same time.
Meanwhile, the nonlinear weights in the WENO reconstruction need to be the same for $h$ and $b$, similar to the approach in \cite{Xing2005High}.
For example, the left limit value of bottom topography $b$ at $\left(\left(\xi_1\right)_{i+\frac12},\left(\xi_2\right)_{j}\right)$ is first reconstructed as
$b_{i+\frac{1}{2},j}^{{\tt WENO},-} = \sum\limits_{r = -p+1}^{p-1} \beta_{r} b_{i+r,j}$,
then the reconstruction of $h$ is given by
$h_{i+\frac{1}{2},j}^{{\tt WENO},-} = \sum\limits_{r = -p+1}^{p-1} \beta_{r} h_{i+r,j}$.
It can be verified that the second dissipation term maintains the WB property by construction.
The high-order ES fluxes in the $\xi_2$-direction $\left(\bm{\widetilde{\mathcal{F}}}_2\right)_{i,j+\frac{1}{2}}^{\tt ES}$ are established similarly,
thus the overall schemes are WB.

In summary, the high-order WB ES schemes are obtained by replacing the high-order EC fluxes in \eqref{eq:2D_HighOrder_EC_Discrete} with the high-order ES fluxes \eqref{eq:HO_ES_flux}, i.e.,
\begin{align}\label{eq:2D_HighOrder_Discrete_ES}
    \frac{\mathrm{d}}{\mathrm{d} t} \bm{\mathcal{U}}_{i, j}=&-\frac{1}{\Delta \xi_1}\left(\left(\bm{\widetilde{\mathcal{F}}}_{1}\right)_{i+\frac{1}{2}, j}^{\tt ES}-\left(\bm{\widetilde{\mathcal{F}}}_{1}\right)_{i-\frac{1}{2}, j}^{\tt ES}\right)-\frac{1}{\Delta \xi_2}\left(\left(\bm{\widetilde{\mathcal{F}}}_{2}\right)_{i, j+\frac{1}{2}}^{\tt ES}-\left(\bm{\widetilde{\mathcal{F}}}_{2}\right)_{i, j-\frac{1}{2}}^{\tt ES}\right)\nonumber\\
    &- \frac{gh_{i,j}}{\Delta \xi_1}\left(\left(\bm{\widetilde{\mathcal{B}}}_{1}\right)_{i+\frac{1}{2}, j}^{{\tt 2pth}}-\left(\bm{\widetilde{\mathcal{B}}}_{1}\right)_{i-\frac{1}{2}, j}^{{\tt 2pth}}\right)
    - \frac{gh_{i,j}}{\Delta \xi_2}\left(\left(\bm{\widetilde{\mathcal{B}}}_{2}\right)_{i, j+\frac{1}{2}}^{{\tt 2pth}}-\left(\bm{\widetilde{\mathcal{B}}}_{2}\right)_{i, j-\frac{1}{2}}^{{\tt 2pth}}\right),
\end{align}
and the semi-discrete VCL \eqref{eq:2D_VCL_HighOrder}
and discrete SCLs \eqref{eq:2D_SCL_HighOrder_1}-\eqref{eq:2D_SCL_HighOrder_2} remain the same.
The following proposition proves the ES property,
which adapts the proof in \cite{Duan2022High} to accommodate the second dissipation term \eqref{eq:diss2}.
\begin{proposition}\rm
    The high-order schemes \eqref{eq:2D_HighOrder_Discrete_ES} and \eqref{eq:2D_VCL_HighOrder}-\eqref{eq:2D_SCL_HighOrder_2} are ES in the sense that,
  the numerical solutions satisfy the semi-discrete energy inequality
  \begin{equation}\label{eq:numEntropyInq}
    \frac{\mathrm{d}}{\mathrm{d} t}{\mathcal{E}}_{i, j}+\frac{1}{\Delta \xi_{1}}\left(\left({\widetilde{{\mathcal{Q}}}}_1\right)_{i+\frac{1}{2},j}^{\tt ES}-\left({\widetilde{{\mathcal{Q}}}}_1\right)_{i-\frac{1}{2},j}^{\tt ES}\right)+\frac{1}{\Delta \xi_{2}}\left(\left({\widetilde{{\mathcal{Q}}}}_2\right)_{i,j+\frac{1}{2}}^{\tt ES}-\left({\widetilde{{\mathcal{Q}}}}_2\right)_{i,j-\frac{1}{2}}^{\tt ES}\right)\leqslant 0,
  \end{equation}
  with the numerical energy fluxes
  \begin{align*}
    &\left({\widetilde{{\mathcal{Q}}}}_1\right)_{i+\frac{1}{2},j}^{\tt ES} = \left({\widetilde{{\mathcal{Q}}}}_1\right)_{i+\frac{1}{2},j}^{{\tt 2pth}}
    -\frac{1}{2}\left[\alpha_1\mean{\bm{\widetilde{V}}}^{\mathrm{T}}\widehat{\bm{Y}}\jumpangle{\bm{\widetilde{V}}}^{\tt WENO}\right]_{i+\frac{1}{2},j}
    -\frac{1}{2}\left[\left|J\frac{\partial \xi_1}{\partial t}\right|\mean{\bm{V}}^{\mathrm{T}}\mathring{\bm{Y}}\jumpangle{\bm{U}}^{\tt WENO}\right]_{i+\frac{1}{2},j},\nonumber\\
    &\left({\widetilde{{\mathcal{Q}}}}_2\right)_{i,j+\frac{1}{2}}^{\tt ES} = \left({\widetilde{{\mathcal{Q}}}}_2\right)_{i,j+\frac{1}{2}}^{\tt 2pth}
    -\frac{1}{2}\left[\alpha_2\mean{\bm{\widetilde{V}}}^{\mathrm{T}}\widehat{\bm{Y}}\jumpangle{\bm{\widetilde{V}}}^{\tt WENO}\right]_{i,j+\frac{1}{2}}-\frac{1}{2}\left[\left|J\frac{\partial \xi_2}{\partial t}\right|\mean{\bm{V}}^{\mathrm{T}}\mathring{\bm{Y}}\jumpangle{\bm{U}}^{\tt WENO}\right]_{i,j+\frac{1}{2}}.\nonumber
  \end{align*}
\end{proposition}

\begin{proof}\rm
  Take the inner product of the schemes \eqref{eq:2D_HighOrder_Discrete_ES} with $\bm{V}_{i,j}$,
  use the semi-discrete VCL \eqref{eq:2D_VCL_HighOrder} and discrete SCLs \eqref{eq:2D_SCL_HighOrder_1}-\eqref{eq:2D_SCL_HighOrder_2},
  then one has
  \begin{align}
    \frac{\mathrm{d}}{\mathrm{d} t} {\mathcal{E}}_{i, j}=&-\frac{1}{\Delta \xi_{1}}\left(\left({\widetilde{{\mathcal{Q}}}}_1\right)_{i+\frac{1}{2},j}^{\tt 2pth}-\left({\widetilde{{\mathcal{Q}}}}_1\right)_{i-\frac{1}{2},j}^{\tt 2pth}\right)-\frac{1}{\Delta \xi_{2}}\left(\left({\widetilde{{\mathcal{Q}}}}_2\right)_{i,j+\frac{1}{2}}^{\tt 2pth}-\left({\widetilde{{\mathcal{Q}}}}_2\right)_{i,j-\frac{1}{2}}^{\tt 2pth}\right)\nonumber\\
    &+\frac{1}{2\Delta \xi_{1}}\left(\alpha_1\right)_{i+\frac{1}{2},j}\widehat{\bm{V}}_{i,j}^{\mathrm{T}}\left(\bm{T}^{-1} \bm{R}(\bm{T}\widehat{\bm{U}})\right)_{i+\frac{1}{2},j}\widehat{\bm{Y}}_{i+\frac{1}{2},j}\jumpangle{\bm{\widetilde{V}}}_{i+\frac{1}{2},j}^{\tt WENO}\nonumber\\
    &-\frac{1}{2\Delta \xi_{1}}\left(\alpha_1\right)_{i-\frac{1}{2},j}\widehat{\bm{V}}_{i,j}^{\mathrm{T}}\left(\bm{T}^{-1} \bm{R}(\bm{T}\widehat{\bm{U}})\right)_{i-\frac{1}{2},j}\widehat{\bm{Y}}_{i-\frac{1}{2},j}\jumpangle{\bm{\widetilde{V}}}_{i-\frac{1}{2},j}^{\tt WENO}\nonumber\\
    &+\frac{1}{2\Delta \xi_{2}}\left(\alpha_2\right)_{i,j+\frac{1}{2}}\widehat{\bm{V}}_{i,j}^{\mathrm{T}}\left(\bm{T}^{-1} \bm{R}(\bm{T}\widehat{\bm{U}})\right)_{i,j+\frac{1}{2}}\widehat{\bm{Y}}_{i,j+\frac{1}{2}}\jumpangle{\bm{\widetilde{V}}}_{i,j+\frac{1}{2}}^{\tt WENO}\nonumber\\
    &-\frac{1}{2\Delta \xi_{2}}\left(\alpha_2\right)_{i,j-\frac{1}{2}}\widehat{\bm{V}}_{i,j}^{\mathrm{T}}\left(\bm{T}^{-1} \bm{R}(\bm{T}\widehat{\bm{U}})\right)_{i,j-\frac{1}{2}}\widehat{\bm{Y}}_{i,j-\frac{1}{2}}\jumpangle{\bm{\widetilde{V}}}_{i,j-\frac{1}{2}}^{\tt WENO}\nonumber\\
    &+\frac{1}{2\Delta \xi_{1}}\left|J\frac{\partial \xi_1}{\partial t}\right|_{i+\frac{1}{2},j}\bm{V}_{i,j}^{\mathrm{T}}\mathring{\bm{Y}}_{i+\frac{1}{2},j}\jumpangle{\bm{U}}_{i+\frac{1}{2},j}^{\tt WENO} -\frac{1}{2\Delta \xi_{1}}\left|J\frac{\partial \xi_1}{\partial t}\right|_{i-\frac{1}{2},j}\bm{V}_{i,j}^{\mathrm{T}}\mathring{\bm{Y}}_{i-\frac{1}{2},j}\jumpangle{\bm{U}}_{i-\frac{1}{2},j}^{\tt WENO}\nonumber\\
    &+\frac{1}{2\Delta \xi_{2}}\left|J\frac{\partial \xi_2}{\partial t}\right|_{i,j+\frac{1}{2}}\bm{V}_{i,j}^{\mathrm{T}}\mathring{\bm{Y}}_{i,j+\frac{1}{2}}\jumpangle{\bm{U}}_{i,j+\frac{1}{2}}^{\tt WENO} -\frac{1}{2\Delta \xi_{2}}\left|J\frac{\partial \xi_2}{\partial t}\right|_{i,j-\frac{1}{2}}\bm{V}_{i,j}^{\mathrm{T}}\mathring{\bm{Y}}_{i,j-\frac{1}{2}}\jumpangle{\bm{U}}_{i,j-\frac{1}{2}}^{\tt WENO}\nonumber\\
    =&-\frac{1}{\Delta \xi_{1}}\left(\left({\widetilde{{\mathcal{Q}}}}_1\right)_{i+\frac{1}{2},j}^{\tt ES}-\left({\widetilde{{\mathcal{Q}}}}_1\right)_{i-\frac{1}{2},j}^{\tt ES}\right)-\frac{1}{\Delta \xi_{2}}\left(\left({\widetilde{{\mathcal{Q}}}}_2\right)_{i,j+\frac{1}{2}}^{\tt ES}-\left({\widetilde{{\mathcal{Q}}}}_2\right)_{i,j-\frac{1}{2}}^{\tt ES}\right)\nonumber\\
    &-\frac{1}{4\Delta \xi_{1}}\left(\alpha_1\right)_{i+\frac{1}{2},j}\jump{\widehat{\bm{V}}}_{i+\frac{1}{2},j}^{\mathrm{T}}\left(\bm{T}^{-1} \bm{R}(\bm{T}\widehat{\bm{U}})\right)_{i+\frac{1}{2},j}\widehat{\bm{Y}}_{i+\frac{1}{2},j}\jumpangle{\bm{\widetilde{V}}}_{i+\frac{1}{2},j}^{\tt WENO}\nonumber\\
    &-\frac{1}{4\Delta \xi_{1}}\left(\alpha_1\right)_{i-\frac{1}{2},j}\jump{\widehat{\bm{V}}}_{i-\frac{1}{2},j}^{\mathrm{T}}\left(\bm{T}^{-1} \bm{R}(\bm{T}\widehat{\bm{U}})\right)_{i-\frac{1}{2},j}\widehat{\bm{Y}}_{i-\frac{1}{2},j}\jumpangle{\bm{\widetilde{V}}}_{i-\frac{1}{2},j}^{\tt WENO}\nonumber\\
    &-\frac{1}{4\Delta \xi_{2}}\left(\alpha_2\right)_{i,j+\frac{1}{2}}\jump{\widehat{\bm{V}}}_{i,j+\frac{1}{2}}^{\mathrm{T}}\left(\bm{T}^{-1} \bm{R}(\bm{T}\widehat{\bm{U}})\right)_{i,j+\frac{1}{2}}\widehat{\bm{Y}}_{i,j+\frac{1}{2}}\jumpangle{\bm{\widetilde{V}}}_{i,j+\frac{1}{2}}^{\tt WENO}\nonumber\\
    &-\frac{1}{4\Delta \xi_{2}}\left(\alpha_2\right)_{i,j-\frac{1}{2}}\jump{\widehat{\bm{V}}}_{i,j-\frac{1}{2}}^{\mathrm{T}}\left(\bm{T}^{-1} \bm{R}(\bm{T}\widehat{\bm{U}})\right)_{i,j-\frac{1}{2}}\widehat{\bm{Y}}_{i,j-\frac{1}{2}}\jumpangle{\bm{\widetilde{V}}}_{i,j-\frac{1}{2}}^{\tt WENO}\nonumber\\
    &-\frac{1}{4\Delta \xi_{1}}\left|J\frac{\partial \xi_1}{\partial t}\right|_{i+\frac{1}{2},j}\jump{\bm{V}}_{i+\frac{1}{2},j}^{\mathrm{T}}\mathring{\bm{Y}}_{i+\frac{1}{2},j}\jumpangle{\bm{U}}_{i+\frac{1}{2},j}^{\tt WENO} -\frac{1}{4\Delta \xi_{1}}\left|J\frac{\partial \xi_1}{\partial t}\right|_{i-\frac{1}{2},j}\jump{\bm{V}}_{i-\frac{1}{2},j}^{\mathrm{T}}\mathring{\bm{Y}}_{i-\frac{1}{2},j}\jumpangle{\bm{U}}_{i-\frac{1}{2},j}^{\tt WENO}\nonumber\\
    &-\frac{1}{4\Delta \xi_{2}}\left|J\frac{\partial \xi_2}{\partial t}\right|_{i,j+\frac{1}{2}}\jump{\bm{V}}_{i,j+\frac{1}{2}}^{\mathrm{T}}\mathring{\bm{Y}}_{i,j+\frac{1}{2}}\jumpangle{\bm{U}}_{i,j+\frac{1}{2}}^{\tt WENO} -\frac{1}{4\Delta \xi_{2}}\left|J\frac{\partial \xi_2}{\partial t}\right|_{i,j-\frac{1}{2}}\jump{\bm{V}}_{i,j-\frac{1}{2}}^{\mathrm{T}}\mathring{\bm{Y}}_{i,j-\frac{1}{2}}\jumpangle{\bm{U}}_{i,j-\frac{1}{2}}^{\tt WENO}\nonumber.
  \end{align}
  Due to the fact that
  $\jump{\widehat{\bm{V}}}_{i\pm\frac12,j}^{\mathrm{T}}\left(\bm{T}^{-1} \bm{R}(\bm{T}\widehat{\bm{U}})\right)_{i\pm\frac12,j} = \jump{\bm{\widetilde{V}}}_{i\pm\frac12,j}^{\mathrm{T}}$,
  $\jump{\widehat{\bm{V}}}_{i,j\pm\frac12}^{\mathrm{T}}\left(\bm{T}^{-1} \bm{R}(\bm{T}\widehat{\bm{U}})\right)_{i,j\pm\frac12} = \jump{\bm{\widetilde{V}}}_{i,j\pm\frac12}^{\mathrm{T}}$,
  and from the definitions of $\widehat{\bm{Y}}$ and $\mathring{\bm{Y}}$,
  one has
  \begin{equation*}
    \begin{aligned}
      \jump{\bm{\widetilde{V}}}^{\mathrm{T}}_{i\pm\frac{1}{2},j}\widehat{\bm{Y}}_{i\pm\frac{1}{2},j}\jumpangle{\bm{\widetilde{V}}}_{i\pm\frac{1}{2},j}^{\tt WENO} \geqslant 0,\\
      \jump{\bm{\widetilde{V}}}^{\mathrm{T}}_{i,j\pm\frac{1}{2}}\widehat{\bm{Y}}_{i,j\pm\frac{1}{2}}\jumpangle{\bm{\widetilde{V}}}_{i,j\pm\frac{1}{2}}^{\tt WENO} \geqslant 0,\\
      \jump{\bm{{V}}}^{\mathrm{T}}_{i\pm\frac{1}{2},j}\mathring{\bm{Y}}_{i\pm\frac{1}{2},j}\jumpangle{\bm{U}}_{i\pm\frac{1}{2},j}^{\tt WENO} \geqslant 0,\\
      \jump{\bm{{V}}}^{\mathrm{T}}_{i,j\pm\frac{1}{2}}\mathring{\bm{Y}}_{i,j\pm\frac{1}{2}}\jumpangle{\bm{U}}_{i,j\pm\frac{1}{2}}^{\tt WENO} \geqslant 0,\\
    \end{aligned}
  \end{equation*}
  thus, the semi-discrete energy inequality \eqref{eq:numEntropyInq} holds.
\end{proof}

\begin{remark}\rm
  When the solution is the lake at rest, one has $\bU=(h,0,0,b), \bV=(gC,0,0,gC+(2\gamma-1)gb)$, where $C$ is a constant.
  Then the first and last components in the second switch function
  depend on the signs of the jump of $b$ before and after reconstruction when $\gamma> 1/2$,
  thus we choose $\gamma=1$ in all the numerical tests.
  The location where the second dissipation term for $b$ is added will be shown in the numerical tests.
\end{remark}

\subsection{Time discretization}
To get the fully-discrete schemes, this paper uses the explicit SSP RK3 time discretization 
\begin{equation*}
  \begin{aligned}
    \bm{\mathcal{U}}^{*} &= \bm{\mathcal{U}}^{n}+\Delta t^n \bm{\mathcal{L}}\left(\bU^{n}, \bx^n\right),\\
    J^{*} &= J^{n}+\Delta t^n {\mathcal{L}}\left(
    \bx^n\right),\\
    \bx^{*} &= \bx^{n}+\Delta t^n \dot{\bx}^n,
    \\
    \bm{\mathcal{U}}^{**} &= \frac{3}{4}\bm{\mathcal{U}}^{n}+\frac{1}{4}\left(\bm{\mathcal{U}}^{*}+\Delta t^n \bm{\mathcal{L}}\left(\bU^{*},
    {\bx}^*\right)\right),\\
    J^{**} &= \frac{3}{4}J^{n}+\frac{1}{4}\left(J^{*}+\Delta t^n {\mathcal{L}}\left(
    {\bx}^*\right)\right),\\
    \bx^{**} &= \frac{3}{4}\bx^{n}+\frac{1}{4}\left(\bx^{*} + \Delta t^n\dot{\bx}^n\right) = \bx^{n}+\frac12\Delta t^n \dot{\bx}^n,
    \\
    \bm{\mathcal{U}}^{n+1} &= \frac{1}{3}\bm{\mathcal{U}}^{n}+\frac{2}{3}\left(\bm{\mathcal{U}}^{**}+\Delta t^n \bm{\mathcal{L}}\left(\bU^{**},
    {\bx}^{**}\right)\right),\\
    J^{n+1} &= \frac{1}{3}J^{n}+\frac{2}{3}\left(J^{**}+\Delta t^n {\mathcal{L}}\left(
    {\bx}^{**}\right)\right),\\
    \bx^{n+1} &= \frac{1}{3}\bx^{n}+\frac{2}{3}\left(\bx^{**} + \Delta t^n{\bx}^n\right) = \bx^{n}+\Delta t^n \dot{\bx}^n,
  \end{aligned}
\end{equation*}
where $\dot{\bx}^n$ is the mesh velocity determined by using the adaptive moving mesh strategy in \cite{Duan2022High,Li2022High},
$\bm{\mathcal{L}}$ is the right-hand side of
\eqref{eq:2D_HighOrder_EC_Discrete} or \eqref{eq:2D_HighOrder_Discrete_ES}
for the semi-discrete high-order WB EC or ES schemes, respectively,
and ${\mathcal{L}}$ is the right-hand side of the semi-discrete VCL \eqref{eq:2D_VCL_HighOrder}.

\section{Numerical results}\label{section:Result}
In this section, several numerical tests are conducted to demonstrate the performance of our high-order accurate WB ES adaptive moving mesh schemes.
In all the numerical experiments, the parameter $\gamma$ in \eqref{eq:EntropyPair} is taken as $1$.
The reconstruction in the dissipation term adopts the $5$th-order WENO-Z reconstruction \cite{Borges2008An}.
For the adaptive mesh redistribution,
the total number of iterations for solving the adaptive mesh equation (see \cite{Duan2022High,Li2022High}) is $10$,
and the monitor function $\omega$ will be given in each example.
Unless otherwise specified, the time stepsize $\Delta t^{n}$ of the 2D schemes is
\begin{equation}\label{eq:dt}
  \Delta t^{n} \leqslant \frac{C_{\text{\tiny \tt CFL}}}{\sum\limits_{\ell = 1}^{2}\max_{i,j}\limits\left(\mathcal{\varrho}_{\ell}\right)^{n}_{i,j}/\Delta {\xi_\ell}},
\end{equation}
where $C_{\text{\tiny \tt CFL}}=0.4$ is the CFL number,
and $\left(\mathcal{\varrho}_{\ell}\right)_{i,j}^{n}$ in the $\xi_1$-direction is defined in \eqref{eq:eigens}.
In the accuracy test, the time stepsize is taken as $C_{\text{\tiny \tt CFL}} (\min_{\ell}\Delta\xi_{\ell})^{5/3}$ to make the spatial error dominant.
For the 1D tests, the corresponding 1D schemes are detailed in \ref{section:OneDimensionCase}.
Our WB EC and ES schemes on the fixed uniform mesh and moving mesh will be denoted by ``{\tt UM-EC}'',
``{\tt UM-ES}'', ``{\tt MM-EC}'', and ``{\tt MM-ES}''.

\subsection{1D tests}
\begin{example}[Accuracy test with manufactured solution]\label{ex:1DSmooth}\rm
  This example is first used to test the convergence rates of our schemes.
  The physical domain is $[0,2]$ with the periodic boundary conditions.
  To construct a manufactured solution,
  this test solves the following equations with extra source terms
  \begin{equation*}
    \pd{\bU}{t} + \pd{\bF}{x} = -gh\pd{\bm{B}}{x} + \bm{S},
  \end{equation*}
  and the exact solution is chosen as
  \begin{equation*}
    h(x,t) = 4+\cos(\pi x)\cos(\pi t),~v_1(x,t) = \frac{\sin(\pi x)\sin(\pi t)}{h},~b(x) = 1.5+\sin(\pi x),
  \end{equation*}
  with the gravitational acceleration $g=1$,
  then the extra source terms are
  \begin{align*}
    \bm{S}(x,t) = (0, ~&4\pi \cos(\pi x) + \pi\cos(\pi t)\cos^2(\pi x) - 3\pi\cos(\pi t)\sin(\pi x) - \pi\cos^2(\pi t)\cos(\pi x)\sin(\pi x) \\&+ (\pi\cos(\pi t)\sin^2(\pi t) \sin^3(\pi x))/(\cos(\pi t)\cos(\pi x) + 4)^2 \\&+ (2\pi\cos(\pi x)\sin^2(\pi t)\sin(\pi x))/(\cos(\pi t)\cos(\pi x) + 4)
    ,~0)^{\mathrm{T}}.
  \end{align*}
  The output time is $t = 0.2$ and the monitor function is
  \begin{equation*}
    \omega=\left({1+\frac{\theta |\nabla_{{\xi}} (h+b)|}{\max |\nabla_{{\xi}} (h+b)|}}\right)^{1/2},
  \end{equation*}
  with $\theta = 10$.
\end{example}

Figure \ref{1D_Smooth_balance} shows the shape of the bottom topography and the water surface level $h+b$ obtained by the {\tt MM-EC} and {\tt MM-ES} schemes.
Figure \ref{1D test} plots the errors and convergence rates in water depth $h$ obtained by the {\tt UM-EC}, {\tt UM-ES}, {\tt MM-EC}, and {{\tt MM-ES}} schemes at $t =0.2$, which show that our schemes can achieve the expected $6$th- and $5$th-order accuracy, respectively.
\begin{figure}[hbt!]
  \centering
  \includegraphics[width=0.4\textwidth]{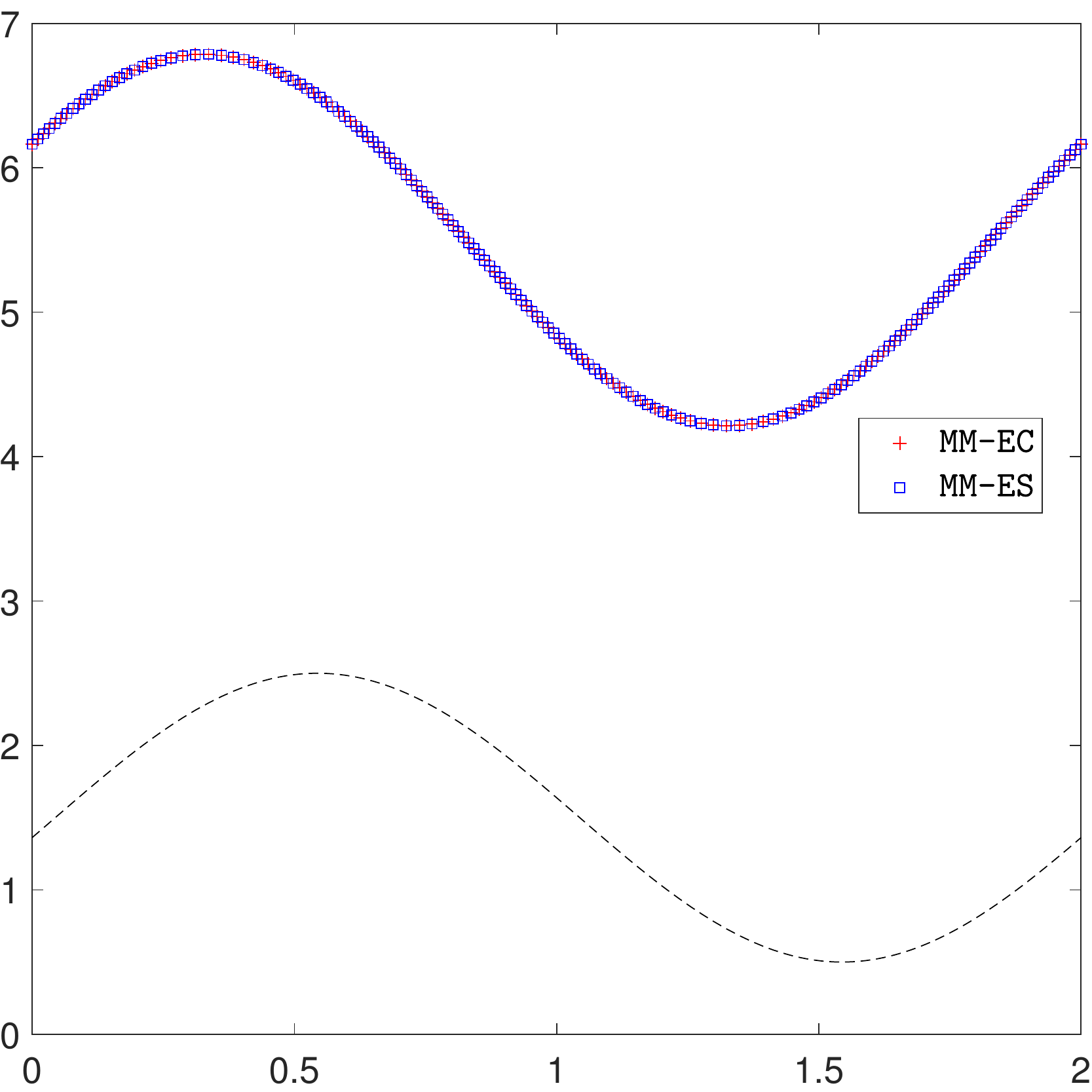}
  \caption{Example \ref{ex:1DSmooth}. The numerical solutions of water surface level $h+b$ obtained by the {\tt MM-EC} and {\tt MM-ES} schemes with $100$ mesh points at $t$ =0.2. The dashed line denotes the bottom topography.}
  \label{1D_Smooth_balance}
\end{figure}

\begin{figure}[hbt!]
  \centering
  \includegraphics[width=0.4\textwidth]{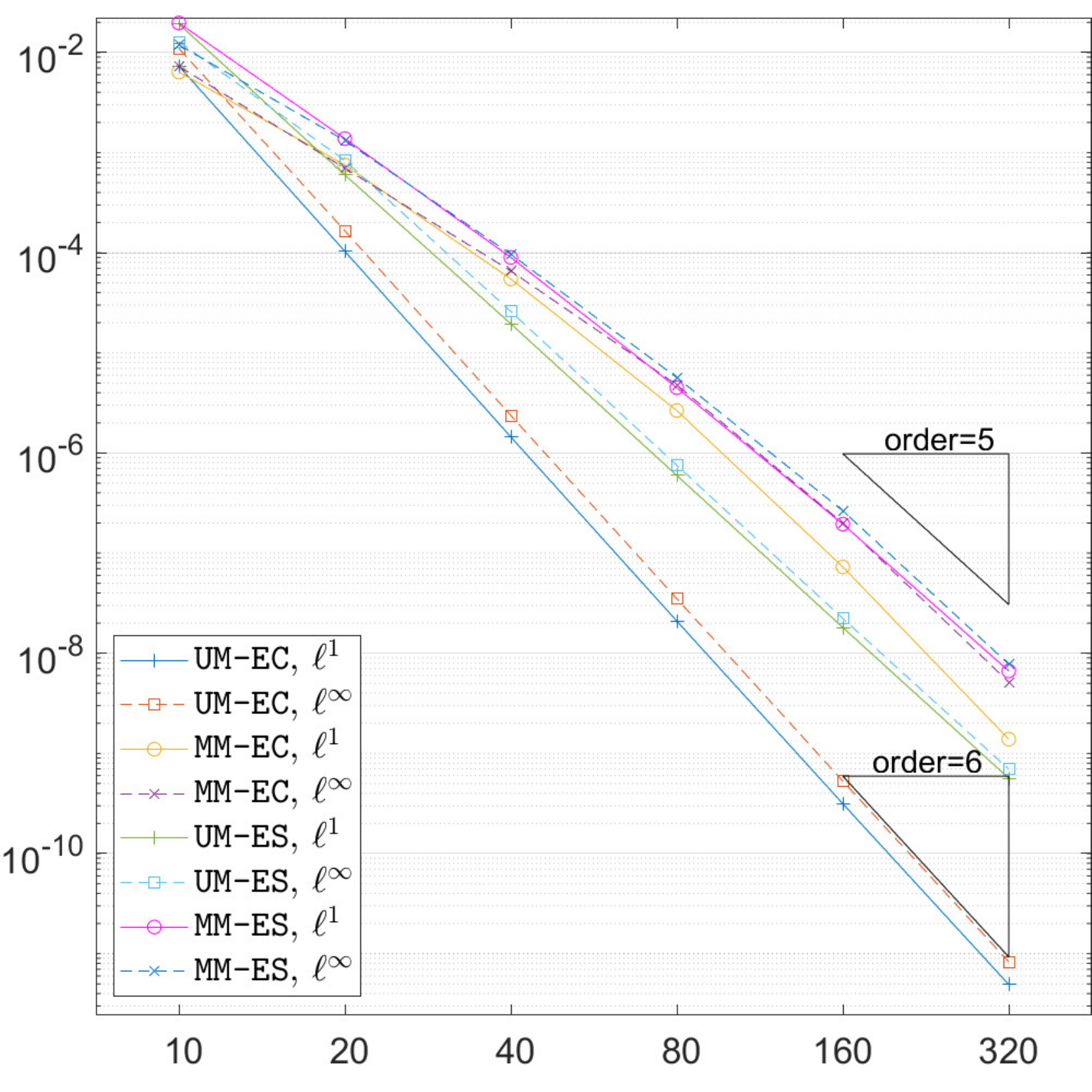}
  \caption{Example \ref{ex:1DSmooth}. The errors and convergence orders in water depth $h$ at $t =0.2$.}\label{1D test}
\end{figure}

\begin{example}[1D WB test]\label{ex:1D_WB_Test}\rm
  This example is used to verify the WB property of our schemes for the 1D SWEs.
  The physical domain is $[0, 10]$ with the outflow boundary conditions,
  and the bottom topography is a smooth Gaussian profile
  \begin{equation}\label{eq:b_Smooth}
    b(x) = 5e^{-\frac{2}{5}(x-5)^2},
  \end{equation}
  or a discontinuous square step
  \begin{equation}\label{eq:b_dis}
    b(x)= \begin{cases}4,~&\text{if}\quad x \in[4,8], \\ 0,~&\text{otherwise}.\end{cases}
  \end{equation}
  The initial water depth is $h = 10 - b$ with zero velocity.
  The output time is $t = 0.2$, and the gravitational acceleration constant is taken as $g = 1$.
  The monitor function is
  \begin{equation*}
    \omega=\left({1+\frac{100|\nabla_{{\xi}} h|}{\max |\nabla_{{\xi}} h|}}\right)^{1/2}.
  \end{equation*}
\end{example}

Table \ref{1D Well_Balance} gives the errors in $h+b$ and $v_1$ at $t = 0.2$,
which are at the level of rounding error in double precision, thus our schemes are WB.
The bottom topography $b$ and water surface level $h+b$ obtained by using the {\tt UM-ES} and {\tt MM-ES} schemes with $100$ mesh points are plotted in Figure \ref{1D_Well_balance},
from which one can see that our schemes can preserve the lake at rest well.
To see the effects of adding the second dissipation term ${\mathring{\bm{D}}}$ in \eqref{eq:HO_ES_flux},
the results obtained by the {\tt MM-ES} without that term are compared in Figure \ref{1D_Well_balance_Compare}.
For the smooth bottom topography, {\tt MM-ES} with and without $\mathring{\bm{D}}$ can give similar results,
while {\tt MM-ES} without $\mathring{\bm{D}}$ produces some overshoots or undershoots near the discontinuities,
which leads to overshoots or undershoots in $h$ as $h+b$ is constant.
The location where ${\mathring{\bm{D}}}$ is added for the bottom topography \eqref{eq:b_dis} is also plotted in Figure \ref{1D_Well_balance_Compare},
which shows that when $b$ is discontinuous,
the second dissipation term is almost added.
The results indicate that the second dissipation term is vital in the construction of our schemes.

\begin{table}[hbt!]
  \centering
  \begin{tabular}{c|c|cc|cc|cc|cc}
    \toprule
    \multicolumn{2}{c|}{\multirow{2}{*}{}} & \multicolumn{2}{c|}{{\tt UM-EC}} & \multicolumn{2}{c|}{{\tt UM-ES}} & \multicolumn{2}{c|}{{\tt MM-ES}} & \multicolumn{2}{c}{{\tt MM-ES} without $\bm{\mathring{D}}$ } \\ 
    \cline{3-10}
    \multicolumn{2}{c|}{} & \multicolumn{1}{c}{$\ell^{1}$~error}  & \multicolumn{1}{c|}{$\ell^{\infty}$~error} &  \multicolumn{1}{c}{$\ell^{1}$~error}  & \multicolumn{1}{c|}{$\ell^{\infty}$~error} &  \multicolumn{1}{c}{$\ell^{1}$~error}  & \multicolumn{1}{c|}{$\ell^{\infty}$~error} &  \multicolumn{1}{c}{$\ell^{1}$~error}  & \multicolumn{1}{c}{$\ell^{\infty}$~error}   \\
    \hline
    \multirow{2}{*}{$b$ in \eqref{eq:b_Smooth}} &
    $h+b$ & 6.57e-15 	 & 5.32e-15 	 & 2.84e-15 	 & 3.55e-15 & 9.41e-14 	 & 1.28e-13 & 3.60e-14 &1.24e-14\\ 
    & $v_1$ & 2.64e-15 	 & 1.73e-15 	 & 1.60e-15 	 & 1.24e-15 & 3.09e-14 	 & 3.39e-14 & 1.12e-14&   4.20e-15\\ 
    \hline
    \multirow{2}{*}{$b$ in \eqref{eq:b_dis}} &
    $h+b$ & 3.14e-14 	 & 1.24e-14 	 & 1.44e-14 	 & 7.11e-15 & 4.67e-14 	 & 2.66e-14 &  3.58e-14 &  1.06e-14\\ 
    & $v_1$ & 8.37e-15 	 & 3.35e-15 	 & 4.77e-15 	 & 1.87e-15 & 1.72e-14 	 & 8.90e-15 &  1.07e-14 & 3.33e-15\\ 
    \bottomrule
  \end{tabular}
  \caption{Example \ref{ex:1D_WB_Test}. Errors in $h+b$ and $v_1$ obtained by our schemes using $100$ mesh points at $t=0.2$, with the bottom topography \eqref{eq:b_Smooth} and \eqref{eq:b_dis}.}\label{1D Well_Balance}
\end{table}

\begin{figure}[hbt!]
  \centering
  \begin{subfigure}[b]{0.3\textwidth}
    \centering
    \includegraphics[width=1.0\linewidth,  clip]{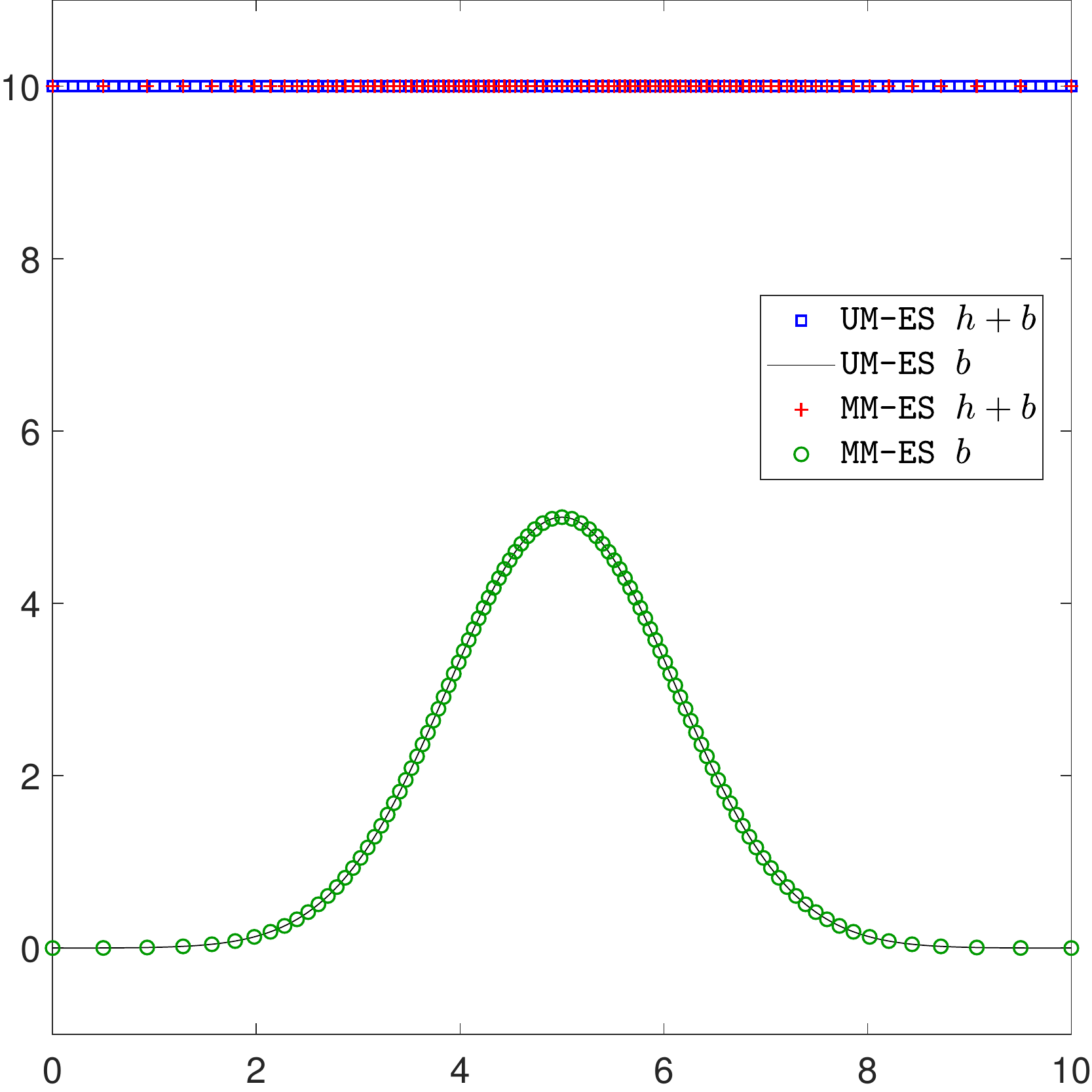}
  \end{subfigure}
  \begin{subfigure}[b]{0.3\textwidth}
    \centering
    \includegraphics[width=1.0\linewidth]{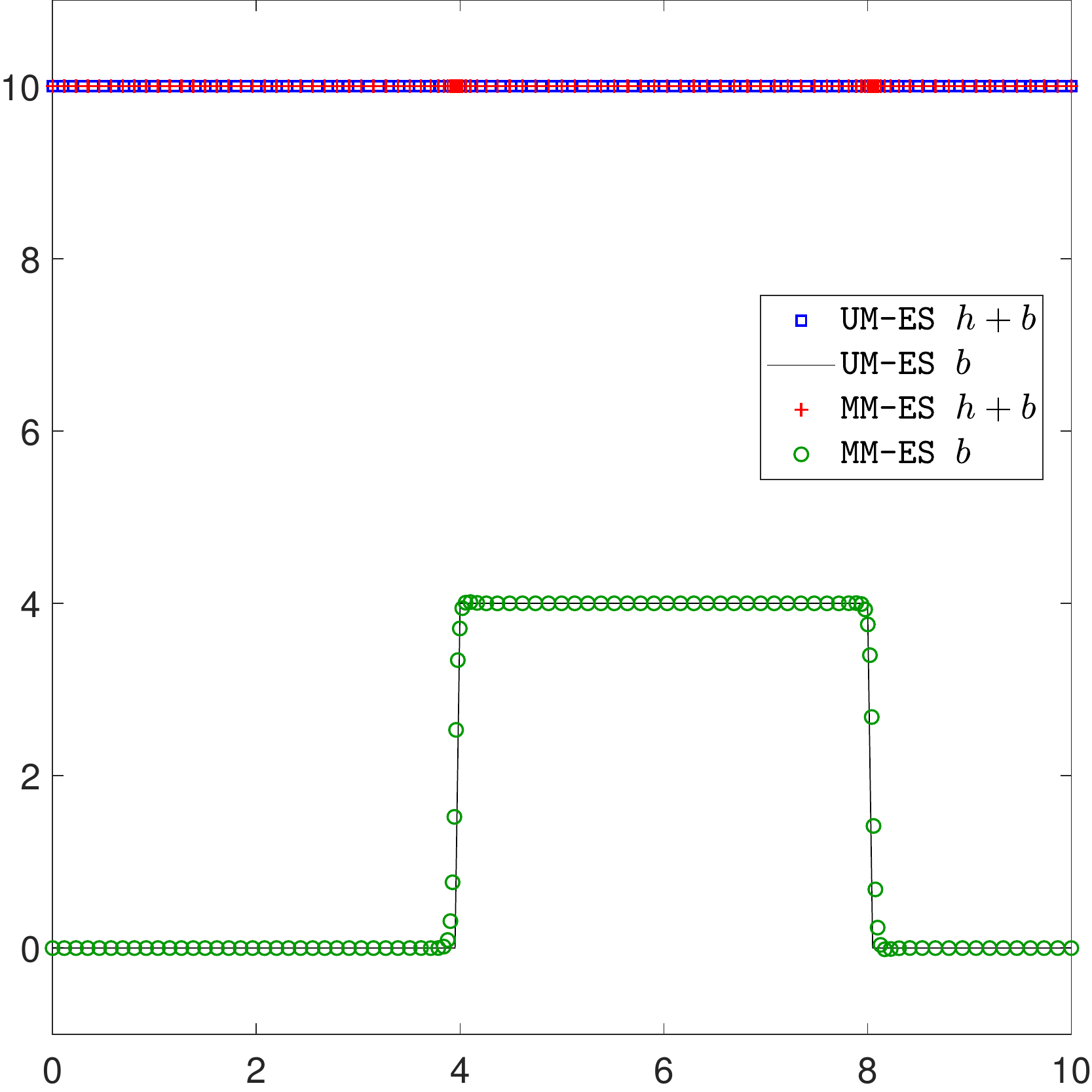}
  \end{subfigure}
  \caption{Example \ref{ex:1D_WB_Test}. The bottom topography $b$ and water surface level $h+b$ obtained by using the {\tt UM-ES} and {\tt MM-ES} schemes with $100$ mesh points at $t =0.2$.
    Left: with the bottom topography \eqref{eq:b_Smooth},
    right: with the bottom topography \eqref{eq:b_dis}.
  }
  \label{1D_Well_balance}
\end{figure}
\begin{figure}[hbt!]
  \centering
  \begin{subfigure}[b]{0.3\textwidth}
    \centering
    \includegraphics[width=1.0\linewidth,  clip]{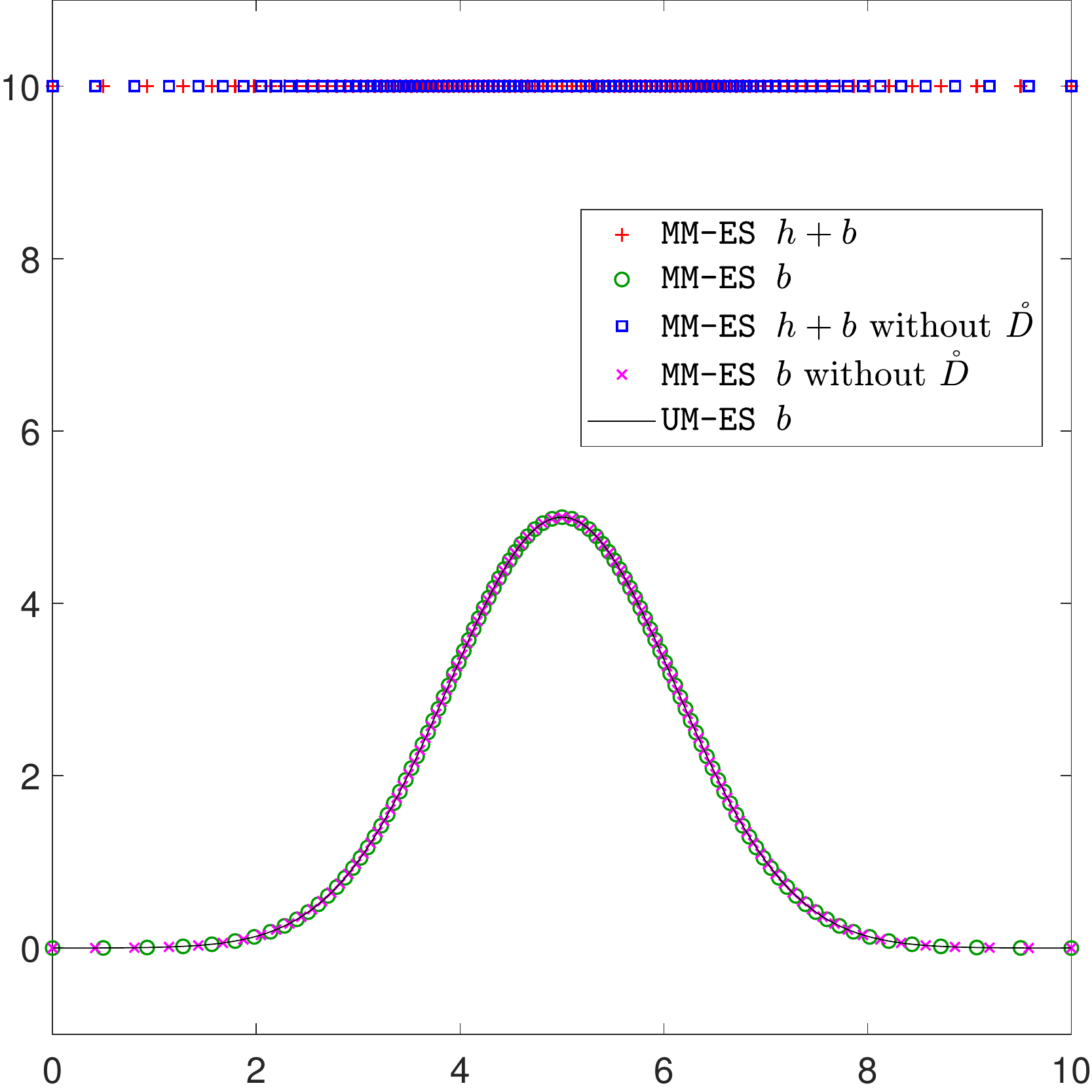}
  \end{subfigure}
  \begin{subfigure}[b]{0.3\textwidth}
    \centering
    \includegraphics[width=1.0\linewidth]{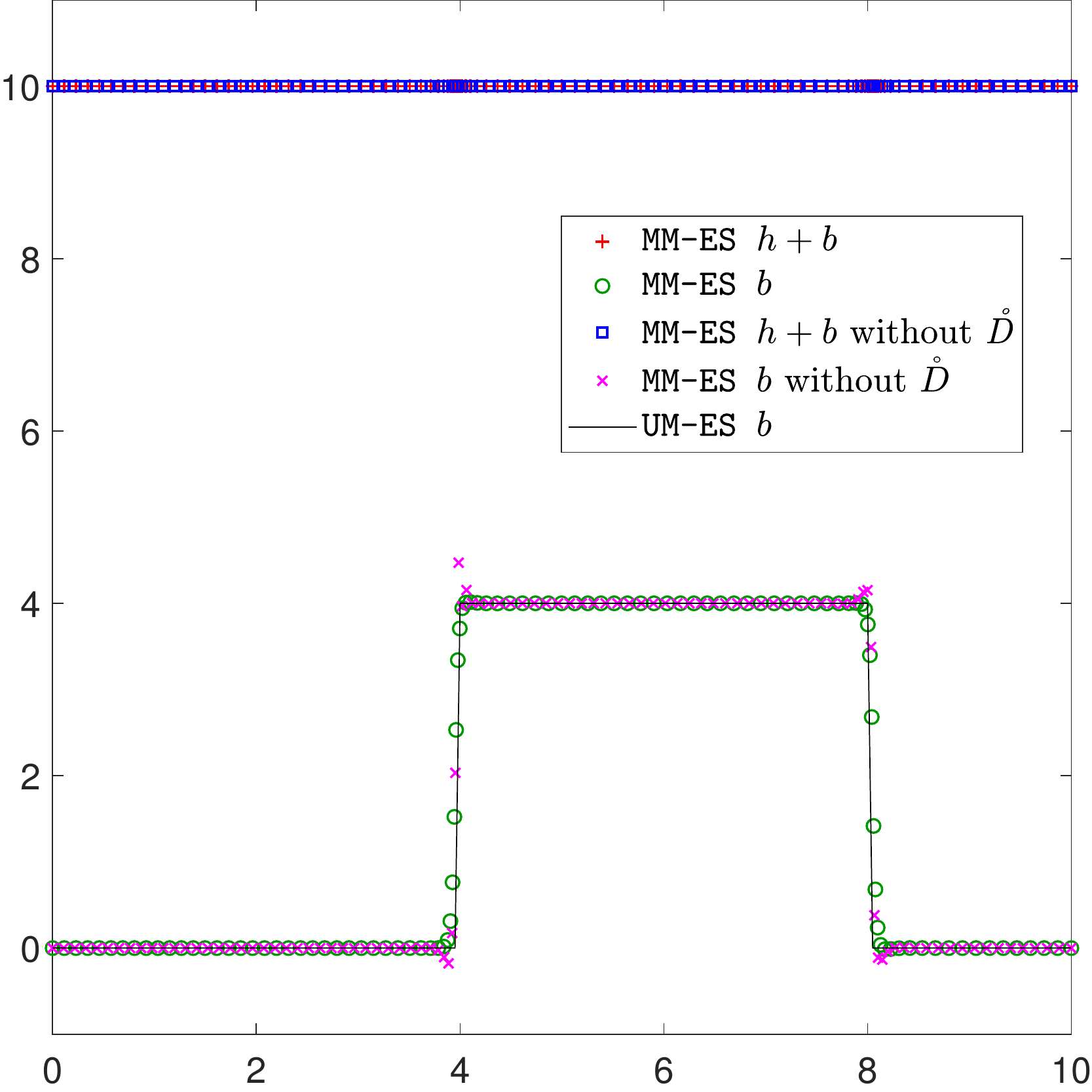}
  \end{subfigure}
  \begin{subfigure}[b]{0.3\textwidth}
    \centering
    \includegraphics[width=1.0\linewidth]{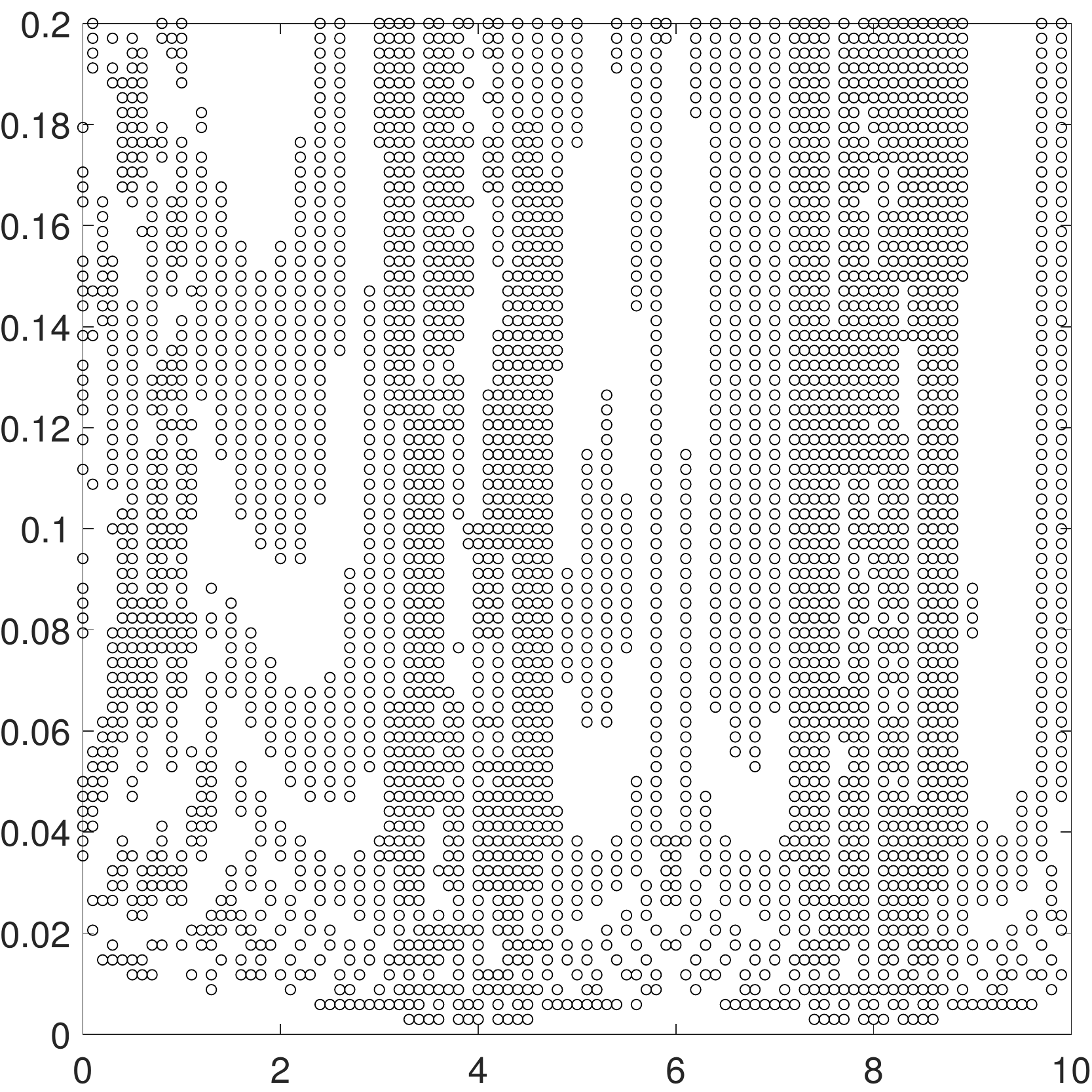}
  \end{subfigure}
  \caption{Example \ref{ex:1D_WB_Test}. The bottom topography $b$ and water surface level $h+b$ obtained by using different schemes with $100$ mesh points at $t$ =0.2.
    Left: with the bottom topography \eqref{eq:b_Smooth},
    middle: with the bottom topography \eqref{eq:b_dis},
    right: the dots denote where the second dissipation term $\mathring{\bm{D}}$ for $b$ is added in the {\tt MM-ES} scheme in the $x-t$ coordinate.
  }
  \label{1D_Well_balance_Compare}
\end{figure}

\begin{example}[Small perturbation test]\label{ex:1D_Pertubation_Test}\rm
  To test the ability of the {\tt MM-ES} scheme to capture small perturbations in the steady-state flow, the bottom topography consisting of a ``hump'' \cite{Leveque1998Balancing} is considered here
  \begin{equation*}
    b(x) = \begin{cases}
      0.25(\cos(10\pi(x-1.5))+1), &\text{if}\quad 1.4 \leqslant x \leqslant 1.6,\\
      0,   &\text{otherwise},
    \end{cases}
  \end{equation*}
  and the initial water depth is
  \begin{equation*}
    h = \begin{cases}
      1-b+\epsilon, &\text{if}\quad 1.1\leqslant x\leqslant 1.2,\\
      1-b, &\text{otherwise},
    \end{cases}
  \end{equation*}
  with the initial zero velocity.
  The physical domain is $[0, 2]$ with the outflow boundary conditions.
  The gravitational acceleration constant is $g = 9.812 $, and the magnitude of the small perturbation is taken as $\epsilon = 0.2$ and $0.001$.
  The choice of monitor function is the same as Example \ref{ex:1DSmooth} with $\theta = 100$.
\end{example}

Figure \ref{1D_ES_Pertubation_WB} presents the numerical solutions obtained by the {\tt UM-ES} and {\tt MM-ES} schemes with $200$ mesh points at $t=0.2$,
and the reference solution is obtained by the {\tt UM-ES} scheme using $3000$ mesh points.
It is seen that the structures in the solution are captured well and there is no obvious numerical oscillation.
Meanwhile, the results obtained by the {\tt MM-ES} scheme are similar to those obtained by the {\tt UM-ES} scheme with $600$ mesh points.
Furthermore, the mesh trajectories show that the mesh points concentrate near where the water surface level $h+b$ changes rapidly,
matching the choice of the monitor function.
To examine the ES properties,
the physical domain is enlarged as $[-5, 5]$ to exclude the influence of the boundary conditions.
Figure \ref{1D_ES_Pertubation_WB_Entropy} gives the evolution of the discrete total energy $\sum\limits_i \eta(\bU_i)\Delta x$ (on the fixed uniform mesh) and $\sum\limits_i J_i \eta(\bU_i)\Delta \xi$ (on the moving mesh) obtained by the {\tt UM-ES} and {\tt MM-ES} schemes with $1000$ mesh points, respectively,
from which one can see that the discrete total energy decays as expected.
The results obtained by using the {\tt MM-ES} without the second dissipation term ${\mathring{\bm{D}}}$ in \eqref{eq:HO_ES_flux} are compared in Figure \ref{1D_ES_Pertubation_WB_bottom},
which shows that without ${\mathring{\bm{D}}}$, the bottom topography in the numerical solution is oscillatory when its derivative is not continuous.

\begin{figure}[hbt!]
  \centering
  \begin{subfigure}[b]{0.35\textwidth}
    \centering
    \includegraphics[width=1.0\linewidth]{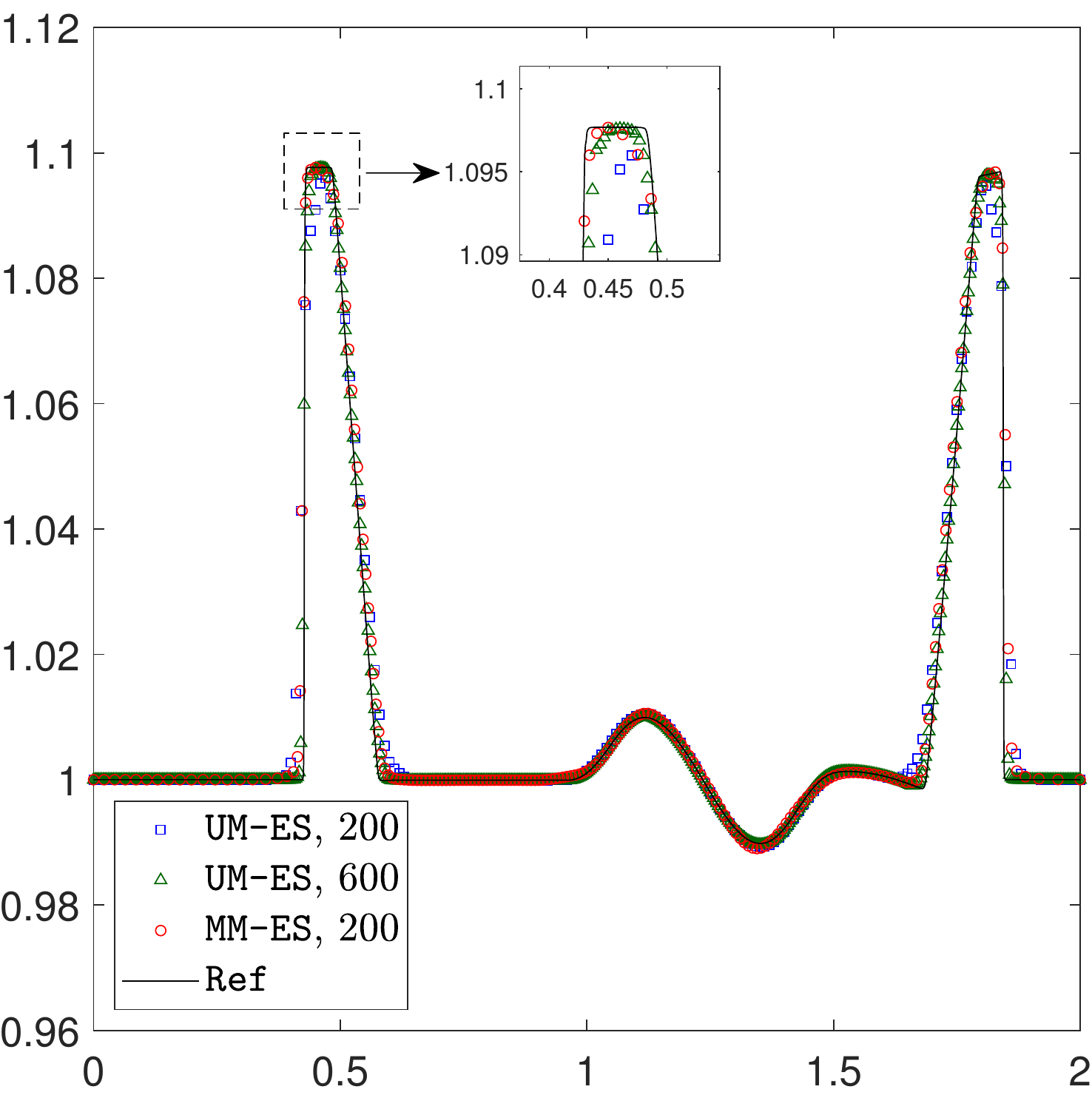}
    \caption{$h+b, \epsilon = 0.2$}
  \end{subfigure}
  \begin{subfigure}[b]{0.35\textwidth}
    \centering
    \includegraphics[width=1.0\linewidth]{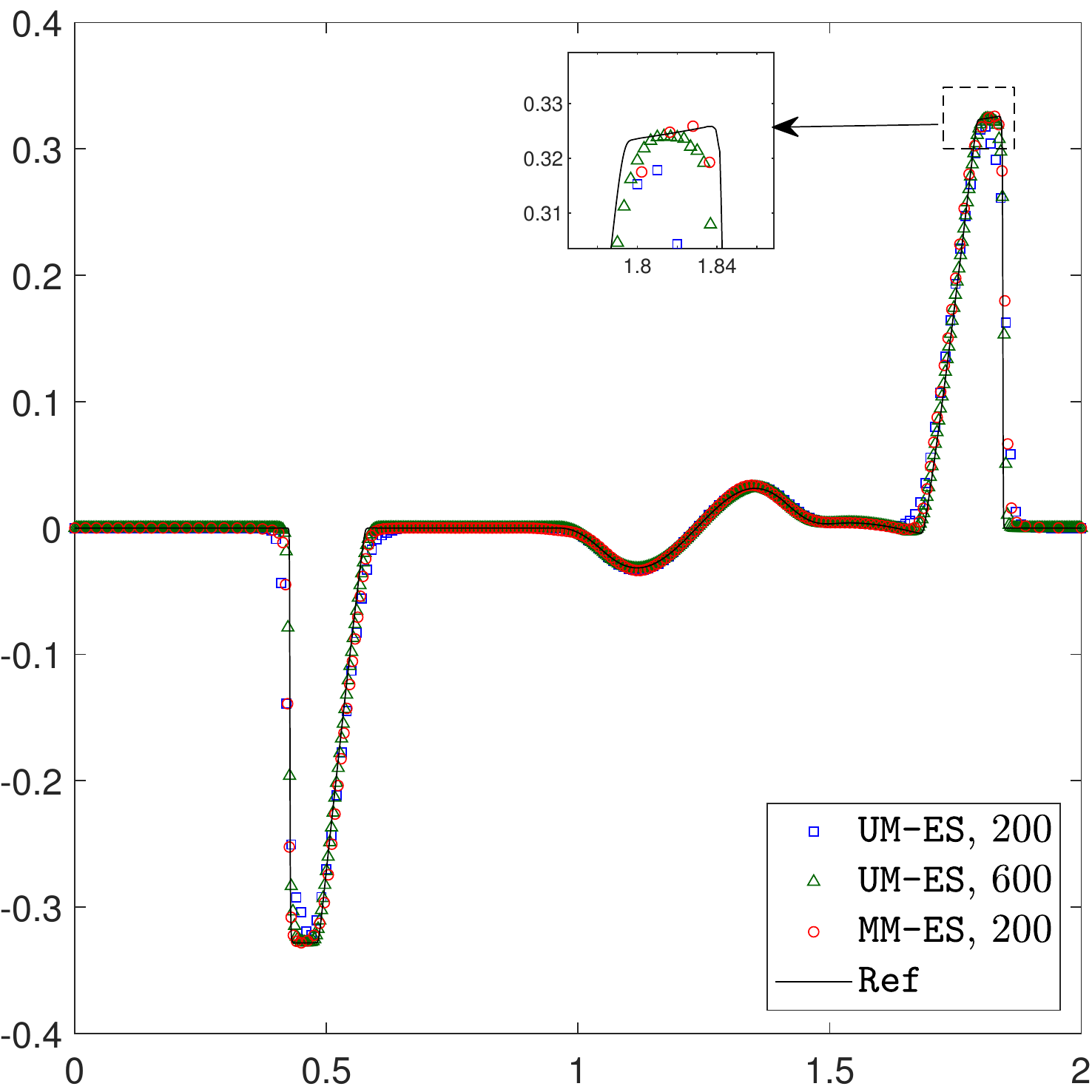}
    \caption{$hv_1, \epsilon = 0.2$}
  \end{subfigure}
  \\
  \begin{subfigure}[b]{0.366\textwidth}
    \centering
    \includegraphics[width=1.0\linewidth]{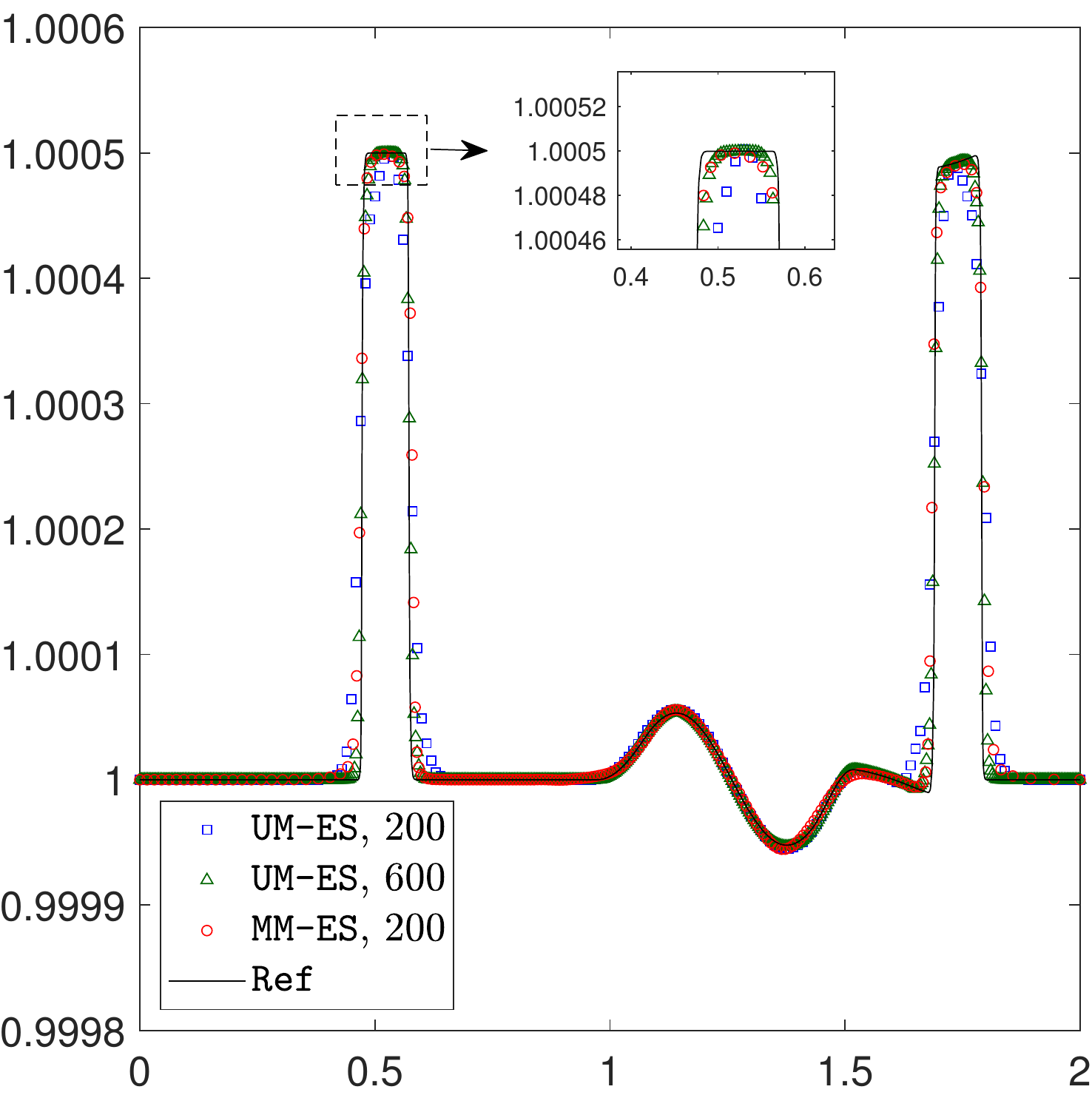}
    \caption{$h+b, \epsilon = 0.001$}
  \end{subfigure}
  \begin{subfigure}[b]{0.378\textwidth}
    \centering
    \includegraphics[width=0.93\linewidth,height=1.0\linewidth]{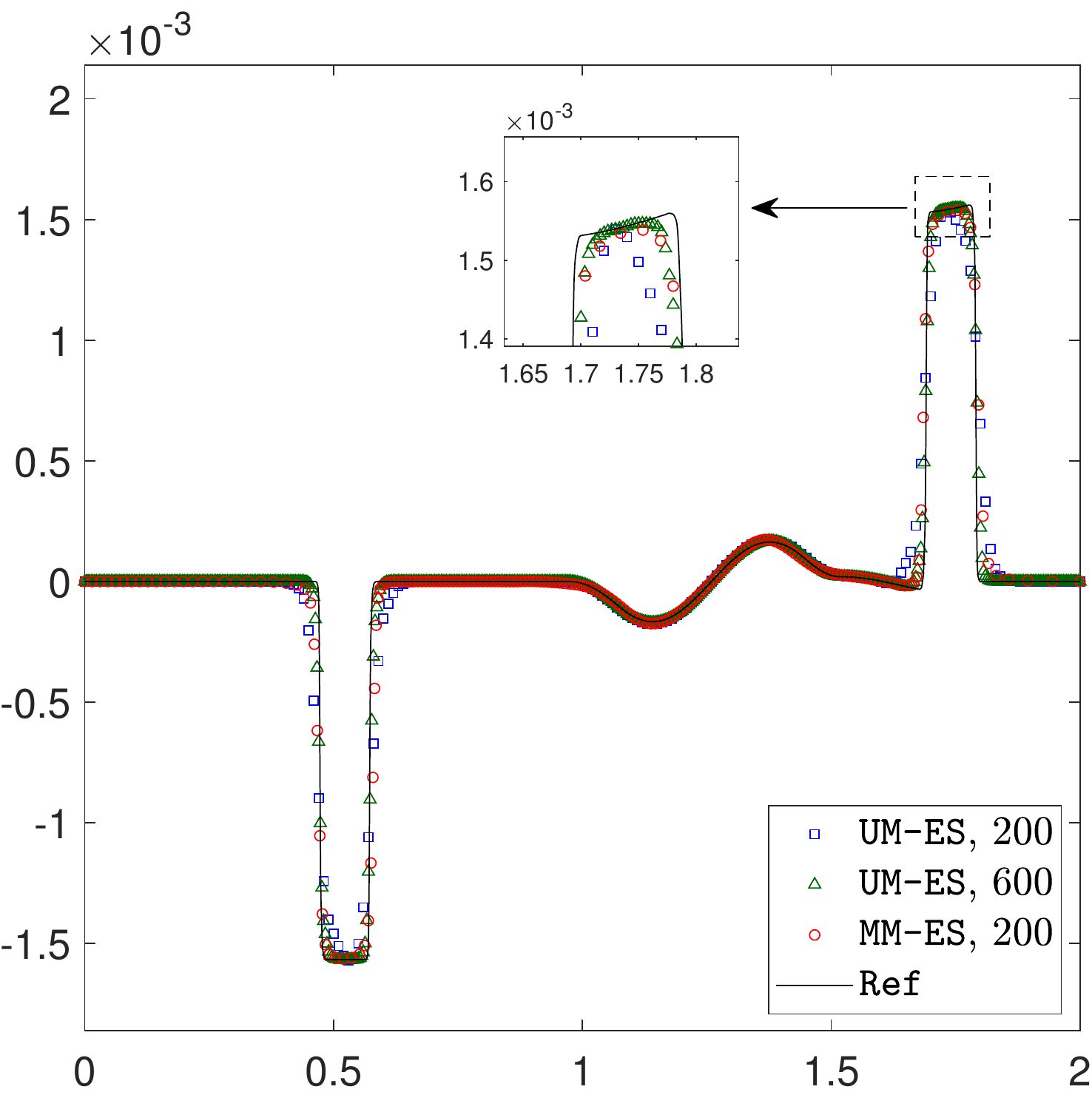}
    \caption{$hv_1, \epsilon = 0.001$}
  \end{subfigure}
  \\
  \begin{subfigure}[b]{0.35\textwidth}
    \centering
    \includegraphics[width=1.0\linewidth]{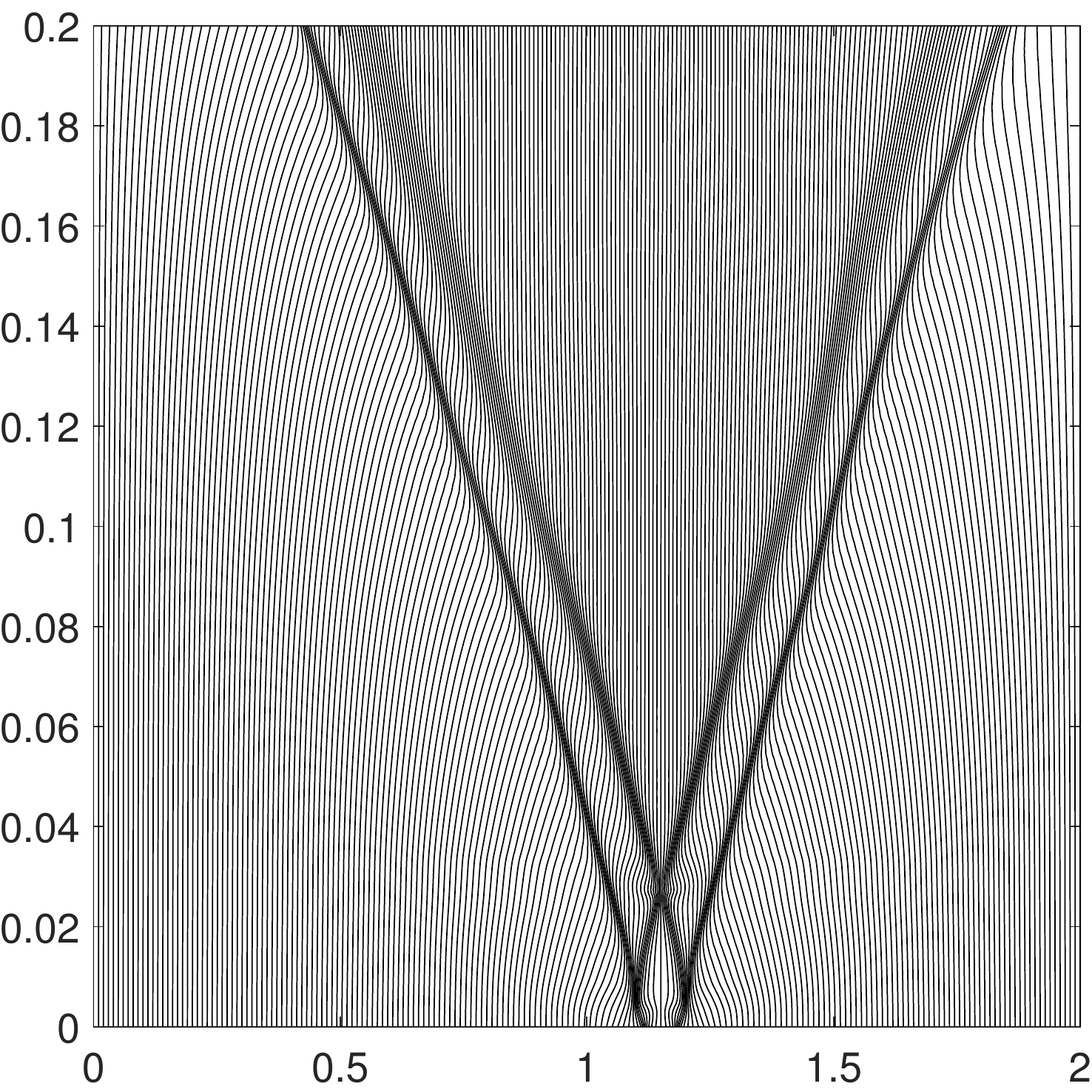}
    \caption{mesh trajectory, $\epsilon = 0.2$}
  \end{subfigure}
  \begin{subfigure}[b]{0.35\textwidth}
    \centering
    \includegraphics[width=1.0\linewidth]{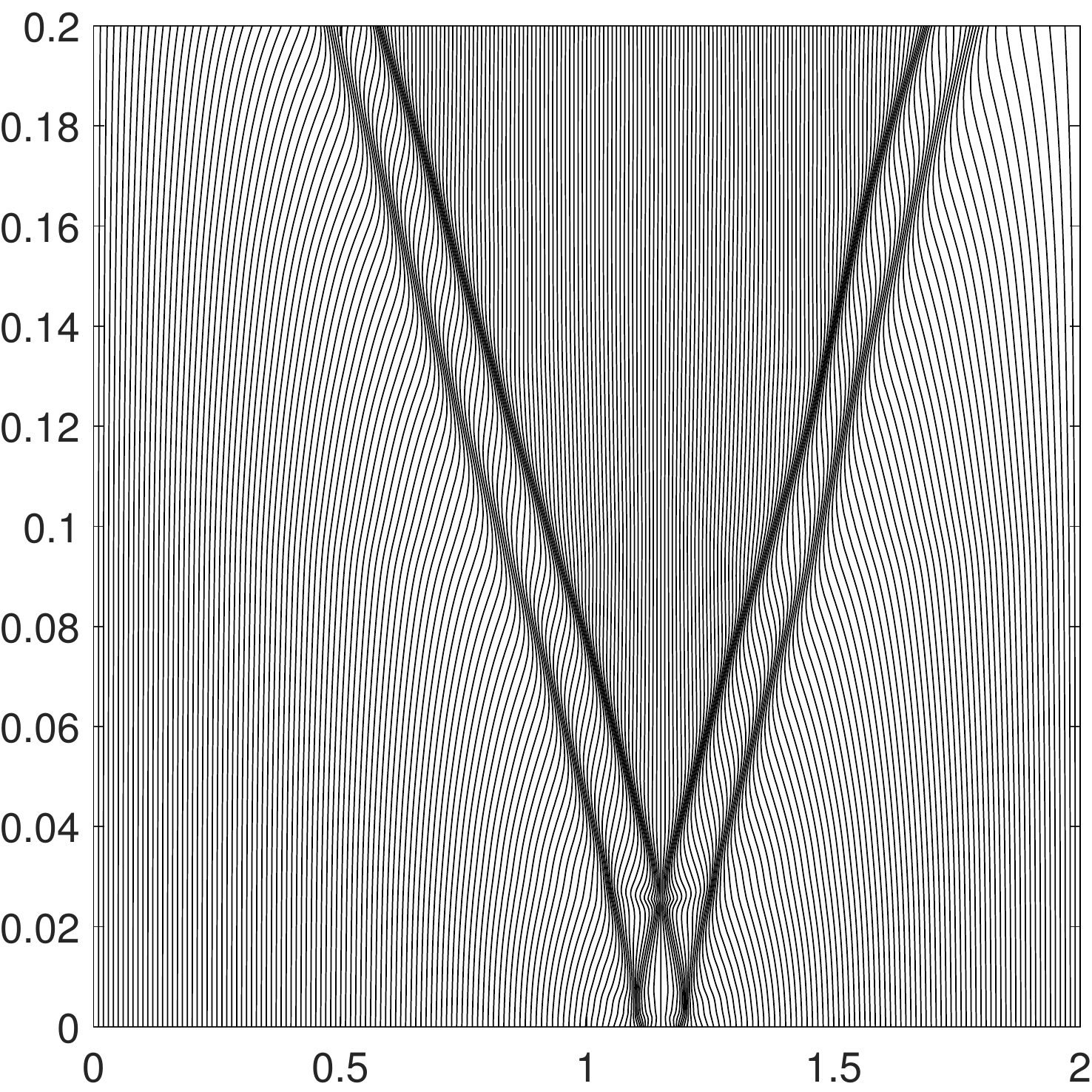}
    \caption{mesh trajectory, $\epsilon = 0.001$}
  \end{subfigure}
  \caption{Example \ref{ex:1D_Pertubation_Test}. The numerical solutions obtained by using the {\tt UM-ES} and {\tt MM-ES} schemes at $t=0.2$.
  The reference solution is obtained by using the {\tt UM-ES} scheme with $3000$ mesh points.}\label{1D_ES_Pertubation_WB}
\end{figure}

\begin{figure}[hbt!]
  \centering
  \begin{subfigure}[b]{0.35\textwidth}
    \centering
    \includegraphics[width=1.0\linewidth]{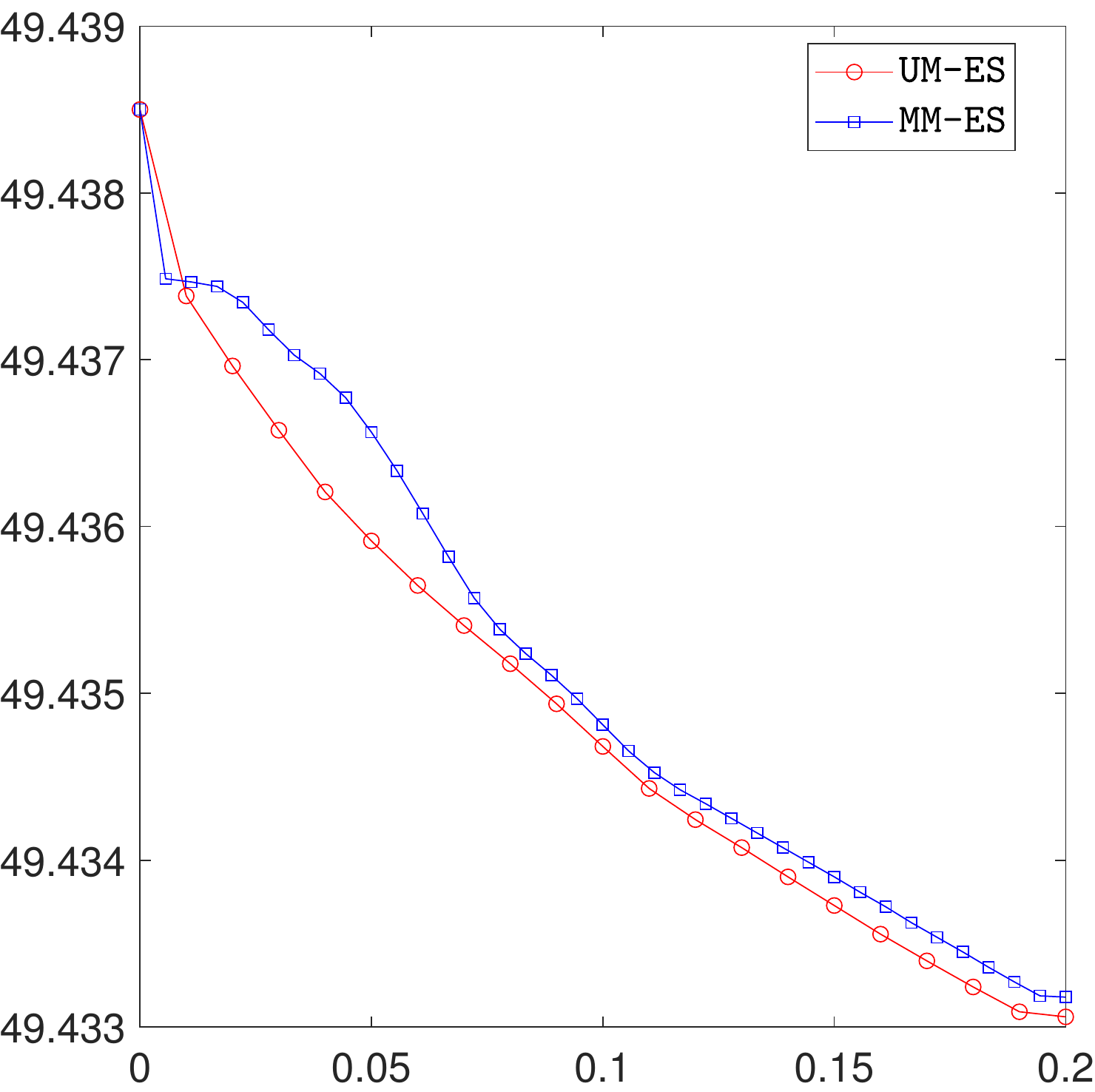}
    \caption{$\epsilon = 0.2$}
  \end{subfigure}
  \begin{subfigure}[b]{0.35\textwidth}
    \centering
    \includegraphics[width=1.0\linewidth]{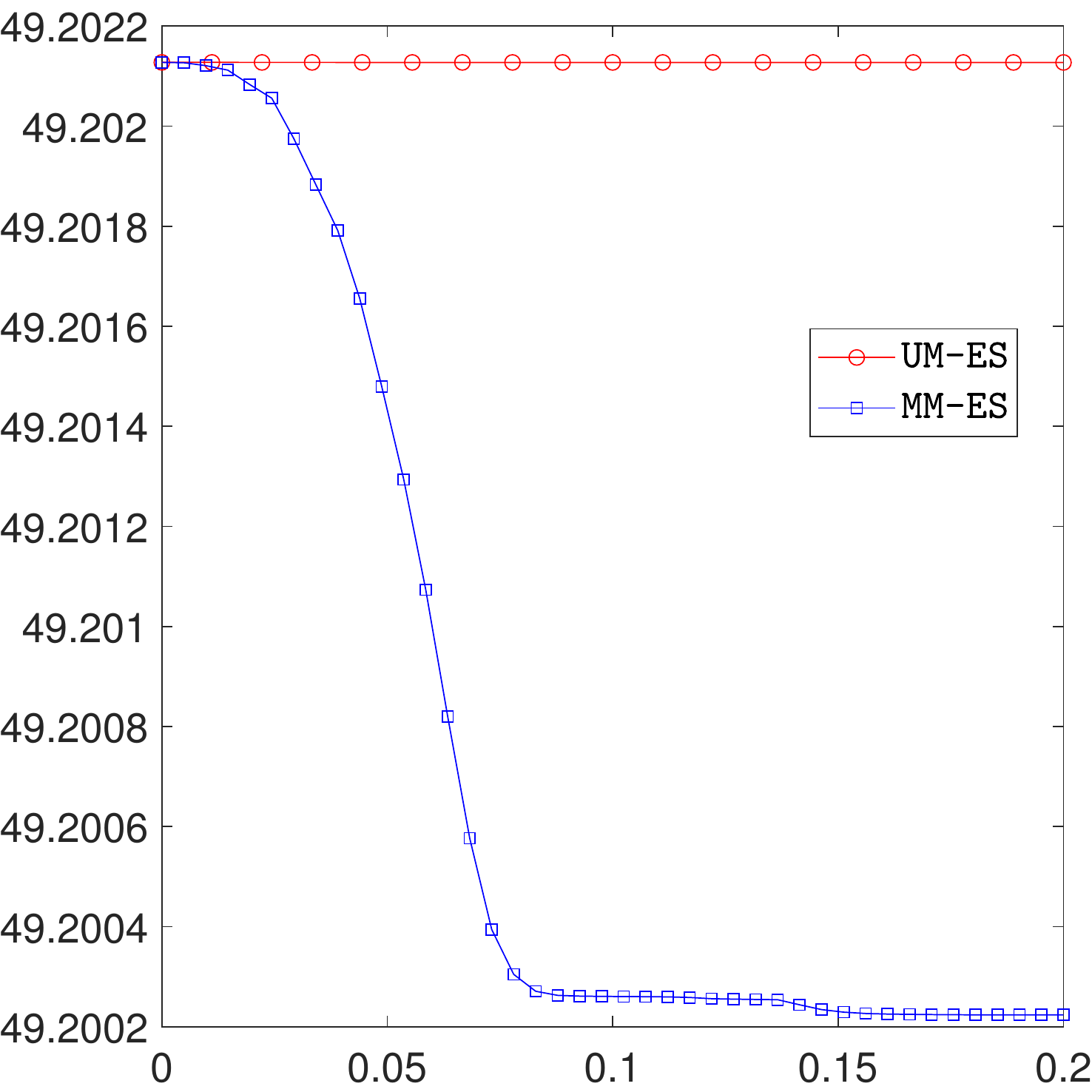}
    \caption{$\epsilon = 0.001$}
  \end{subfigure}
  \caption{Example \ref{ex:1D_Pertubation_Test}. The evolution of the discrete total energy over time by using the {\tt UM-ES} and {\tt MM-ES} schemes with $1000$ mesh points.}\label{1D_ES_Pertubation_WB_Entropy}
\end{figure}

\begin{figure}[hbt!]
  \centering
  \begin{subfigure}[b]{0.35\textwidth}
    \centering
    \includegraphics[width=1.0\linewidth, clip]{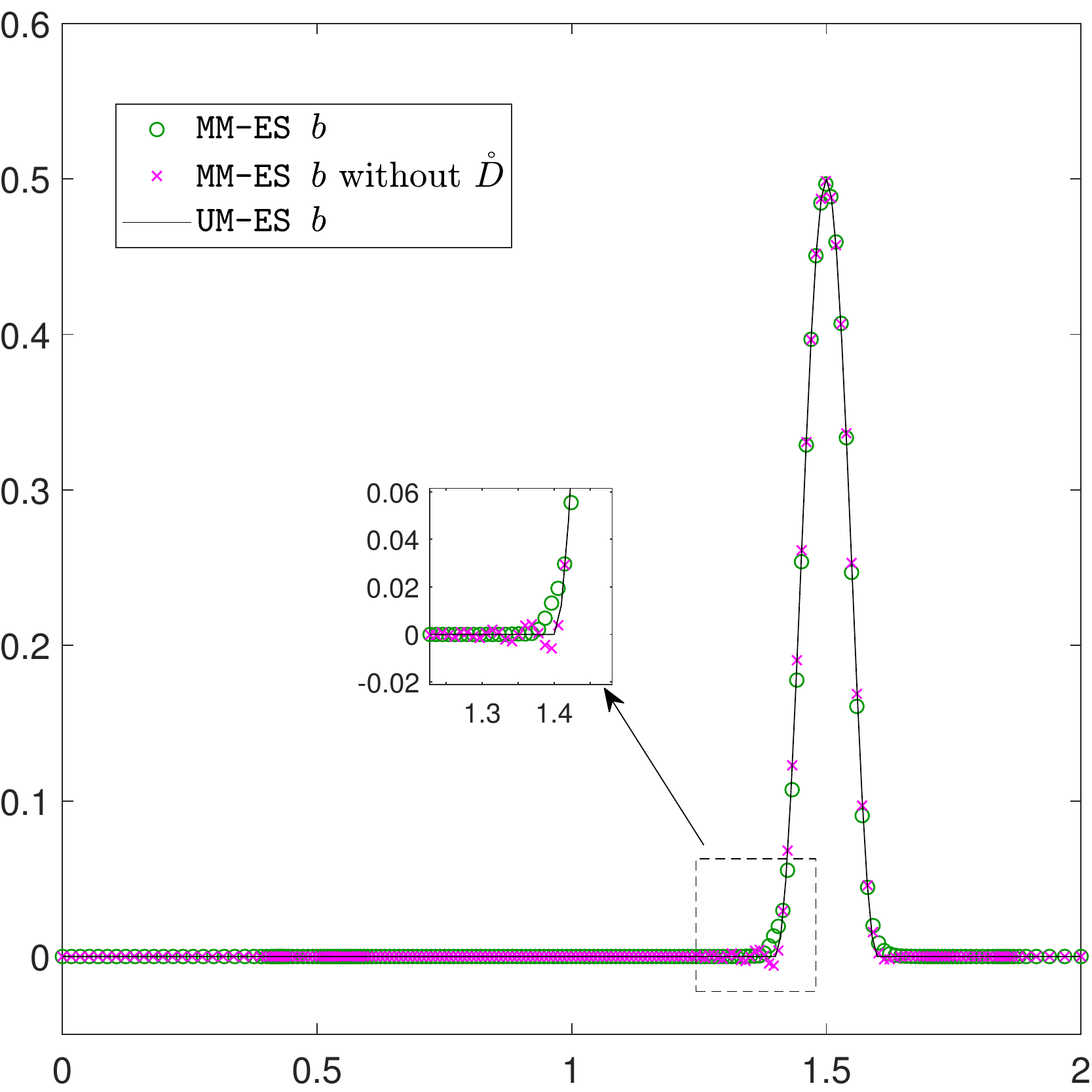}
  \end{subfigure}
  \begin{subfigure}[b]{0.35\textwidth}
    \centering
    \includegraphics[width=1.0\linewidth]{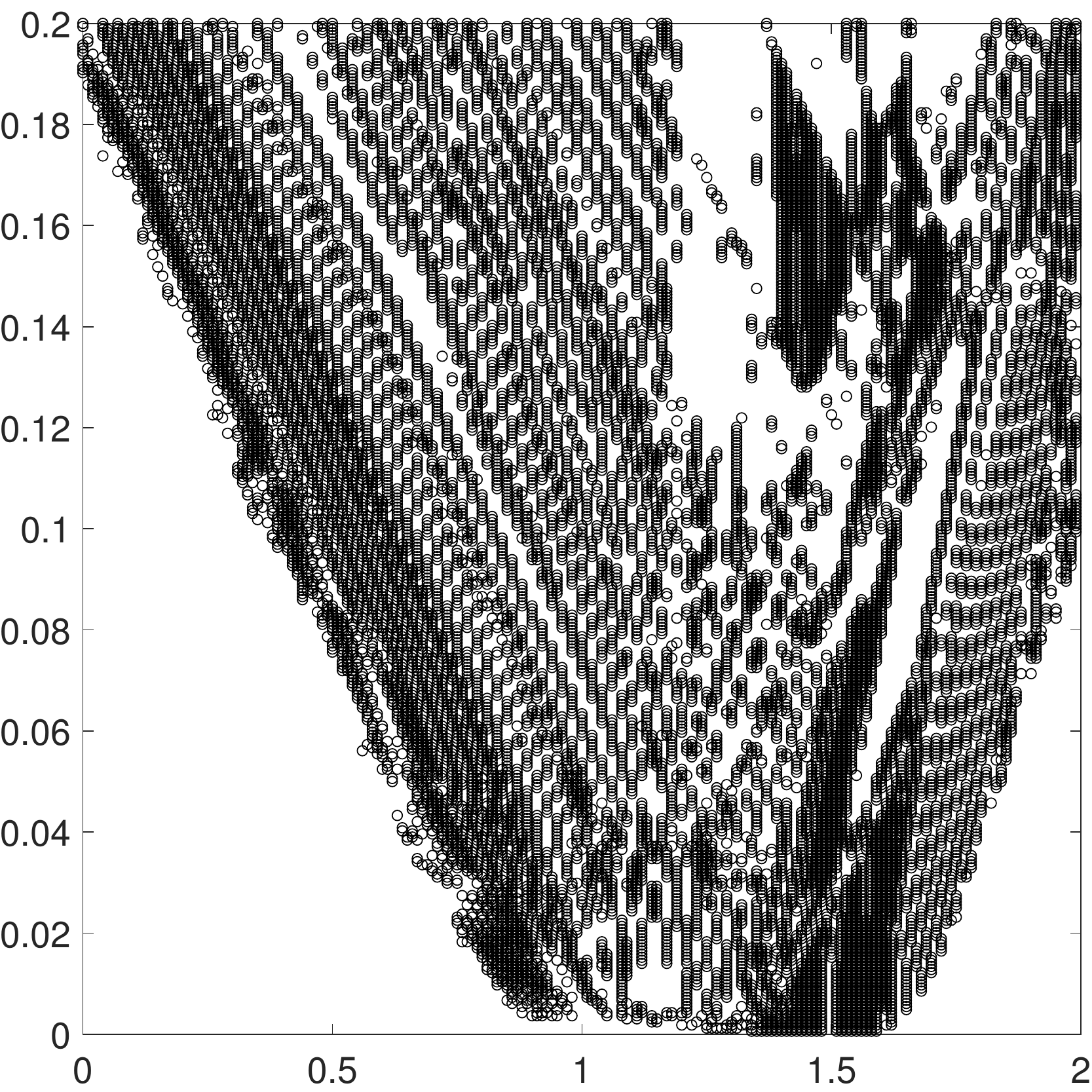}
  \end{subfigure}
  \caption{Example \ref{ex:1D_Pertubation_Test}.
    Left: the bottom topography $b$ obtained by using the {\tt UM-ES}, {\tt MM-ES},
    and {\tt MM-ES} schemes without $\bm{\mathring{D}}$ with $200$ mesh points.
    Right: the dots denote where the second  dissipation term $\bm{\mathring{D}}$ for $b$ is added in the $x-t$ coordinate.
  }
  \label{1D_ES_Pertubation_WB_bottom}
\end{figure}

\subsection{2D tests}
\begin{example}[Accuracy test with moving vortex]\label{eq:Smooth_2D}\rm
  This problem is taken to verify the accuracy of our 2D schemes, utilizing the 2D vortex example in \cite{Duan2021SWMHD} without the magnetic fields.
  A steady vortex can be determined by
  \begin{equation*}
    \begin{aligned}
      &h^{\prime}=h_{\max}-v_{\max }^2 e^{1-r^2} /(2 g), \\
      &\left(v_1^{\prime}, v_2^{\prime}\right)=v_{\max} e^{0.5\left(1-r^2\right)}(-x_2, x_1), \\
    \end{aligned}
  \end{equation*}
  with $h_{\max} = 1,~v_{\max} = 0.2,~r = \sqrt{x_1^2+x_2^2}, ~g=1$.
  Then a moving vortex with constant velocity $(1,1)$ can be obtained by using the Galilean transformation
  \begin{equation*}
    \begin{aligned}
      & h(x_1, x_2, t)=h^{\prime}(x_1-t, x_2-t, t),~\left(v_1, v_2\right)(x_1, x_2, t)=(1,1)+\left(v_1^{\prime}, v_2^{\prime}\right)(x_1-t, x_2-t, t). \\
    \end{aligned}
  \end{equation*}
  The physical domain is $[-10, 10] \times [-10, 10]$ with the periodic boundary conditions, and the output time is $t = 2, 4 $. The monitor function is set to be
  \begin{equation*}
    \omega=\left({1+\frac{15 |\nabla_{\bm{\xi}} (h+b)|}{\max |\nabla_{\bm{\xi}} (h+b)|}+\frac{10 |\Delta_{\bm{\xi}} (h+b)|}{\max |\Delta_{\bm{\xi}} (h+b)|}}\right)^{1/2}.
  \end{equation*}
\end{example}

Figure \ref{2D test} shows the errors and convergence orders in the water depth $h$ at $t =2,4$,
which verify the $5$th-order accuracy.
Figure \ref{fig:Mesh_Entropy} gives the adaptive meshes at $t=0,2,4$ with $5$ equally spaced contours of $h+b$ and the evolution of the discrete total energy $\sum\limits_{i,j} J_{i,j}\eta(\bm{U}_{i,j})\Delta \xi_{1} \Delta {\xi_2}$ obtained by using the {\tt MM-ES} scheme with $40\times40$ mesh.
One can see that the concentration of the mesh points adaptively moves with the vortex,
and the {\tt MM-EC} scheme can keep the discrete total energy almost constant, while the {\tt MM-ES} scheme makes the discrete total energy decay.

\begin{figure}[hbt!]
  \centering
  \includegraphics[width=0.4\textwidth]{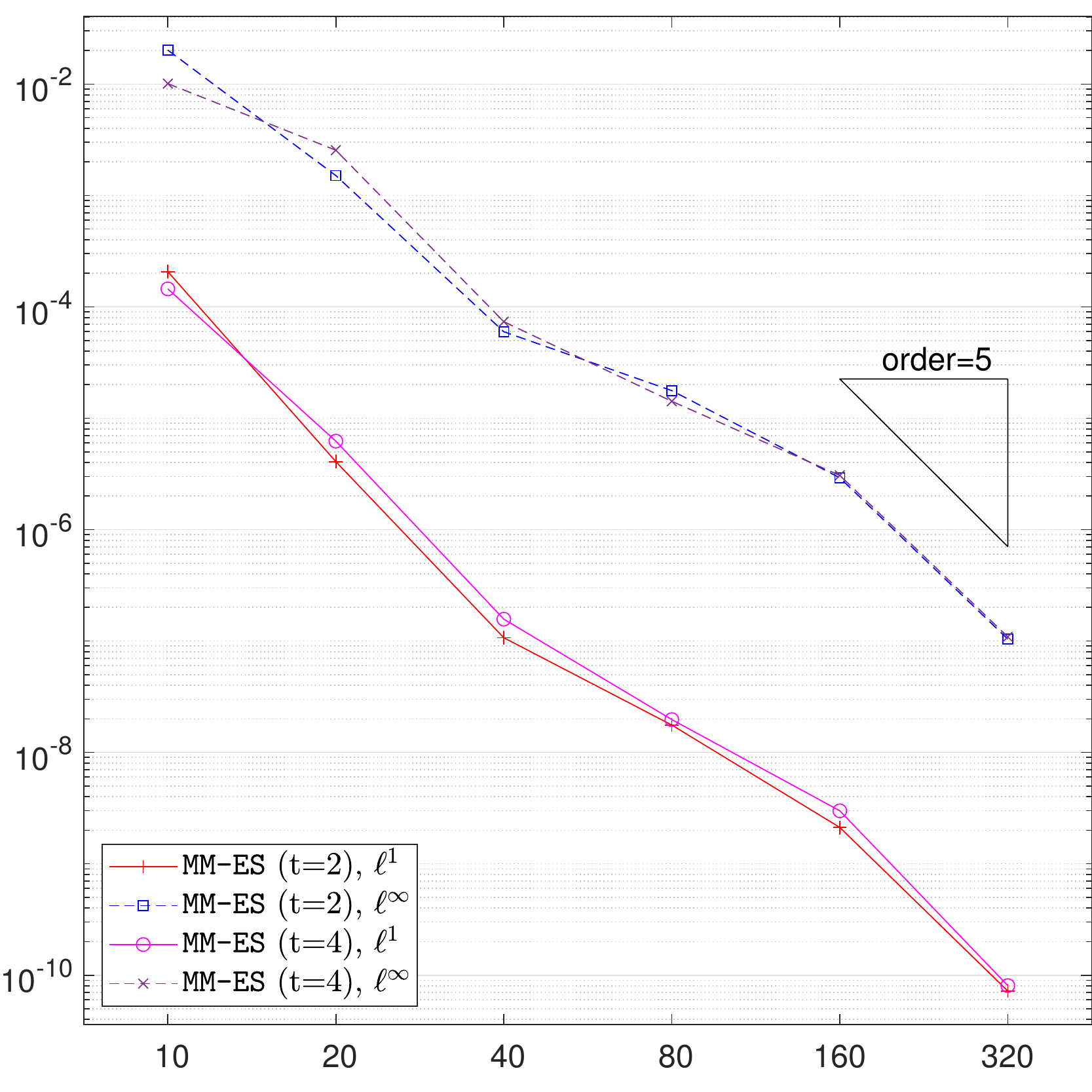}
  \caption{Example \ref{eq:Smooth_2D}. The errors and convergence orders in the water depth $h$ at $t =2$, $4$ obtained by the {\tt MM-ES} scheme.
  }
  \label{2D test}
\end{figure}

\begin{figure}[hbt!]
  \centering
  \begin{subfigure}[b]{0.35\textwidth}
    \centering
    \includegraphics[width=1.0\linewidth]{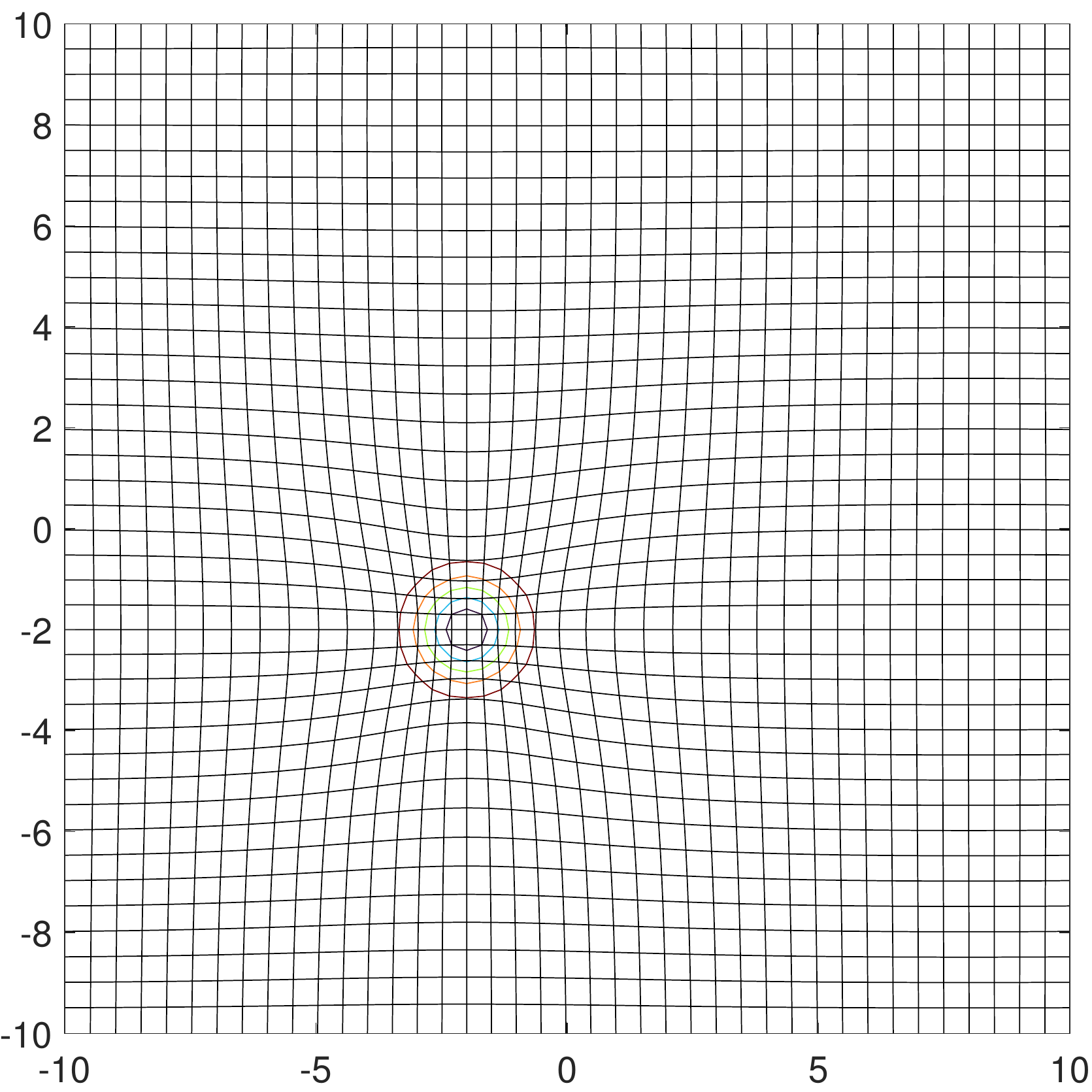}
    \caption{$t=0$}
  \end{subfigure}
  \begin{subfigure}[b]{0.35\textwidth}
    \centering
    \includegraphics[width=1.0\linewidth]{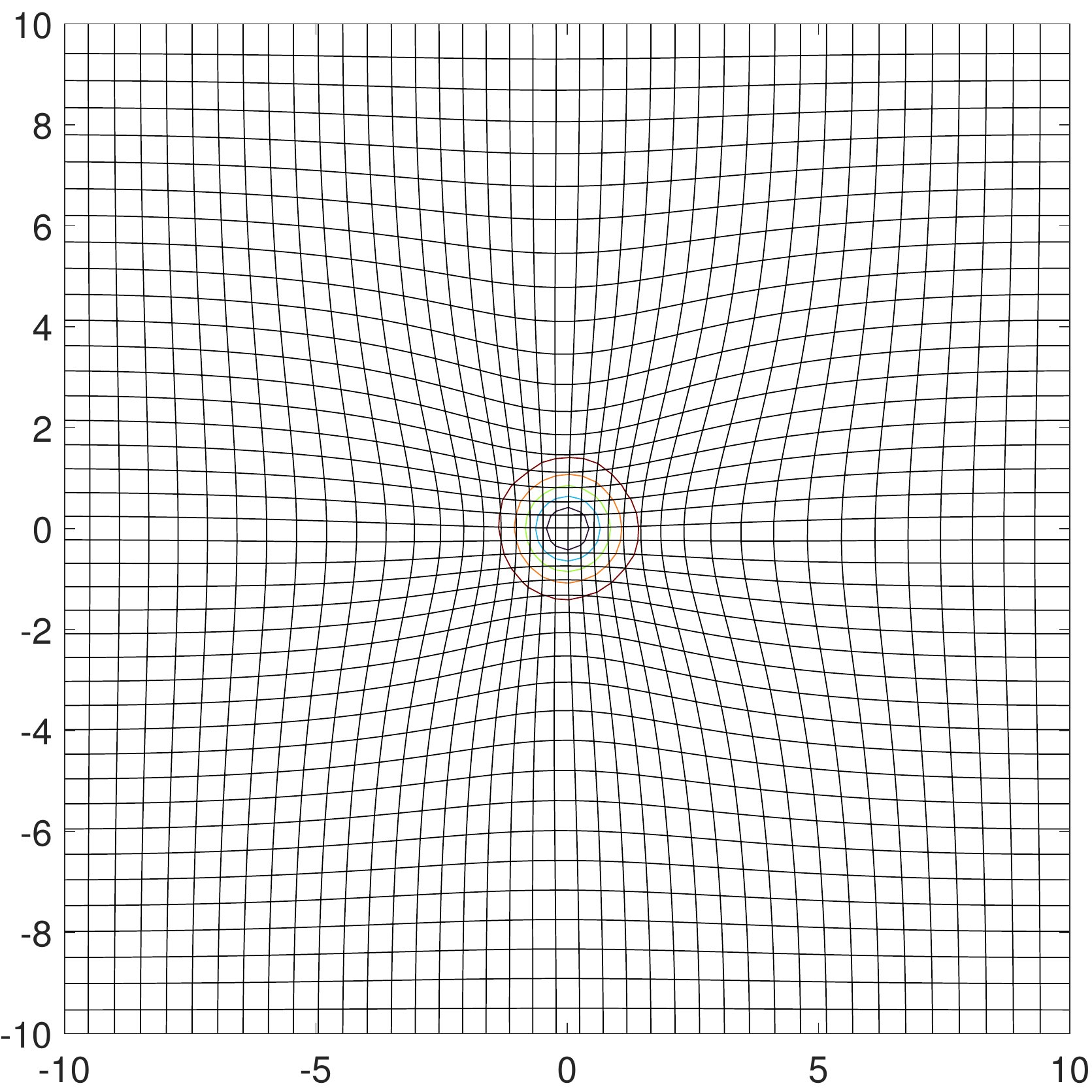}
    \caption{$t=2$}
  \end{subfigure}\\
  \begin{subfigure}[b]{0.35\textwidth}
    \centering
    \includegraphics[width=1.0\linewidth]{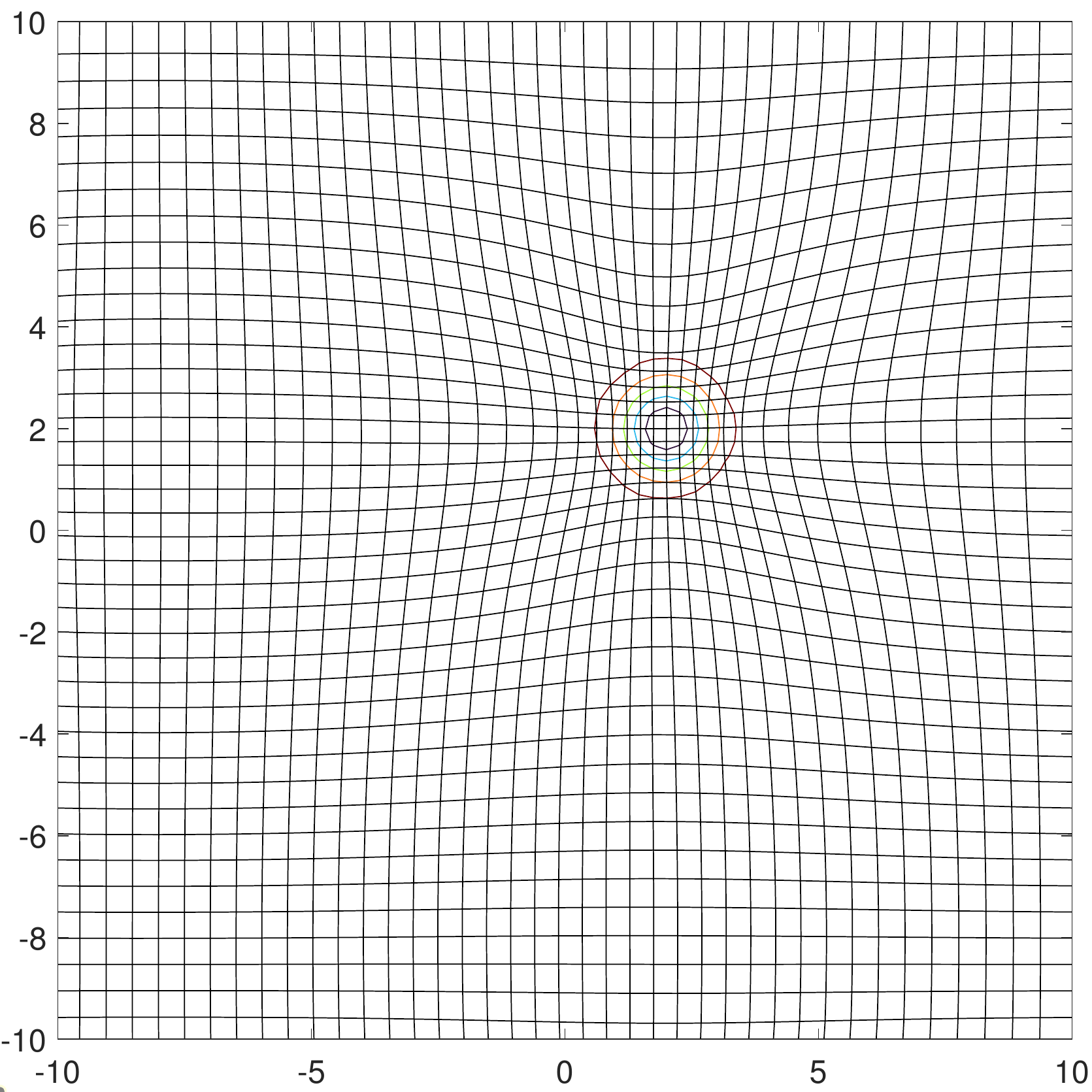}
    \caption{$t=4$}
  \end{subfigure}
  \begin{subfigure}[b]{0.35\textwidth}
    \centering
    \includegraphics[width=1.0\linewidth]{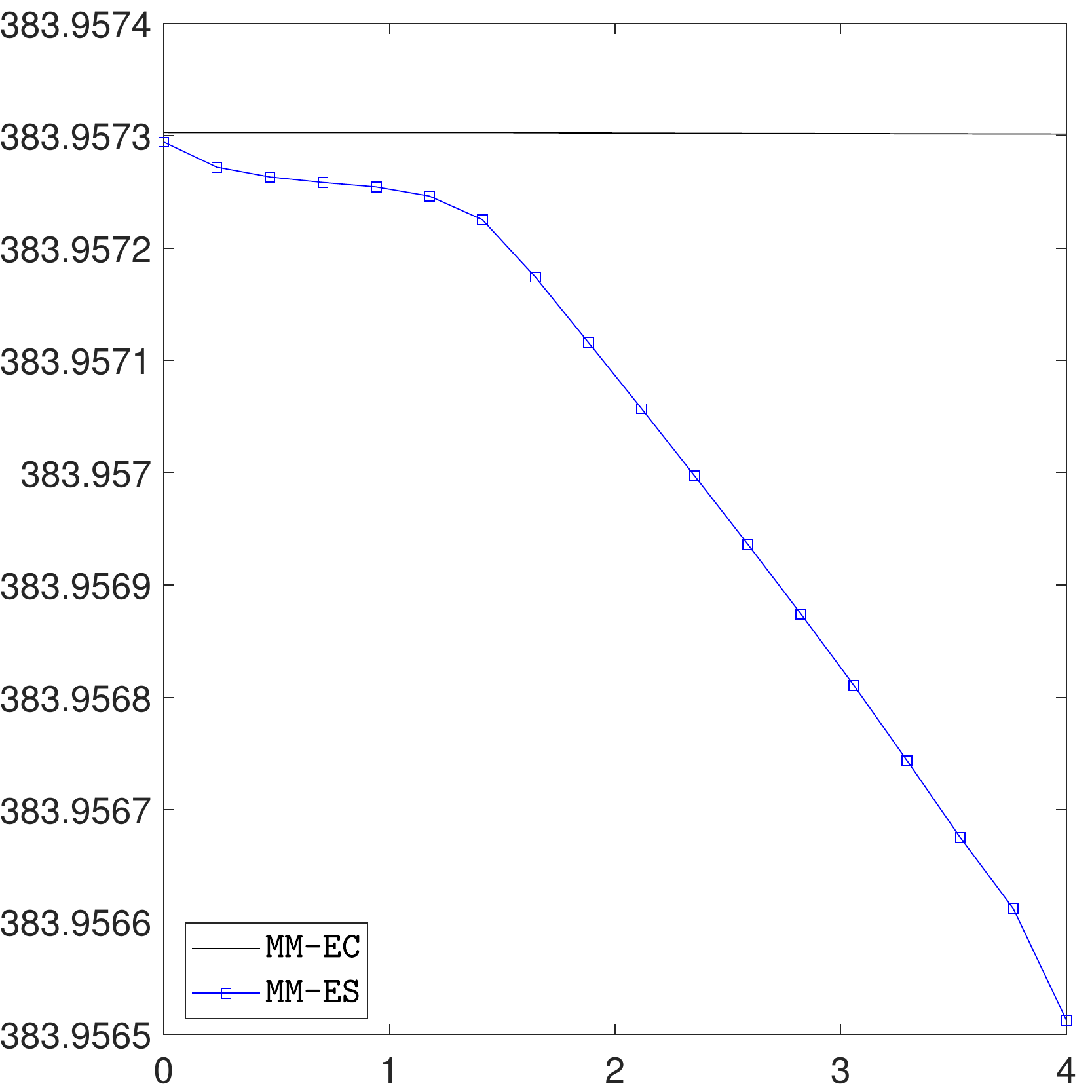}
    \caption{The discrete total energy}
  \end{subfigure}
  \caption{Example \ref{eq:Smooth_2D}.
    The adaptive meshes and 5 equally spaced contour lines of the water surface level $h+b$ at different times by using the {\tt MM-ES} scheme with $40\times40$ mesh.
  And the discrete total energy evolution in time.}\label{fig:Mesh_Entropy}
\end{figure}

%

\begin{example}[2D WB test]\label{ex:2D_WB_Test}\rm
  This example is utilized to verify the WB properties of our 2D {\tt UM-ES} and {\tt MM-ES} schemes.
  The bottom topography is
  \begin{equation}\label{eq:2D_b_Smooth}
    b(x_1, x_2) = 0.8 \exp\left(-50\left((x_1-0.5)^2+(x_2-0.5)^2\right)\right),~ (x_1,x_2) \in [0,1]\times[0,1],
  \end{equation}
  or
  \begin{equation}\label{eq:2D_b_dis}
    b(x_1, x_2) = \begin{cases}
      0.5 , &\text{if}\quad  (x_1,x_2) \in [0.3,0.5]\times[0.3,0.5],\\
      0, &\text{otherwise},
    \end{cases}
  \end{equation}
  with the initial water depth $ h = 1 - b$,
  zero velocities, and $g=1$ with the outflow boundary conditions.
  The monitor function is
  \begin{equation*}
    \omega=\sqrt{1+\theta\frac{\left|\nabla_{\bm{\xi}} \sigma\right|}{ \max \left|\nabla_{\bm{\xi}} \sigma\right|}},
  \end{equation*}
  where $\sigma$ is the water depth $h$ and $\theta = 100$.
\end{example}

Table \ref{tb:2D_Well_Balance} lists the $\ell^1,\ell^{\infty}$ errors in $h+b$ and $v_1$ obtained by using our schemes with $100\times100$ meshes at $t=0.1$,
which are clearly at the level of rounding error in double precision.
Figure \ref{fig:2D_Well_Balance} shows the results of the water surface level $h+b$ and bottom topography $b$ obtained by using the {\tt MM-ES} scheme with $100\times100$ mesh.
The results illustrate that our schemes are WB on the adaptive moving mesh.
Figure \ref{fig:2D_WB_Scatter} plots the location where the second dissipation term ${\mathring{\bm{D}}}$ is added for the bottom topography \eqref{eq:2D_b_dis},
indicating it is almost added where $b$ is discontinuous.

\begin{table}[hbt!]
  \centering
  \begin{tabular}{c|c|cc|cc|cc|cc}
    \toprule
    \multicolumn{2}{c|}{\multirow{2}{*}{}} & \multicolumn{2}{c|}{{\tt UM-EC}} & \multicolumn{2}{c|}{{\tt UM-ES}} & \multicolumn{2}{c|}{{\tt MM-ES}} & \multicolumn{2}{c}{{\tt MM-ES} without $\bm{\mathring{D}}$ } \\ 
    \cline{3-10}
    \multicolumn{2}{c|}{} & \multicolumn{1}{c}{$\ell^{1}$~error}  & \multicolumn{1}{c|}{$\ell^{\infty}$~error} &  \multicolumn{1}{c}{$\ell^{1}$~error}  & \multicolumn{1}{c|}{$\ell^{\infty}$~error} &  \multicolumn{1}{c}{$\ell^{1}$~error}  & \multicolumn{1}{c|}{$\ell^{\infty}$~error} &  \multicolumn{1}{c}{$\ell^{1}$~error}  & \multicolumn{1}{c}{$\ell^{\infty}$~error}   \\
    \hline
    \multirow{2}{*}{$b$ in \eqref{eq:2D_b_Smooth}} &
    $h+b$ & 3.20e-16 	 & 7.77e-15 	 & 1.06e-16 	 & 6.66e-16 	& 3.20e-16 	 & 1.67e-15& 3.40e-16 	 & 1.89e-15 \\ 
    &$v_1$ & 4.01e-16 	 & 7.12e-15 	 & 1.42e-16 	 & 1.29e-15 	 & 2.46e-16 	 & 2.17e-15& 2.62e-16 	 & 1.76e-15\\ 
    \hline
    \multirow{2}{*}{$b$ in \eqref{eq:2D_b_dis}} &
    $h+b$ & 9.79e-18 	 & 1.11e-15 	 & 3.38e-18 	 & 4.44e-16 	& 3.12e-16 	 & 1.55e-15& 2.84e-16 	 & 1.33e-15 \\ 
    &$v_1$  & 1.72e-17 	 & 1.31e-15 	 & 7.58e-18 	 & 7.53e-16 	 & 2.20e-16 	 & 1.59e-15& 2.10e-16 	 & 1.23e-15\\ 
    \bottomrule
  \end{tabular}
  \caption{Example \ref{ex:2D_WB_Test}. Errors in $h+b$ and $v_1$ by using our schemes with $100\times100$ meshes at $t=0.1$,
  for the bottom topography \eqref{eq:2D_b_Smooth} and \eqref{eq:2D_b_dis}.}\label{tb:2D_Well_Balance}
\end{table}
\begin{figure}[hbt!]
  \centering
  \begin{subfigure}[b]{0.39\textwidth}
    \centering
    \includegraphics[width=1.0\linewidth]{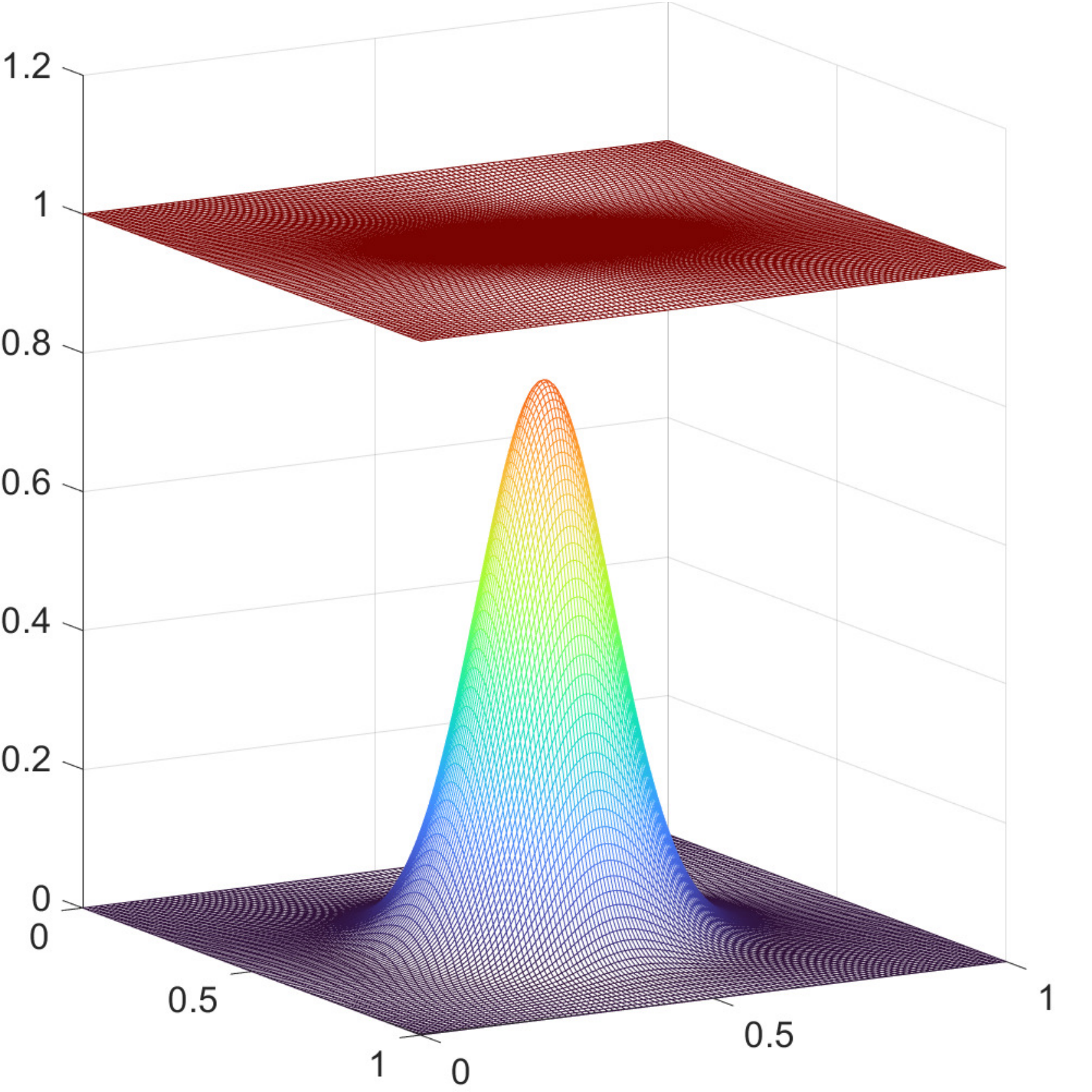}
  \end{subfigure}
  \begin{subfigure}[b]{0.35\textwidth}
    \centering
    \includegraphics[width=1.0\linewidth]{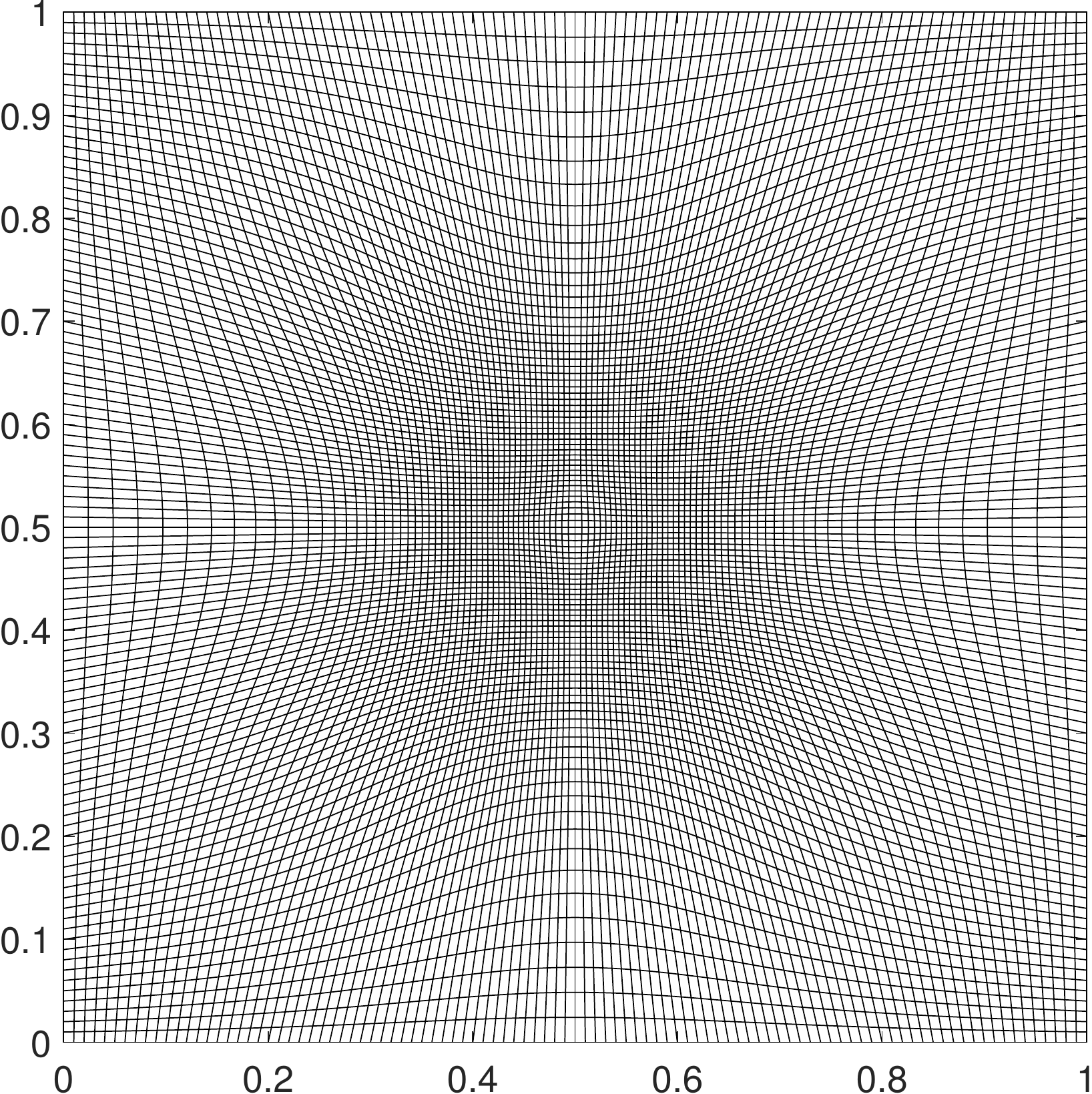}
  \end{subfigure}
  \\
  \begin{subfigure}[b]{0.39\textwidth}
    \centering
    \includegraphics[width=1.0\linewidth]{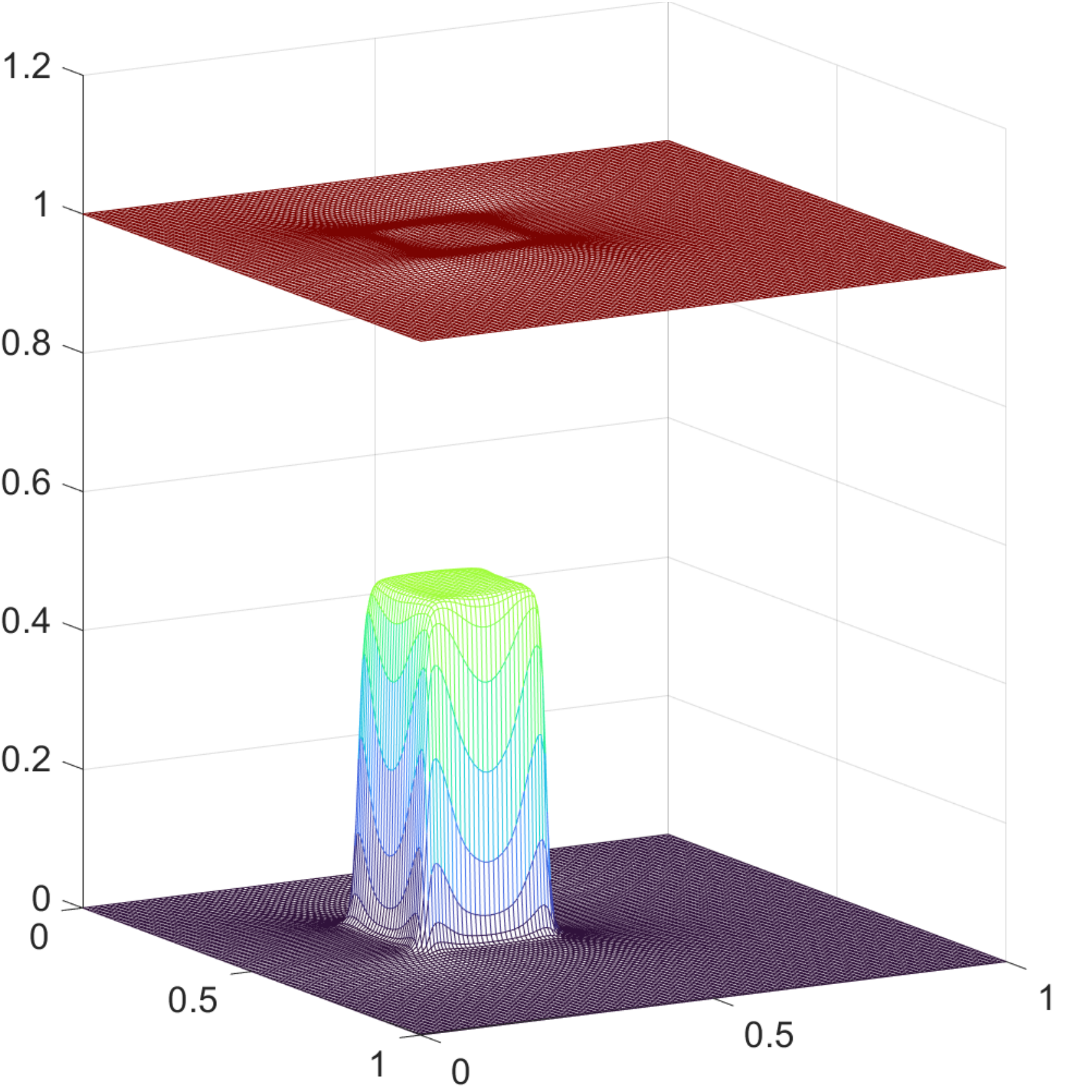}
  \end{subfigure}
  \begin{subfigure}[b]{0.35\textwidth}
    \centering
    \includegraphics[width=1.0\linewidth]{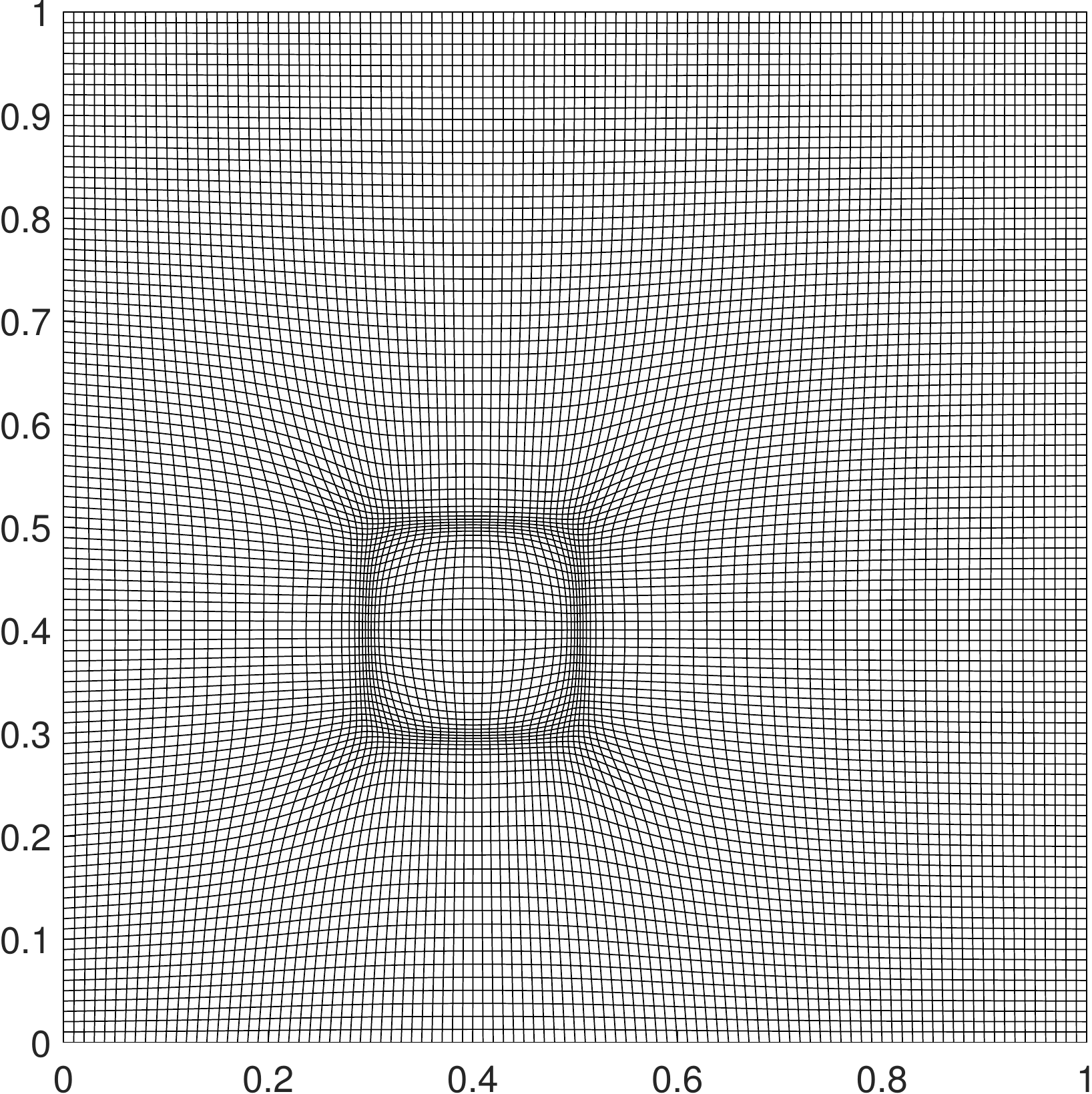}
  \end{subfigure}
  \caption{Example \ref{ex:2D_WB_Test}. The water surface level $h+b$, bottom topography $b$, and adaptive meshes obtained by using the {\tt MM-ES} scheme with $100 \times 100$ mesh at $t=0.1$.
  The figures in the first and second rows correspond to the cases with the bottom topography \eqref{eq:2D_b_Smooth} and \eqref{eq:2D_b_dis}, respectively.}
  \label{fig:2D_Well_Balance}
\end{figure}

\begin{figure}[hbt!]
  \centering
  \begin{subfigure}[b]{0.35\textwidth}
    \centering
    \includegraphics[width=1.0\linewidth]{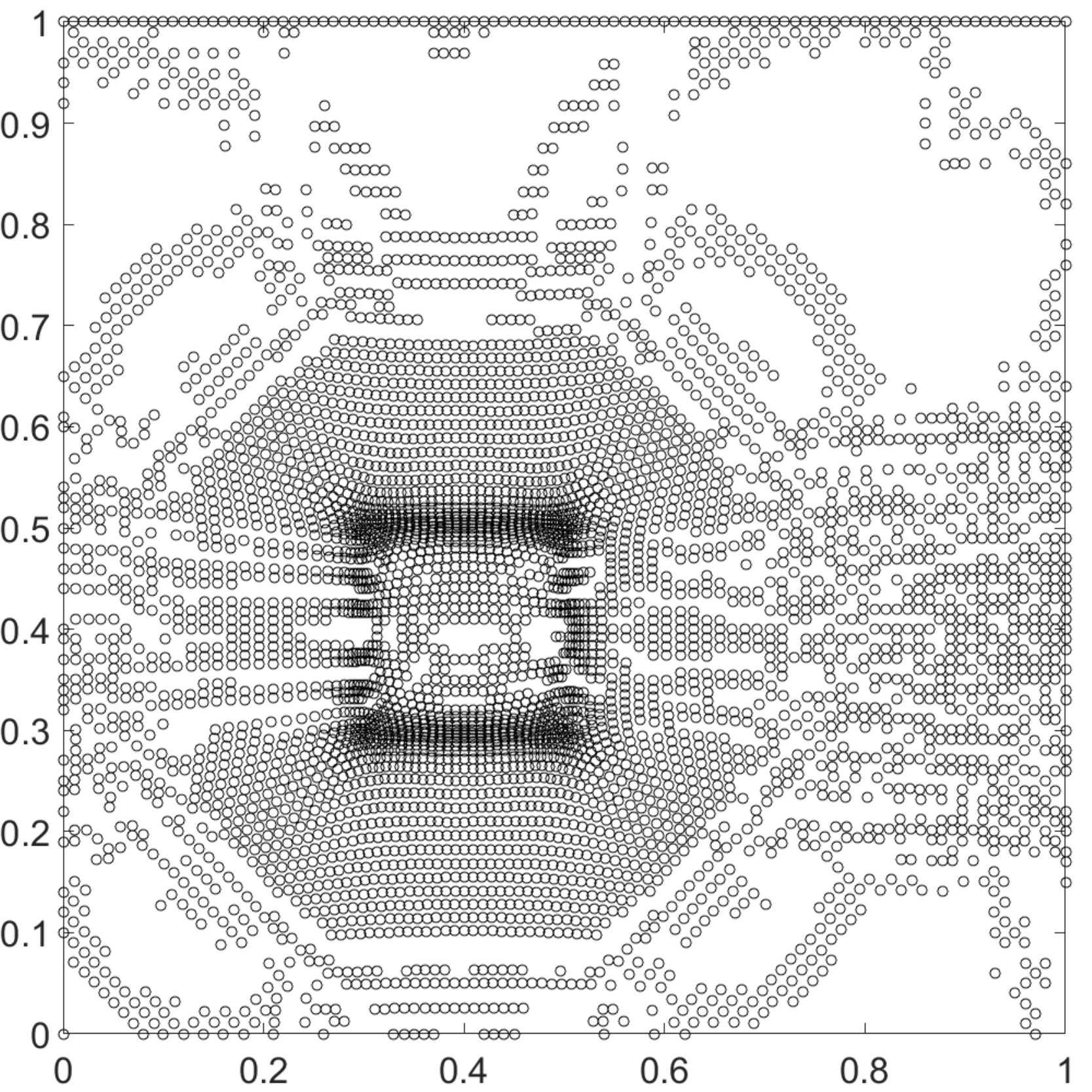}
  \end{subfigure}
  \begin{subfigure}[b]{0.35\textwidth}
    \centering
    \includegraphics[width=1.0\linewidth]{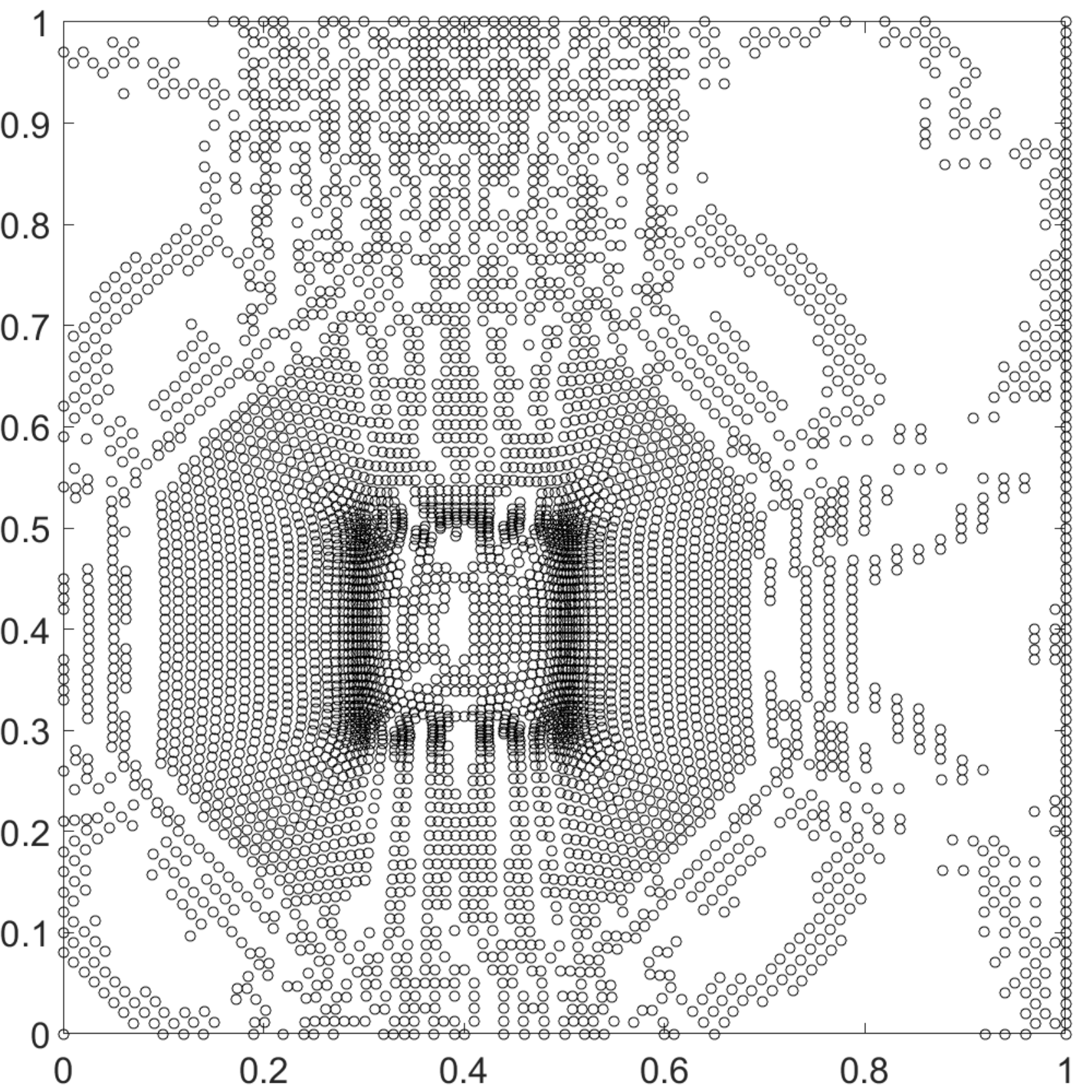}
  \end{subfigure}
  \caption{Example \ref{ex:2D_WB_Test}. The location where the second dissipation term for $b$ is added with the bottom topography \eqref{eq:2D_b_dis} at t = 0.1.
    Left: $\xi_1$-direction, right: $\xi_2$-direction.
  }
  \label{fig:2D_WB_Scatter}
\end{figure}

\begin{example}[The perturbed flow in lake at rest]\label{ex:Pertubation}\rm
  This example, in which the bottom topography resembles an ``oval hump'' \cite{Xing2017Numerical,Xing2005High}, is used to test the ability of our schemes to capture small perturbations over the lake at rest in the domain $[0,2] \times[0,1]$ with the outflow boundary conditions.
  The initial data are
  \begin{align*}
    &h= \begin{cases}1.01-b,~&\text{if}\quad x_1 \in[0.05, 0.15], \\
      1-b,~&\text{otherwise},
    \end{cases} \\
    & v_1=v_2=0,\\
    &b=0.8 \exp(-5(x_1-0.9)^2-50(x_2-0.5)^2).
  \end{align*}
  The initial disturbance will split into two waves propagating at the speeds of $\pm \sqrt{gh},~g=9.812$.
  The monitor function is selected as the one defined in Example \ref{ex:2D_WB_Test},
  except that $\sigma = h+b$ and $\theta=800$.
\end{example}

Figure \ref{fig:2D_MM_Mesh_Value} shows the $40$ equally spaced contour lines of $h+b$ at $t = 0.12$, $0.24$, $0.36$, $0.48$, $0.6$, obtained by using the {\tt MM-ES} scheme with $300\times150$ mesh,
and the plots of the adaptive meshes at different times.
Our scheme can capture complex small features,
and the adaptive moving mesh points concentrate near those features to increase the resolution.
In Figure \ref{2D_Case_compare},
the results are also compared to those by using the {\tt UM-ES} scheme with $300\times150$ and $900\times450$ meshes.
The contour ranges are taken from the ranges of the results obtained by using the {\tt UM-ES} scheme with $900\times450$ mesh,
specifically, $[0.999750, 1.006230]$, $[0.994453, 1.017051]$, $[0.986720, 1.012746]$, $[0.989904, 1.
005154]$, $[0.995003,1.006070]$ at $t = 0.12$, $0.24$, $0.36$, $0.48$, $0.6$, respectively.
It shows that when the numbers of the mesh points are the same, the results obtained by the {\tt MM-ES} scheme are better than the {\tt UM-ES} scheme,
and comparable to those by {\tt UM-ES} with finer mesh.
From Table \ref{tb:Time_Compare},
one can see that when similar resolutions are achieved at $t=0.6$,
the {\tt MM-ES} scheme takes only $14.8\%$ CPU time of that using the {\tt UM-ES} scheme with finer mesh,
which highlights the high efficiency of our adaptive moving mesh method.
The CPU times are measured on a laptop with Intel{\textregistered} Core{\texttrademark} i7-8750H CPU @2.20GHz, 24GB memory, and the code is programmed based on MATLAB R2021b.
\begin{table}[hbt!]
  \centering
  \begin{tabular}{c|r|r|r|r|r}
    \toprule
    & $t = 0.12$ & $t = 0.24$ & $t = 0.36$ & $t = 0.48$ & $t = 0.6$ \\ \hline
    {\tt UM-ES}~($300\times150$) & 2m10s   & 6m25s    & 9m36s    & 12m20s    & 15m28s   \\
    {\tt MM-ES}~($300\times150$) & 14m42s    & 24m07s   & 32m36s   & 40m51s  & 50m11s      \\
    {\tt UM-ES}~($900\times450$) & 1h09m   & 2h38m  & 3h32m  & 4h27m  & 5h36m \\ 
    \bottomrule
  \end{tabular}
  \caption{Example \ref{ex:2D_WB_Test}. The CPU times at different output times.}\label{tb:Time_Compare}
\end{table}

\begin{figure}[hbt!]
  \centering
  \begin{subfigure}[b]{0.4\textwidth}
    \centering
    \includegraphics[width=1.0\linewidth]{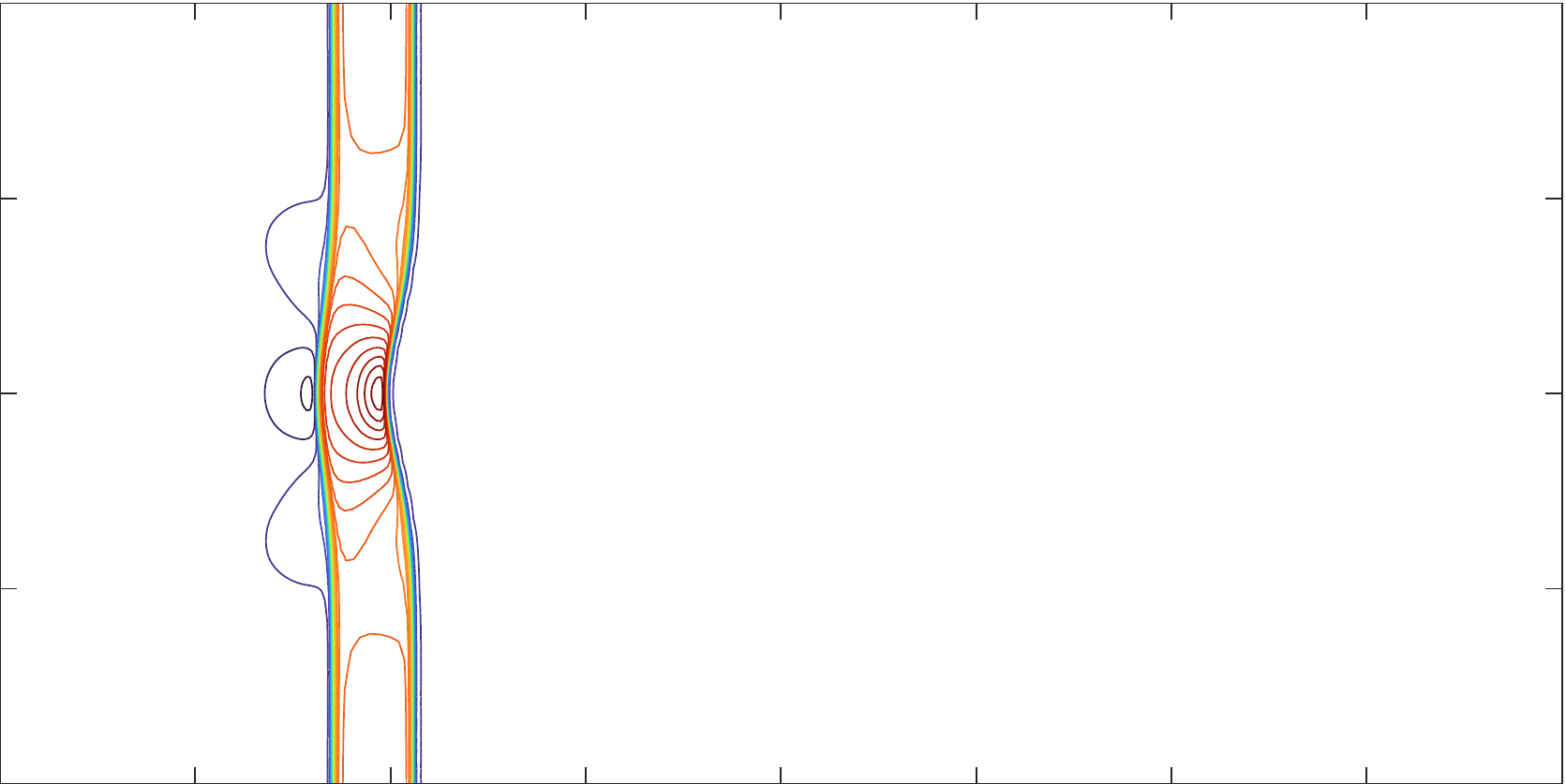}
  \end{subfigure}~
  \begin{subfigure}[b]{0.4\textwidth}
    \centering
    \includegraphics[width=1.0\linewidth]{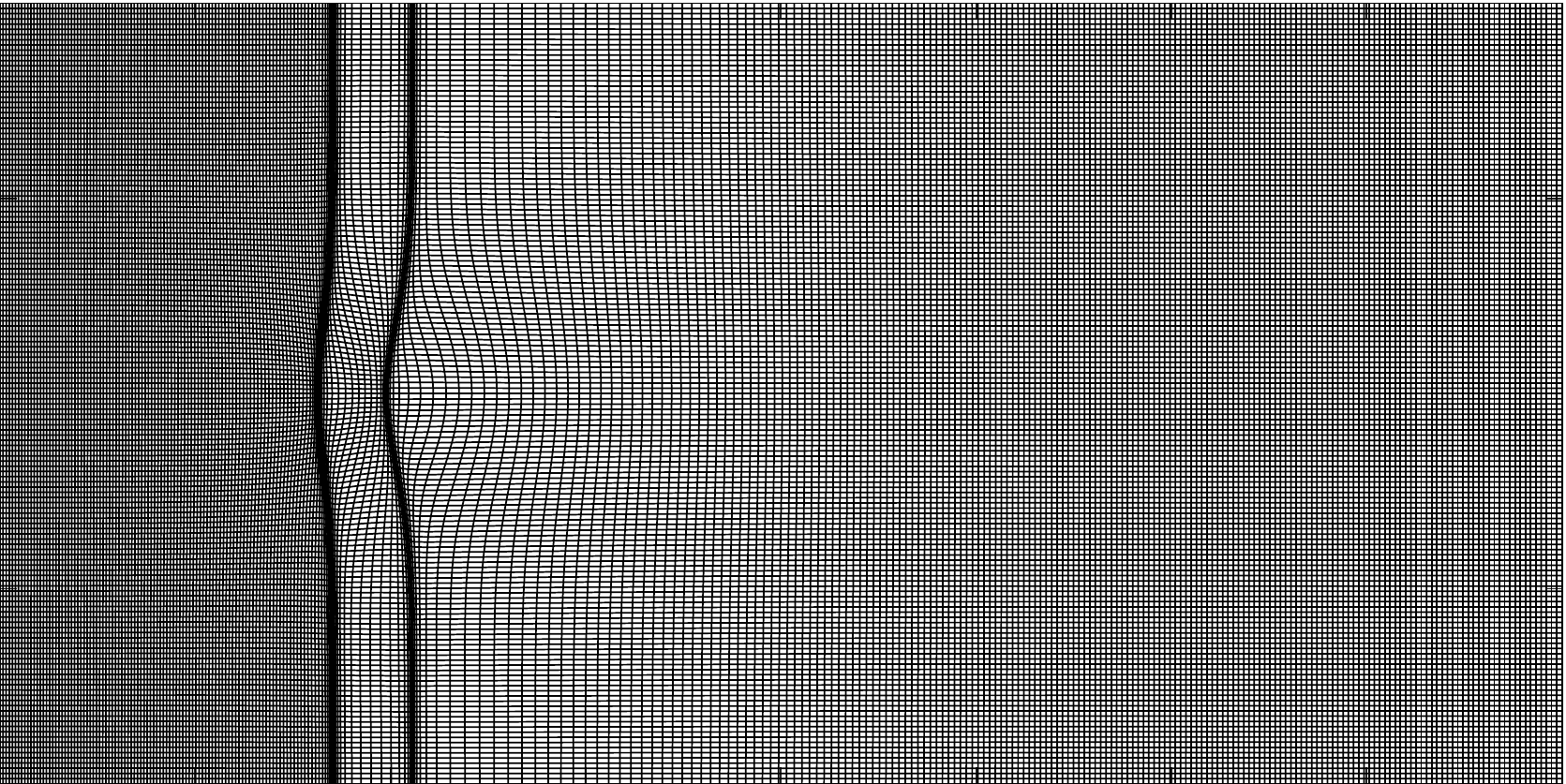}
  \end{subfigure}
  \\
  \begin{subfigure}[b]{0.4\textwidth}
    \centering
    \includegraphics[width=1.0\linewidth]{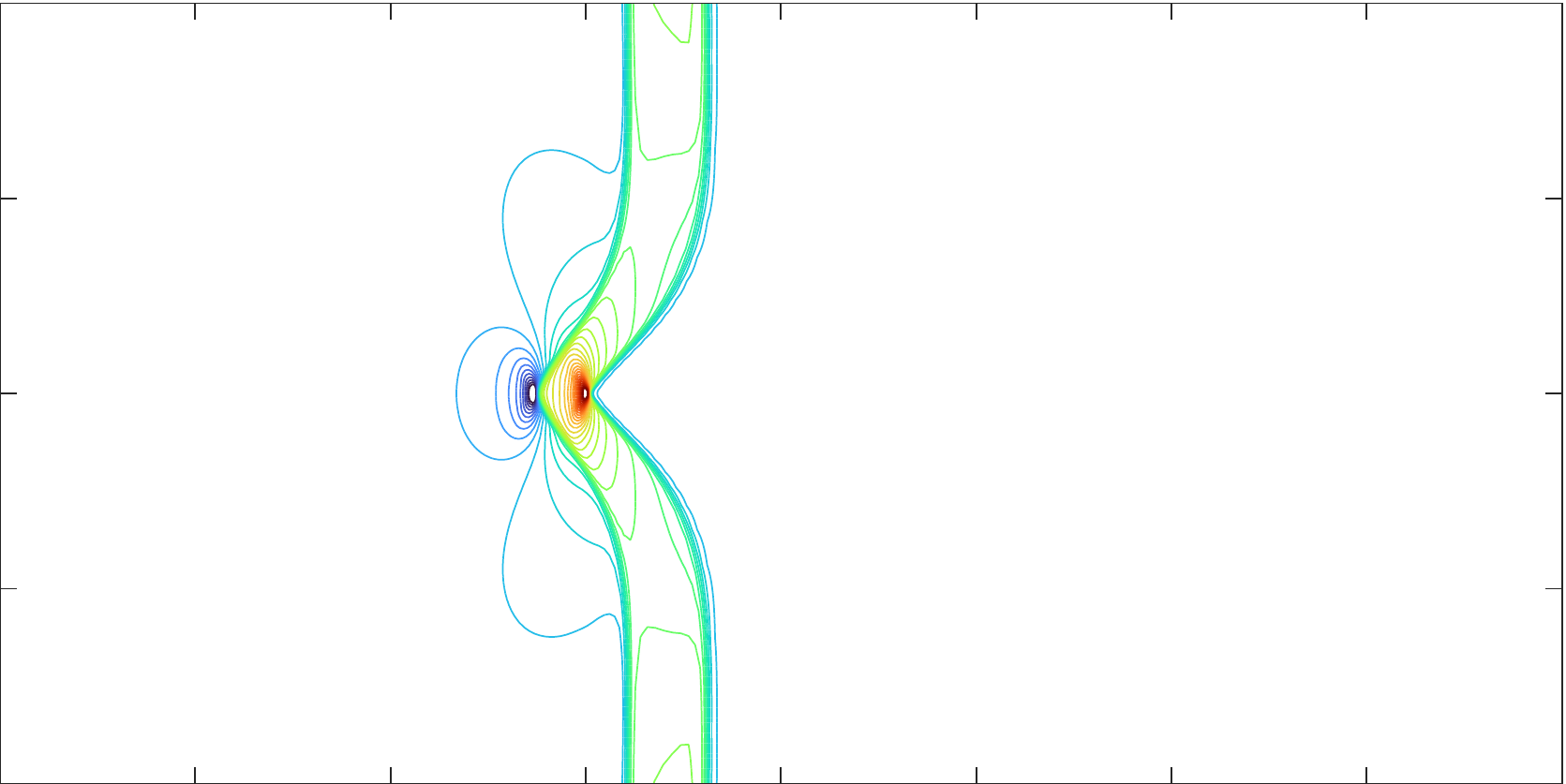}
  \end{subfigure}~
  \begin{subfigure}[b]{0.4\textwidth}
    \centering
    \includegraphics[width=1.0\linewidth]{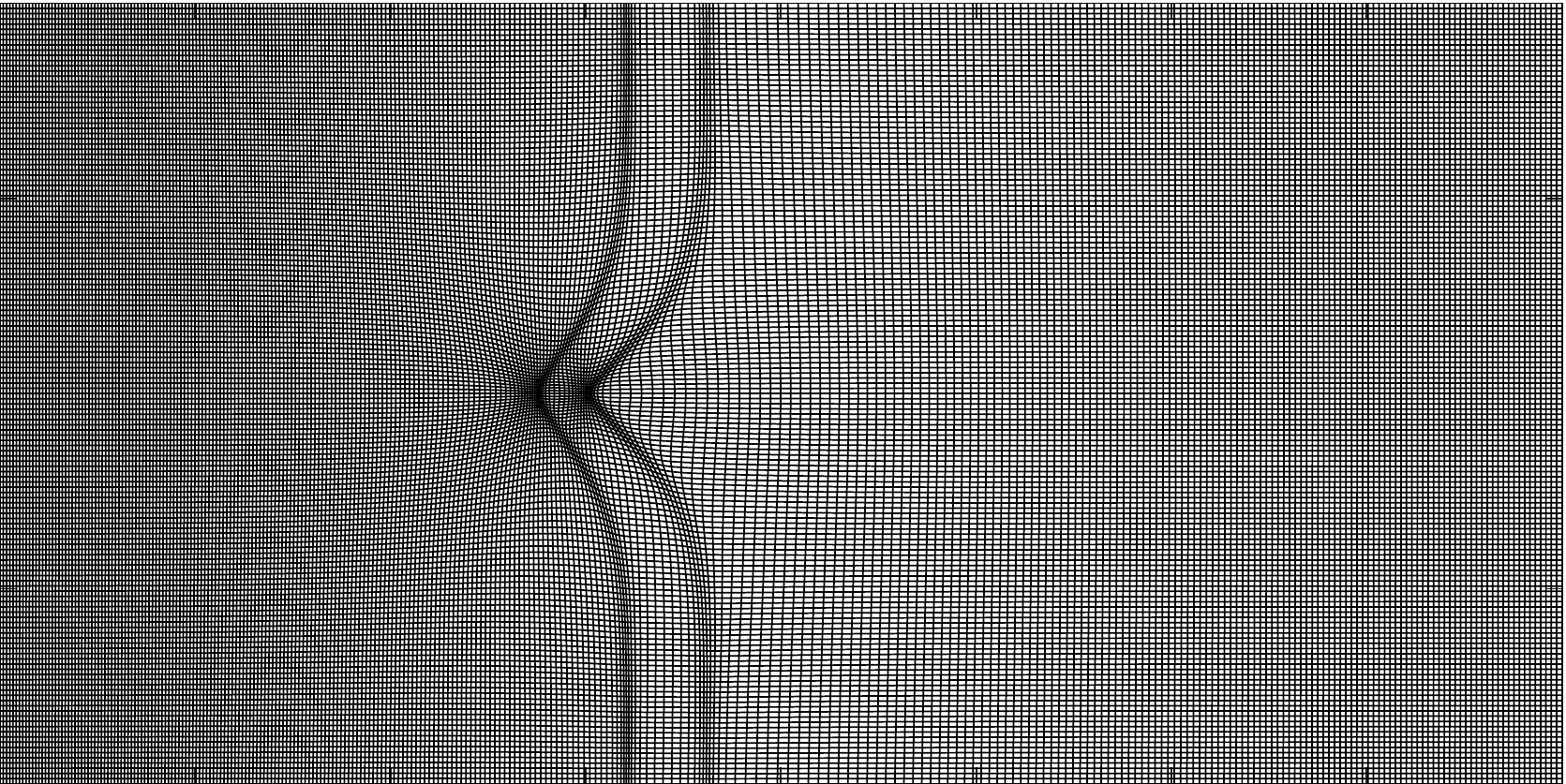}
  \end{subfigure}
  \quad\\
  \begin{subfigure}[b]{0.4\textwidth}
    \centering
    \includegraphics[width=1.0\linewidth]{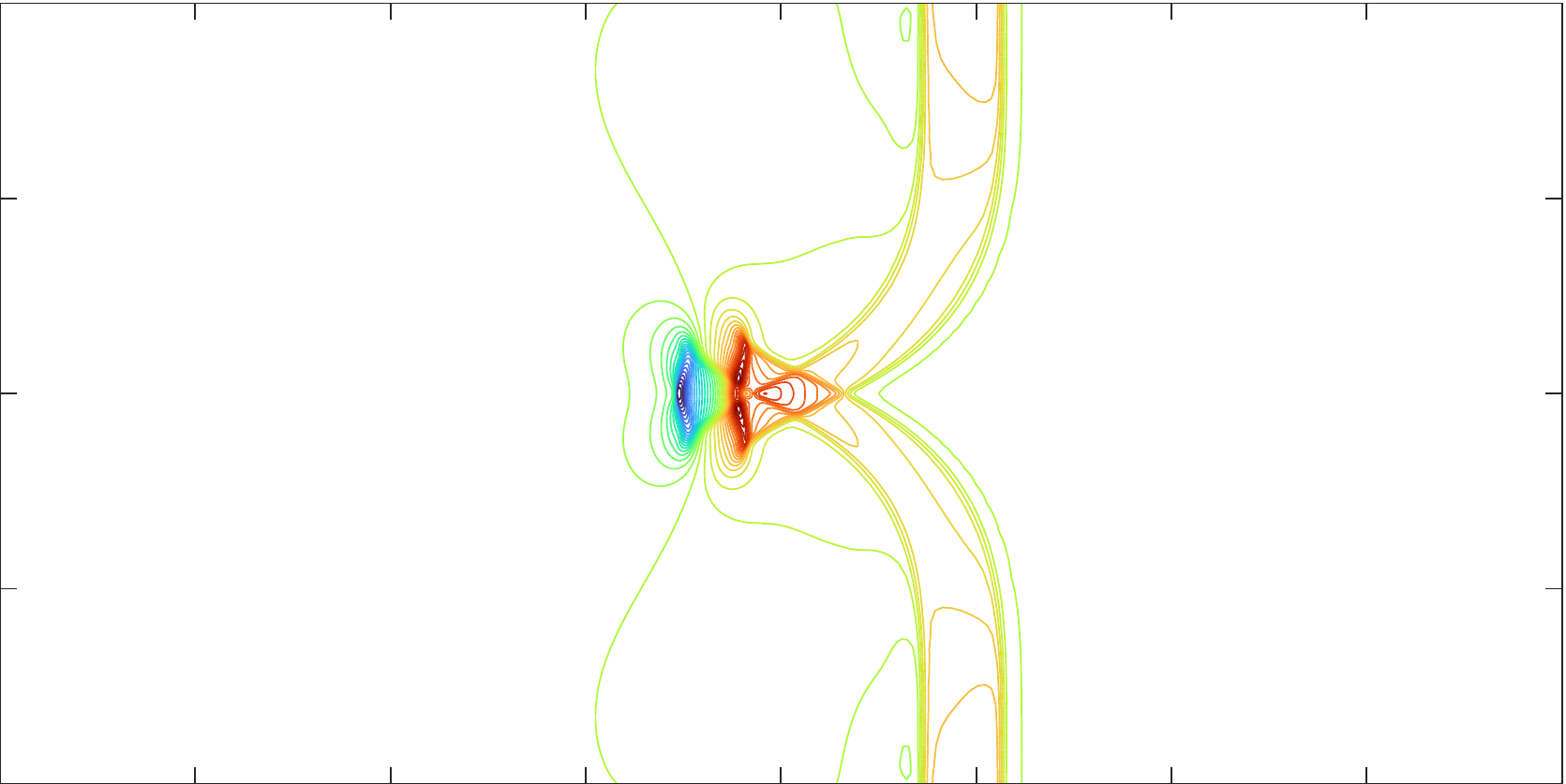}
  \end{subfigure}~
  \begin{subfigure}[b]{0.4\textwidth}
    \centering
    \includegraphics[width=1.0\linewidth]{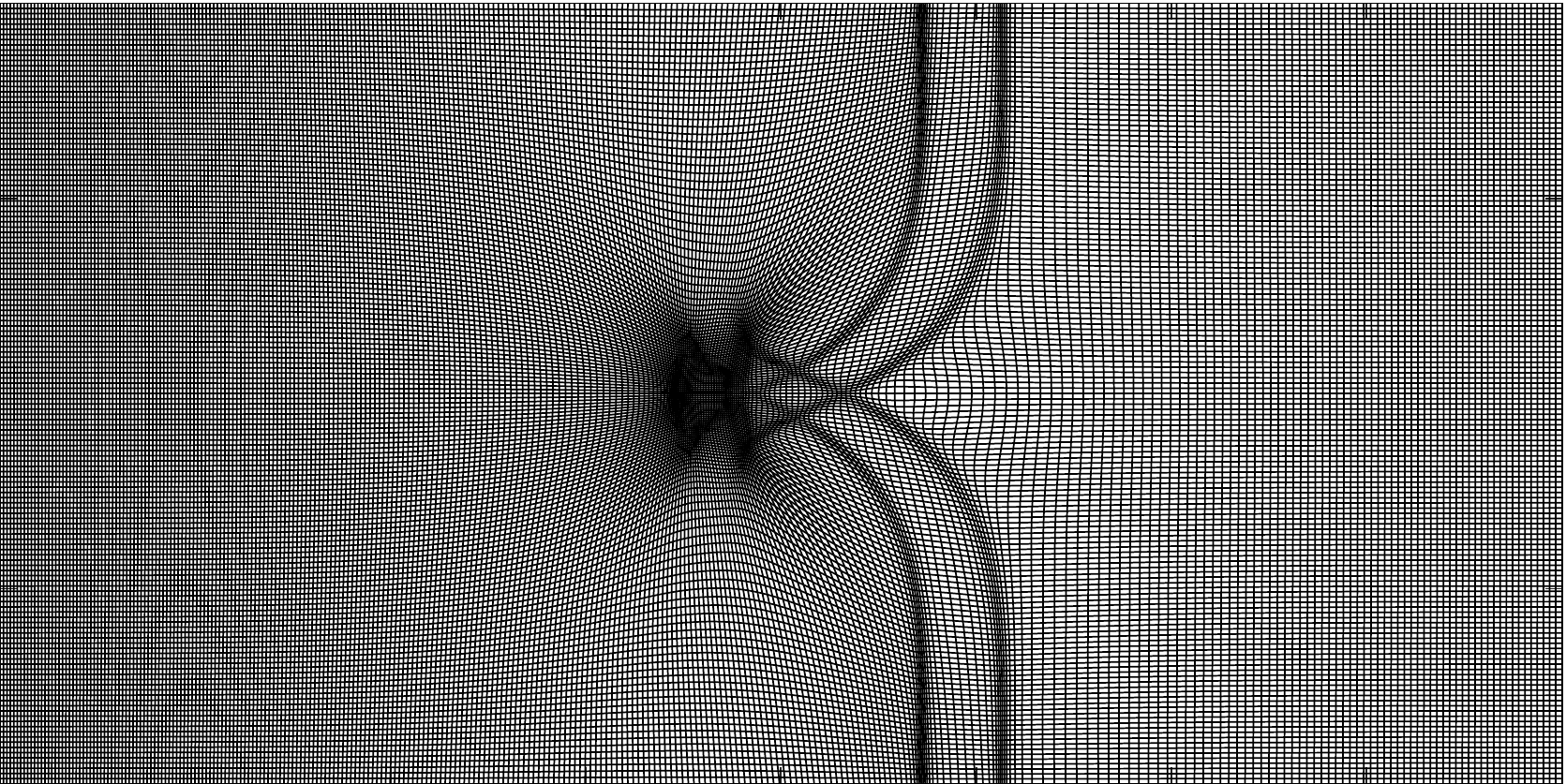}
  \end{subfigure}
  \quad\\
  \begin{subfigure}[b]{0.4\textwidth}
    \centering
    \includegraphics[width=1.0\linewidth]{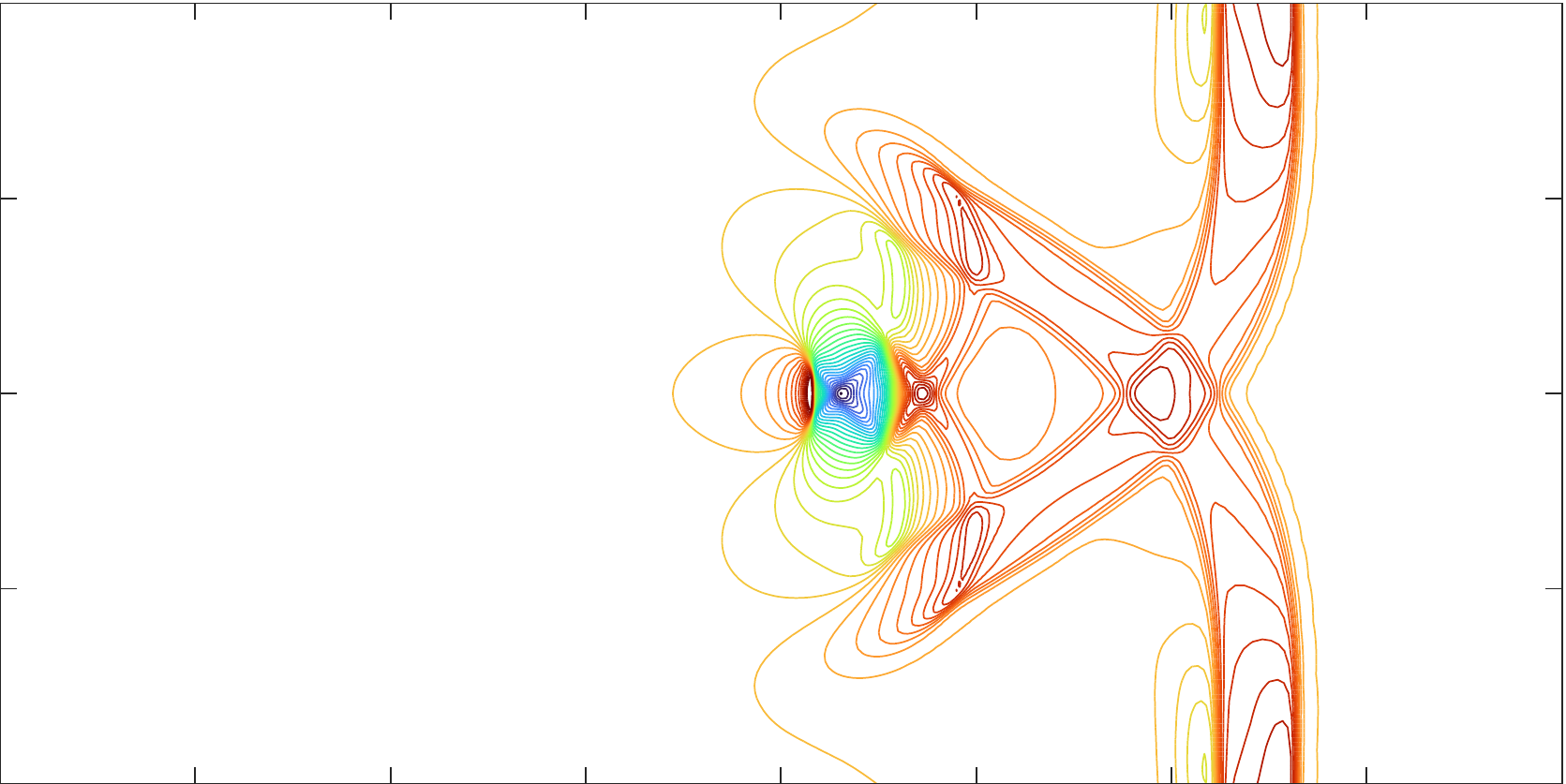}
  \end{subfigure}~
  \begin{subfigure}[b]{0.4\textwidth}
    \centering
    \includegraphics[width=1.0\linewidth]{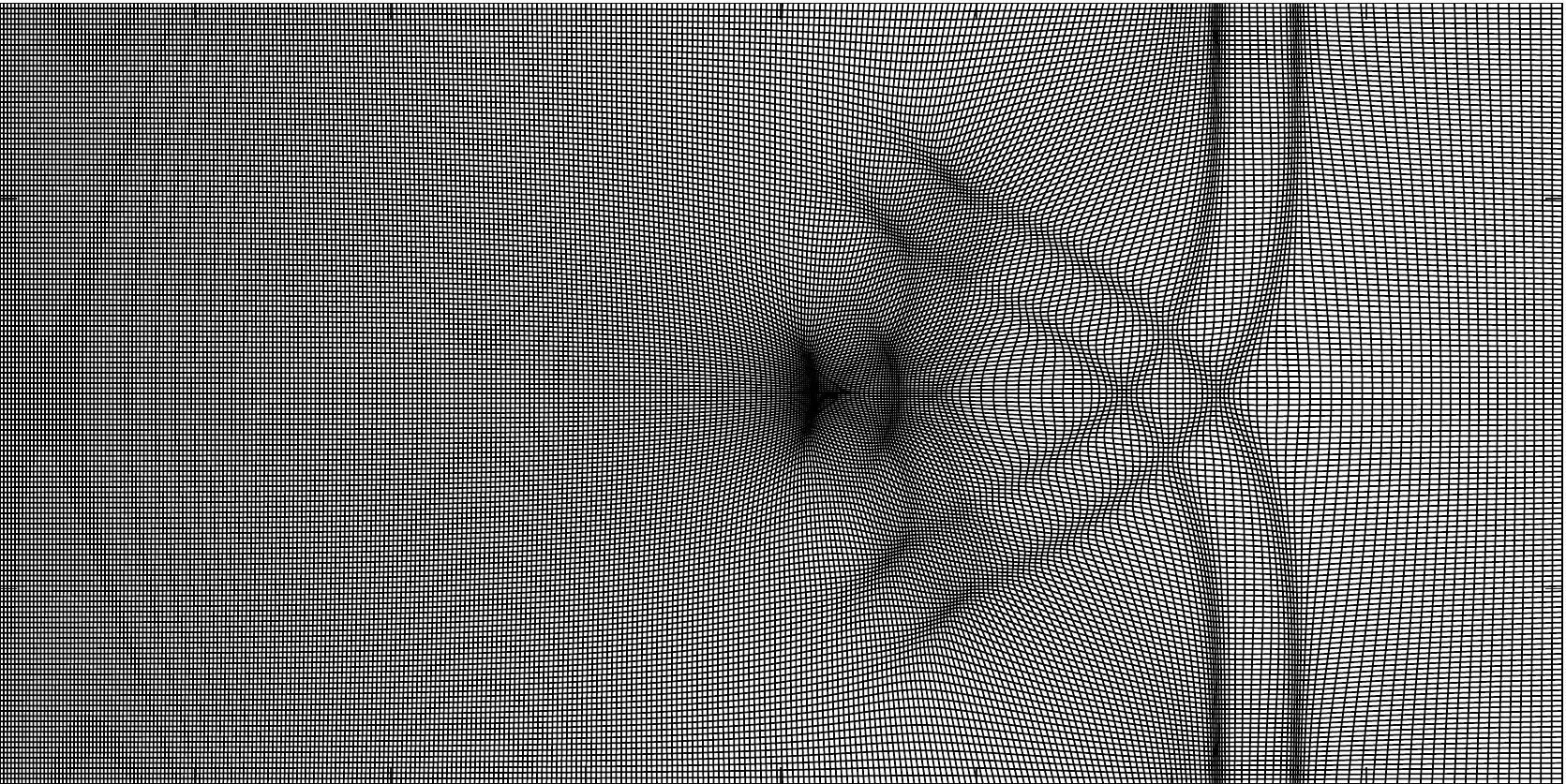}
  \end{subfigure}
  \quad\\
  \begin{subfigure}[b]{0.4\textwidth}
    \centering
    \includegraphics[width=1.0\linewidth]{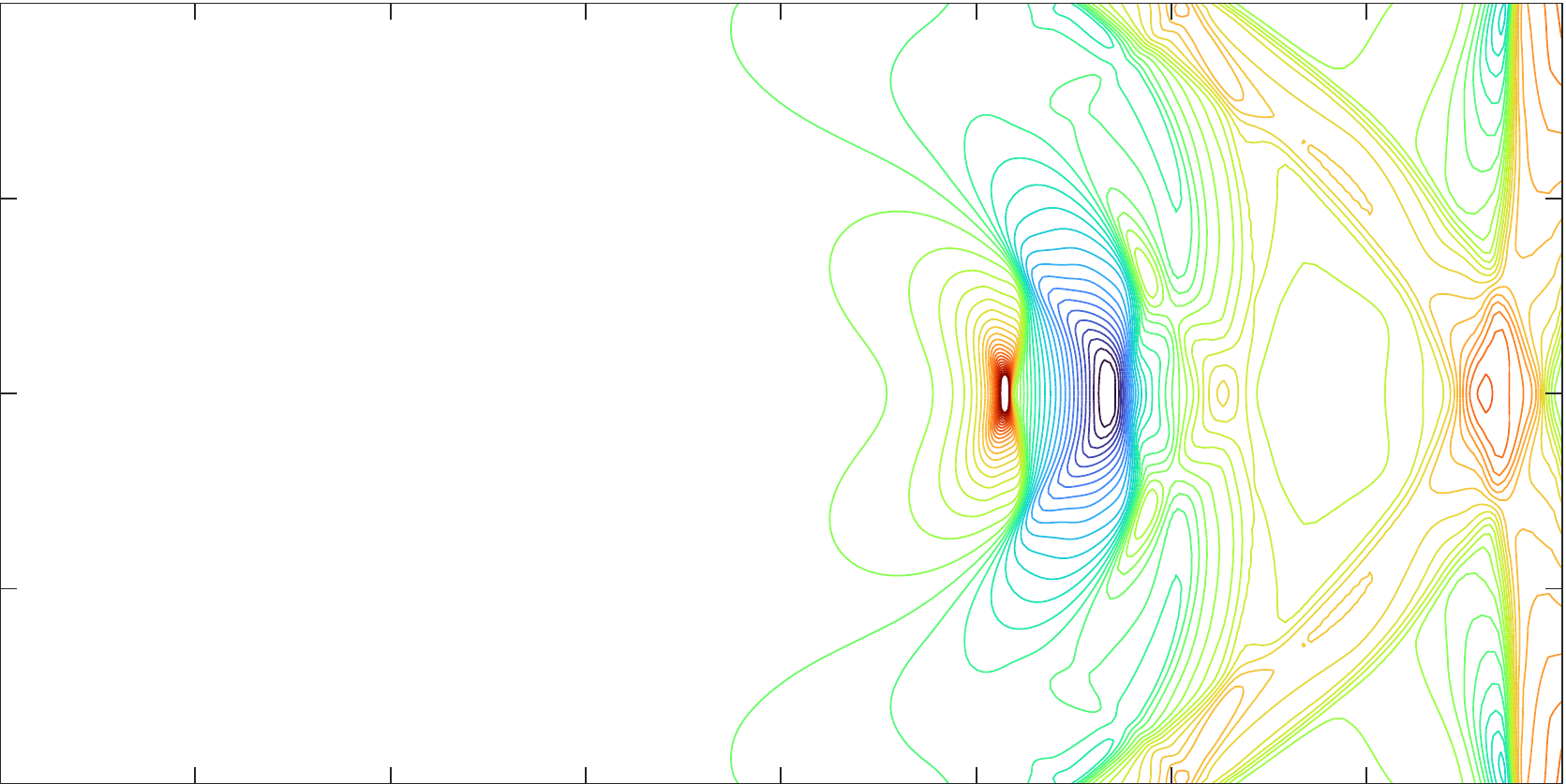}
  \end{subfigure}~
  \begin{subfigure}[b]{0.4\textwidth}
    \centering
    \includegraphics[width=1.0\linewidth]{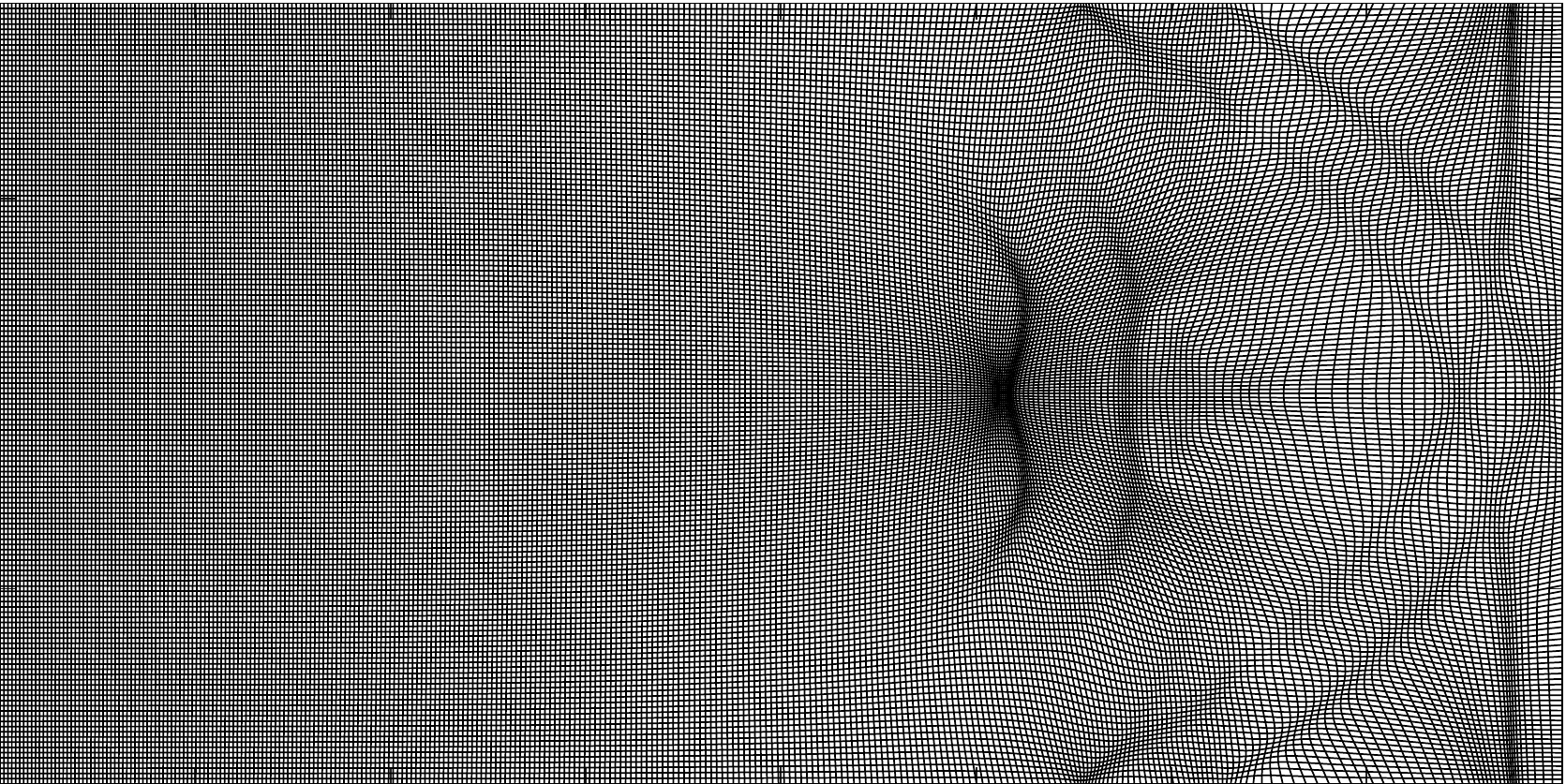}
  \end{subfigure}
  \caption{Example \ref{ex:Pertubation}. The $40$ equally spaced contours of the water surface level $h+b$ and the plots of the adaptive meshes obtained by using the {\tt MM-ES} scheme with $300\times150$ mesh. From top to bottom: $t = 0.12$, $0.24$, $0.36$, $0.48$, $0.6$.}\label{fig:2D_MM_Mesh_Value}
\end{figure}

\begin{figure}[hbt!]
  \centering
  \begin{subfigure}[b]{0.3\textwidth}
    \centering
    \includegraphics[width=1.0\linewidth]{Figures/2DMMPerturbation/2D_h_b_t=0.12.pdf}
  \end{subfigure}~
  \begin{subfigure}[b]{0.3\textwidth}
    \centering
    \includegraphics[width=1.0\linewidth]{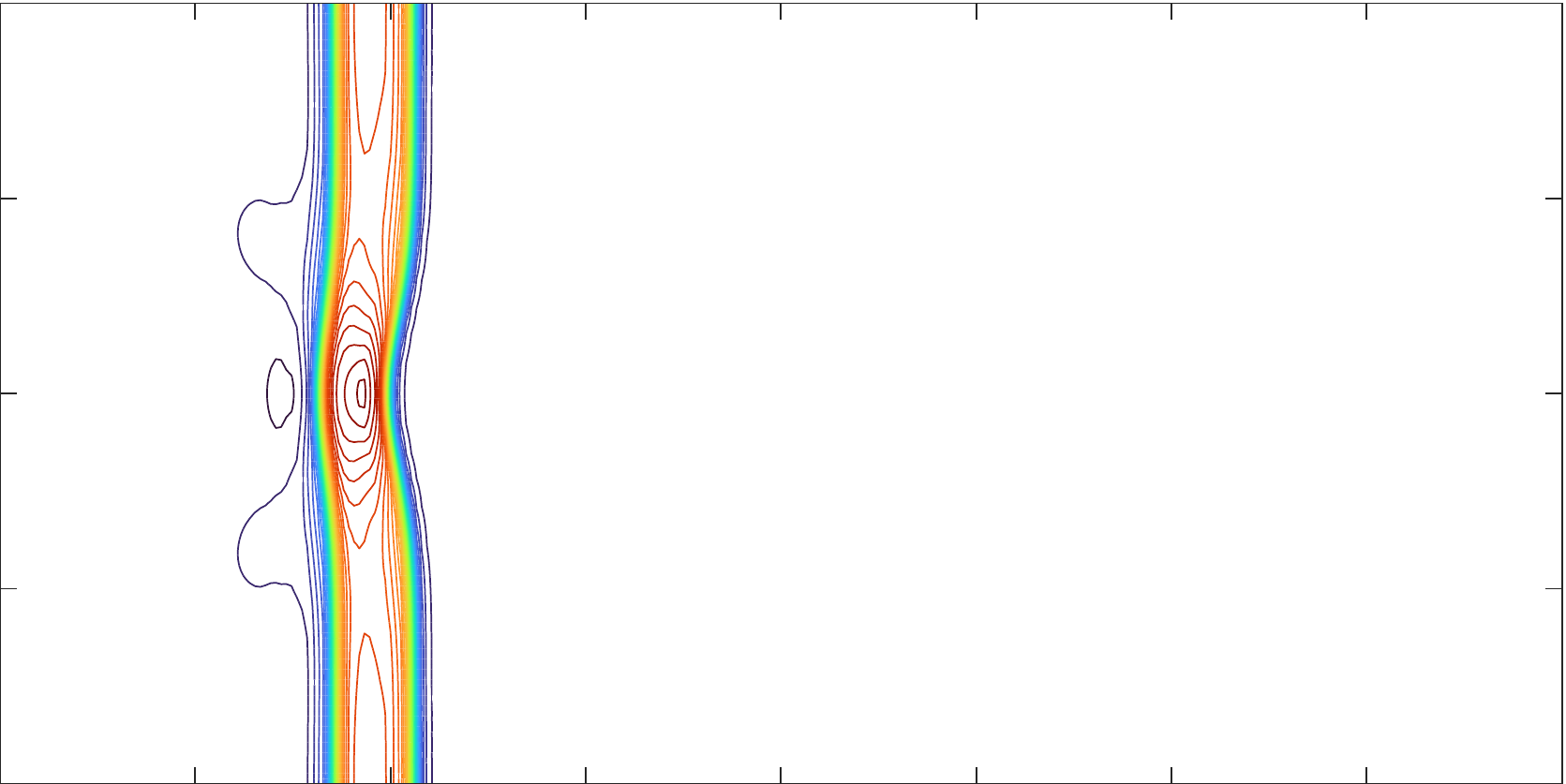}
  \end{subfigure}~
  \begin{subfigure}[b]{0.3\textwidth}
    \centering
    \includegraphics[width=1.0\linewidth]{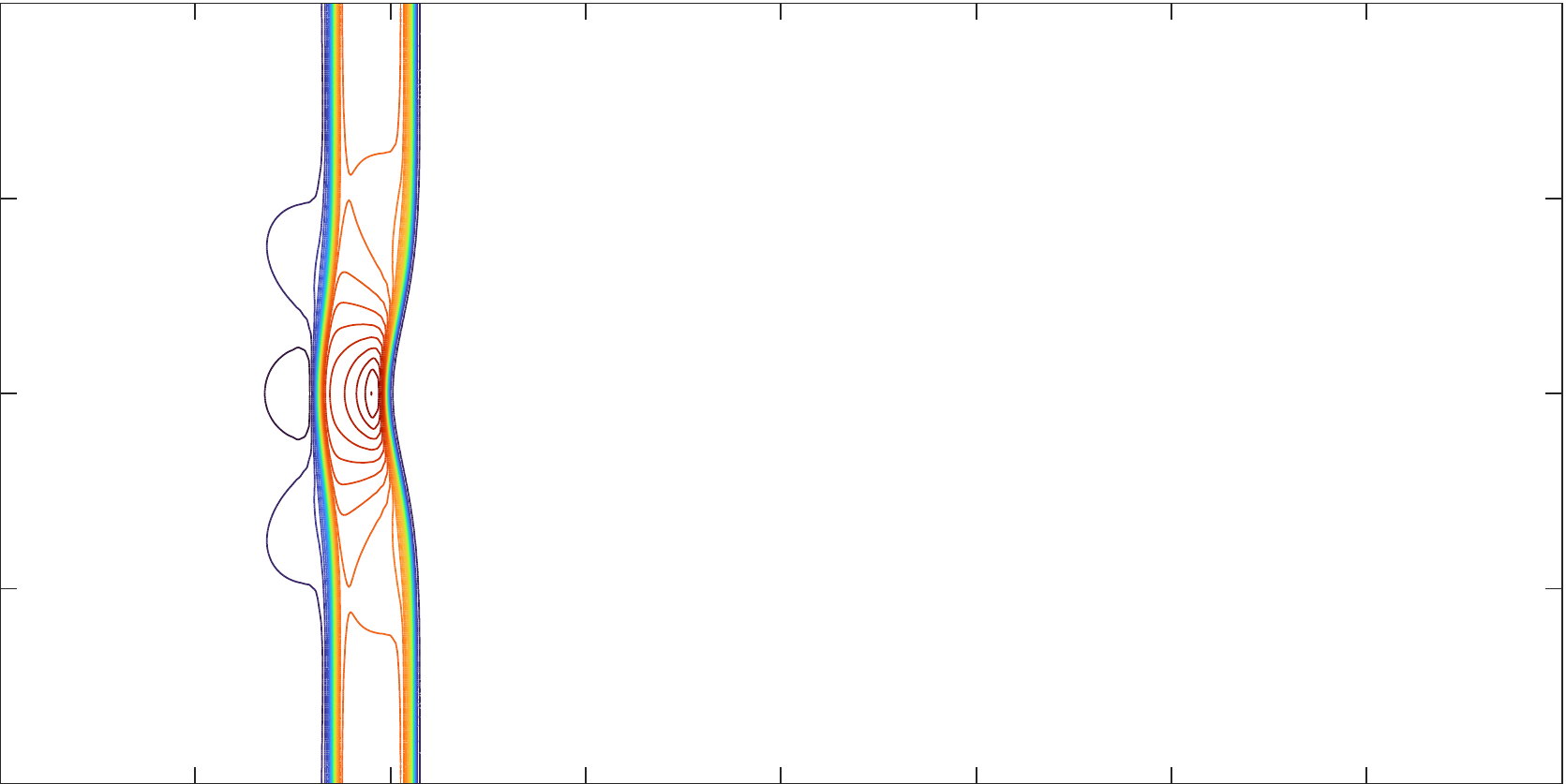}
  \end{subfigure}
  \\
  \begin{subfigure}[b]{0.3\textwidth}
    \centering
    \includegraphics[width=1.0\linewidth]{Figures/2DMMPerturbation/2D_h_b_t=0.24.pdf}
  \end{subfigure}~
  \begin{subfigure}[b]{0.3\textwidth}
    \centering
    \includegraphics[width=1.0\linewidth]{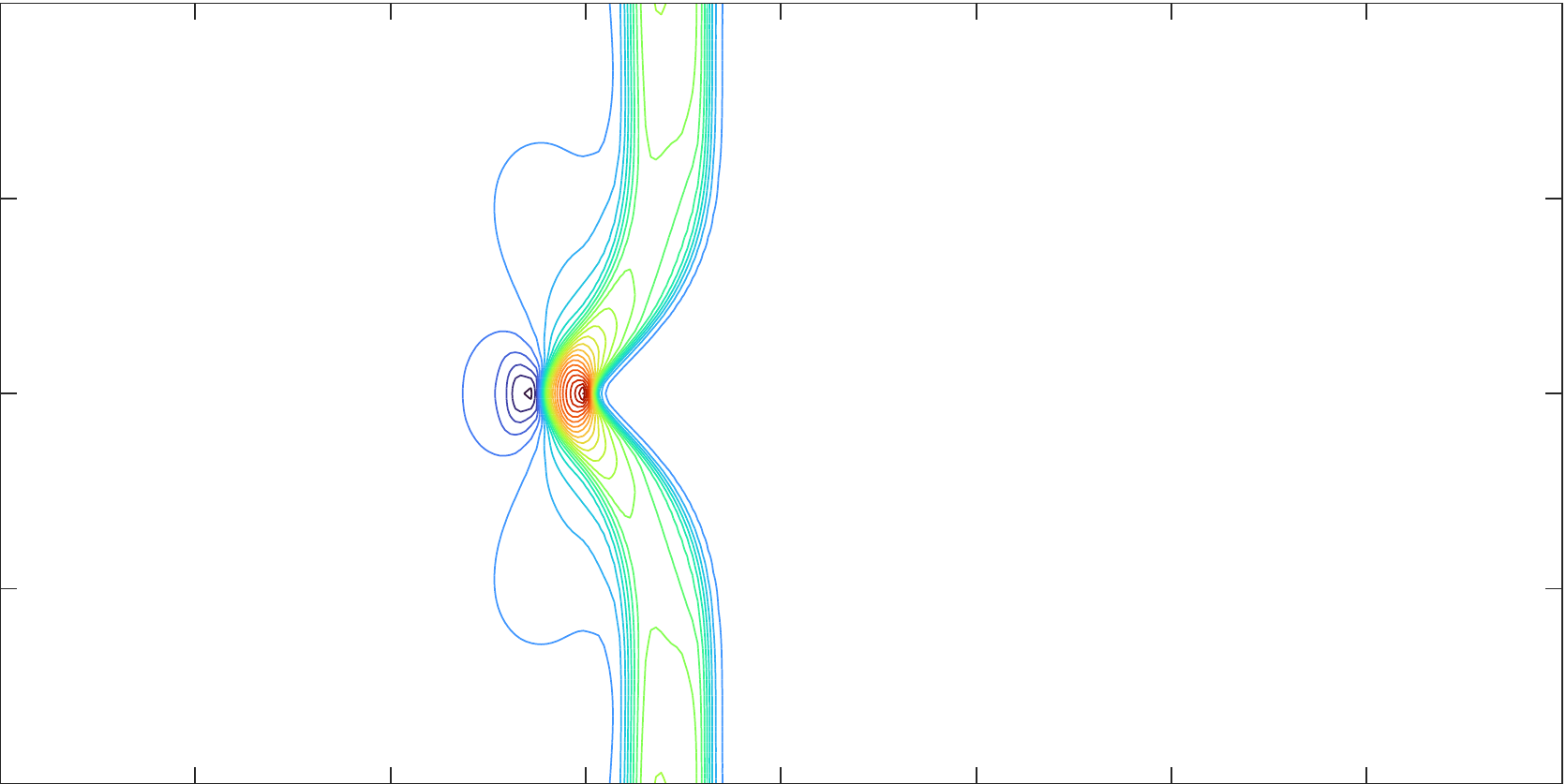}
  \end{subfigure}~
  \begin{subfigure}[b]{0.3\textwidth}
    \centering
    \includegraphics[width=1.0\linewidth]{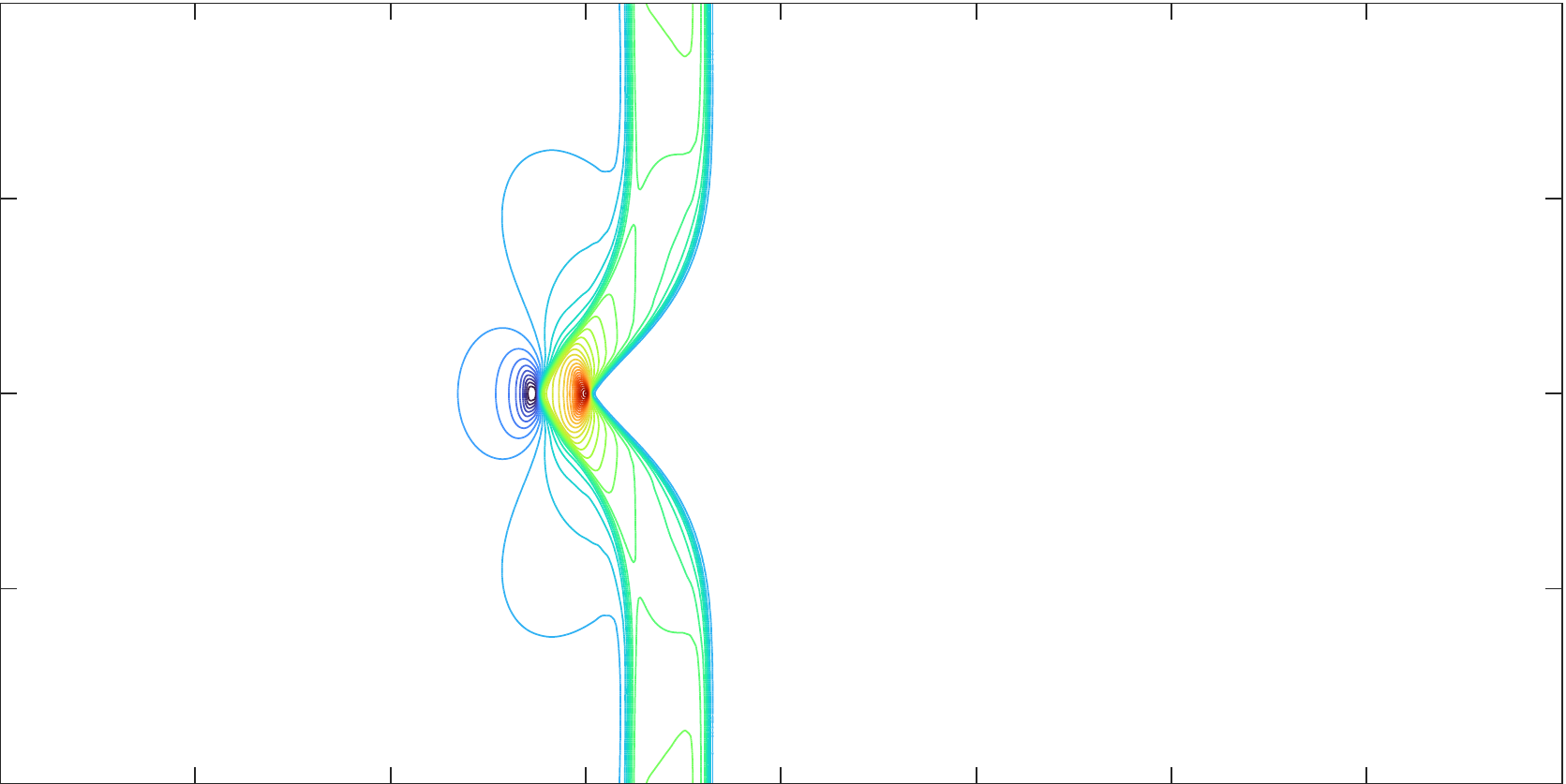}
  \end{subfigure}\\
  \begin{subfigure}[b]{0.3\textwidth}
    \centering
    \includegraphics[width=1.0\linewidth]{Figures/2DMMPerturbation/2D_h_b_t=0.36.pdf}
  \end{subfigure}~
  \begin{subfigure}[b]{0.3\textwidth}
    \centering
    \includegraphics[width=1.0\linewidth]{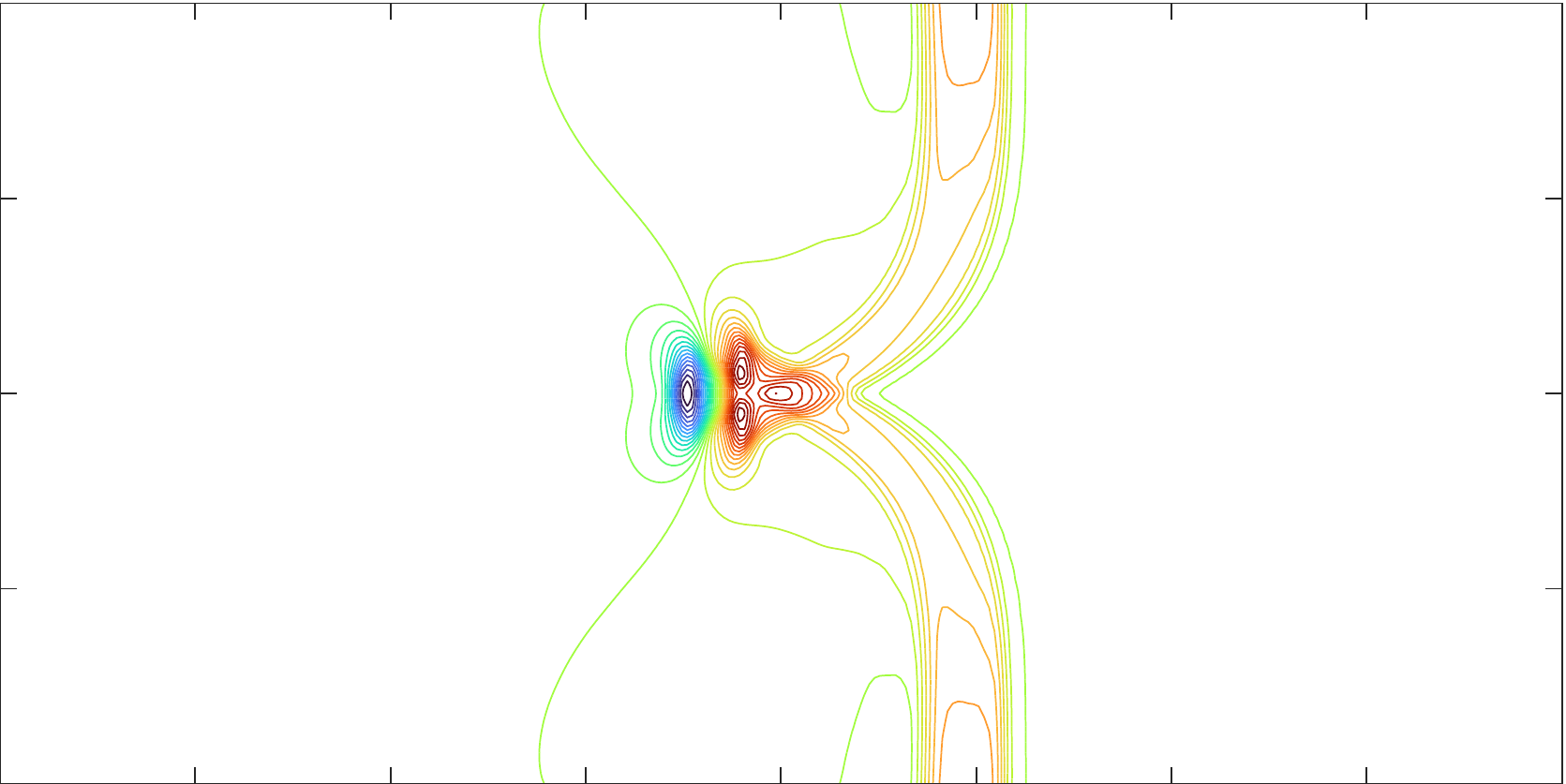}
  \end{subfigure}~
  \begin{subfigure}[b]{0.3\textwidth}
    \centering
    \includegraphics[width=1.0\linewidth]{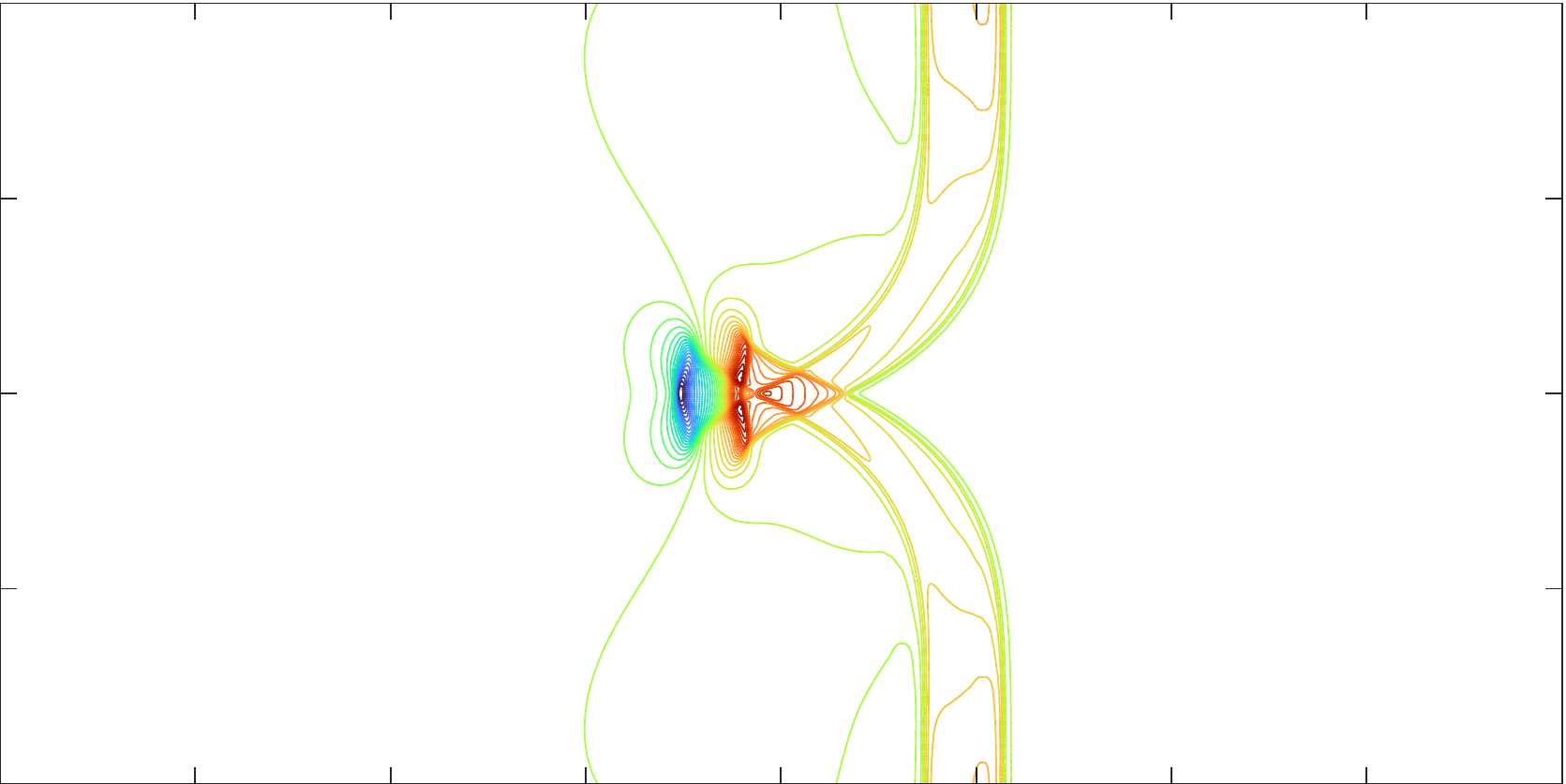}
  \end{subfigure}\\
  \begin{subfigure}[b]{0.3\textwidth}
    \centering
    \includegraphics[width=1.0\linewidth]{Figures/2DMMPerturbation/2D_h_b_t=0.48.pdf}
  \end{subfigure}~
  \begin{subfigure}[b]{0.3\textwidth}
    \centering
    \includegraphics[width=1.0\linewidth]{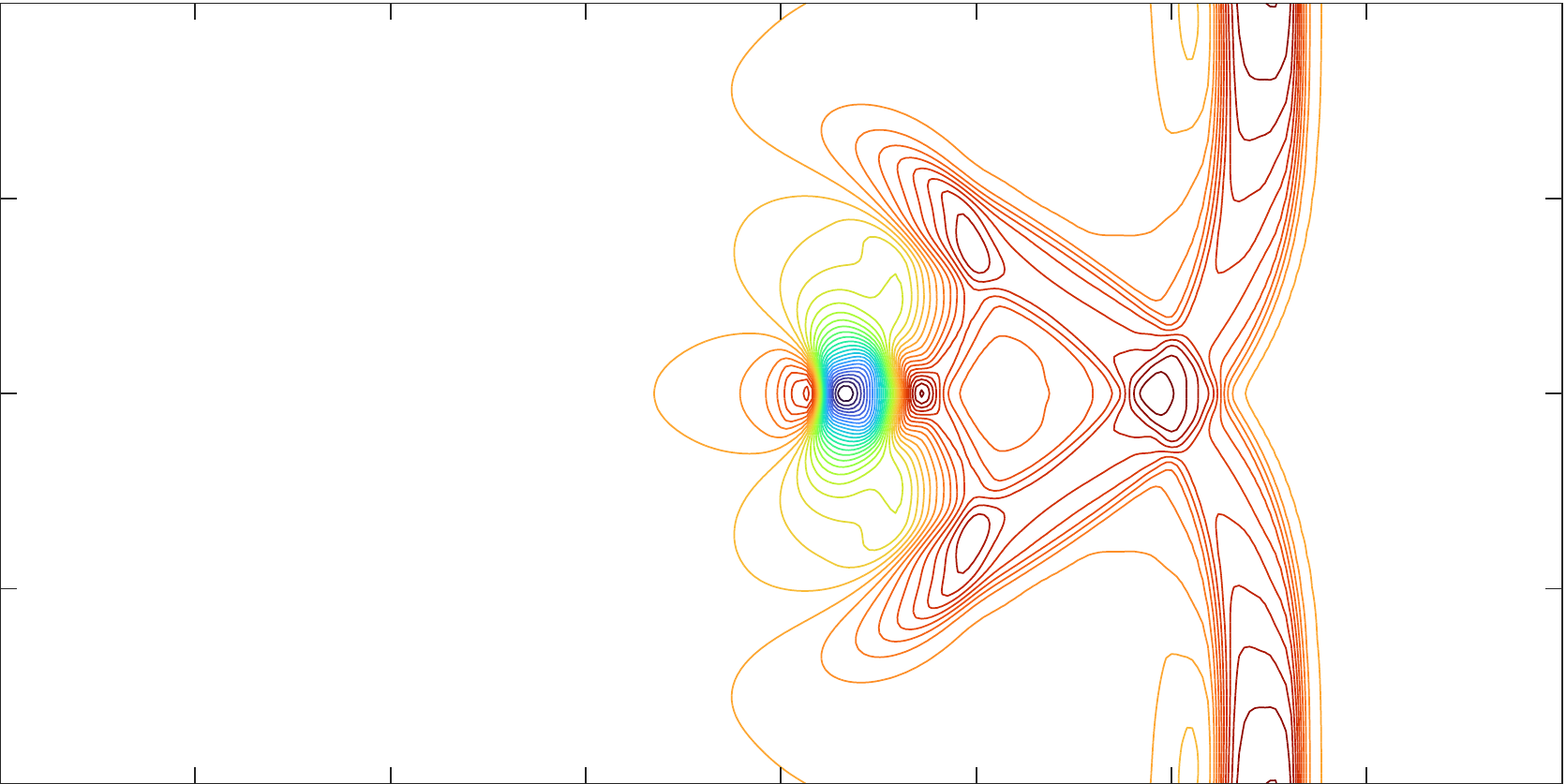}
  \end{subfigure}~
  \begin{subfigure}[b]{0.3\textwidth}
    \centering
    \includegraphics[width=1.0\linewidth]{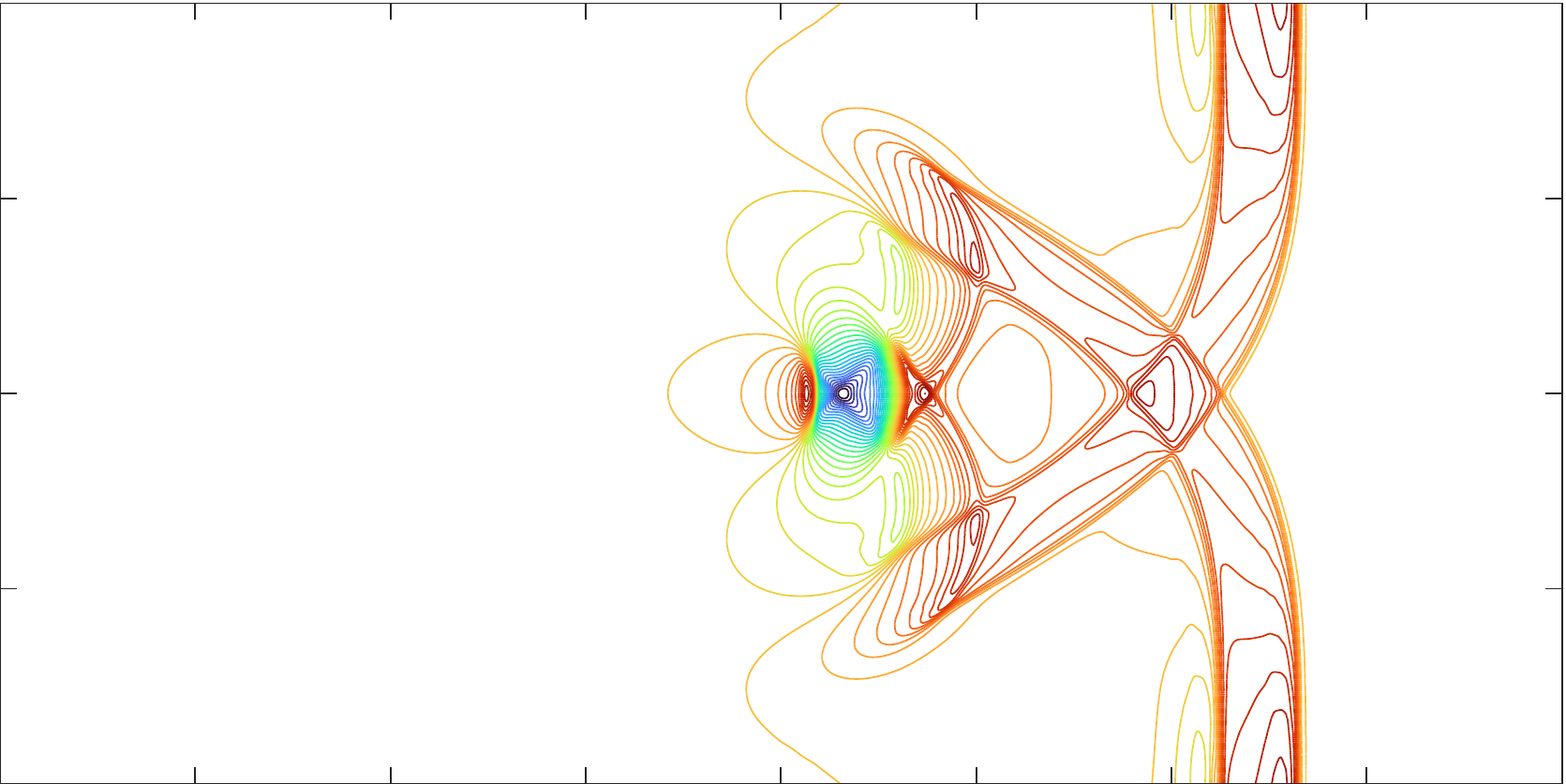}
  \end{subfigure}
  \\
  \begin{subfigure}[b]{0.3\textwidth}
    \centering
    \includegraphics[width=1.0\linewidth]{Figures/2DMMPerturbation/2D_h_b_t=0.6.pdf}
  \end{subfigure}~
  \begin{subfigure}[b]{0.3\textwidth}
    \centering
    \includegraphics[width=1.0\linewidth]{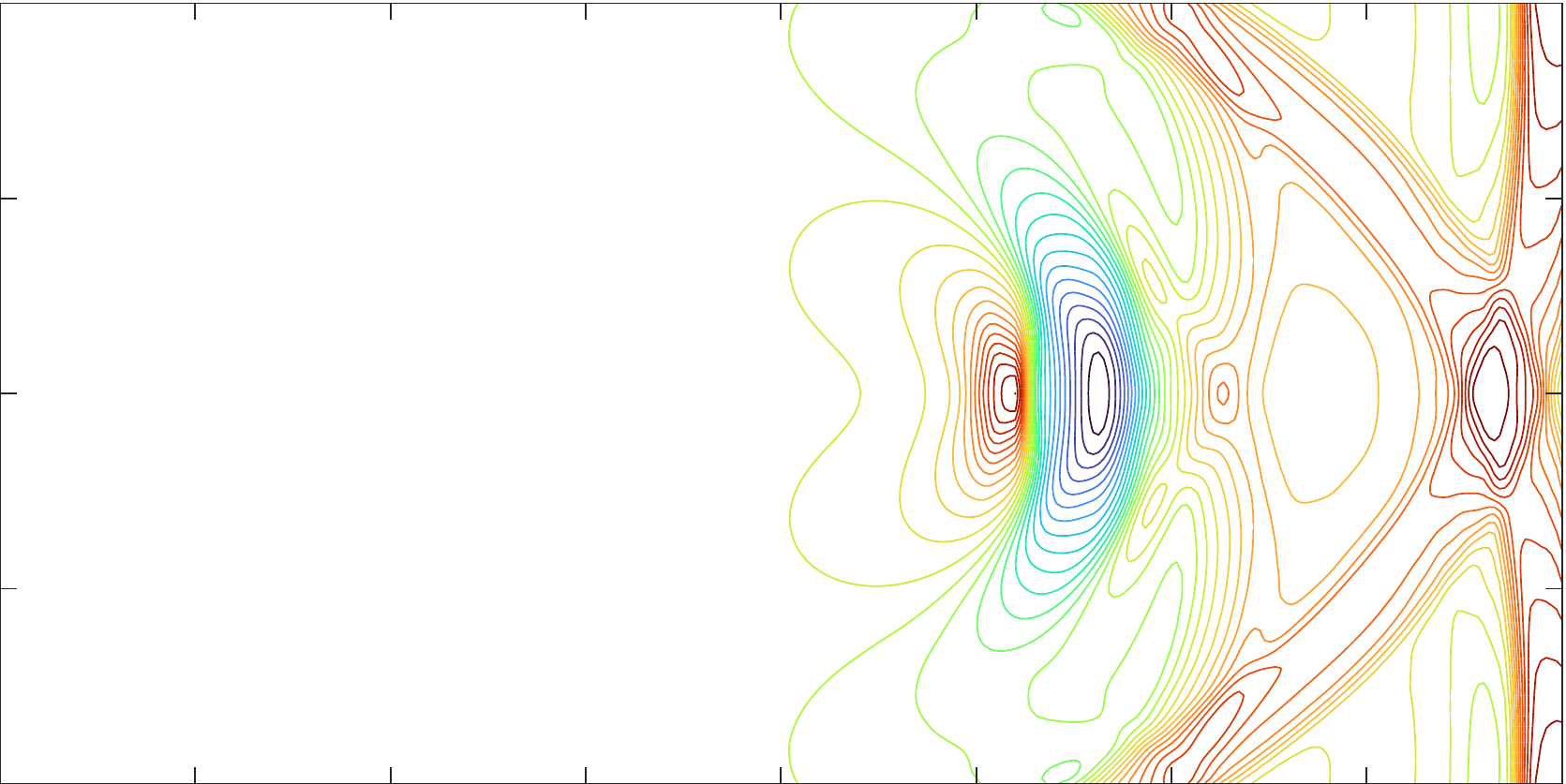}
  \end{subfigure}~
  \begin{subfigure}[b]{0.3\textwidth}
    \centering
    \includegraphics[width=1.0\linewidth]{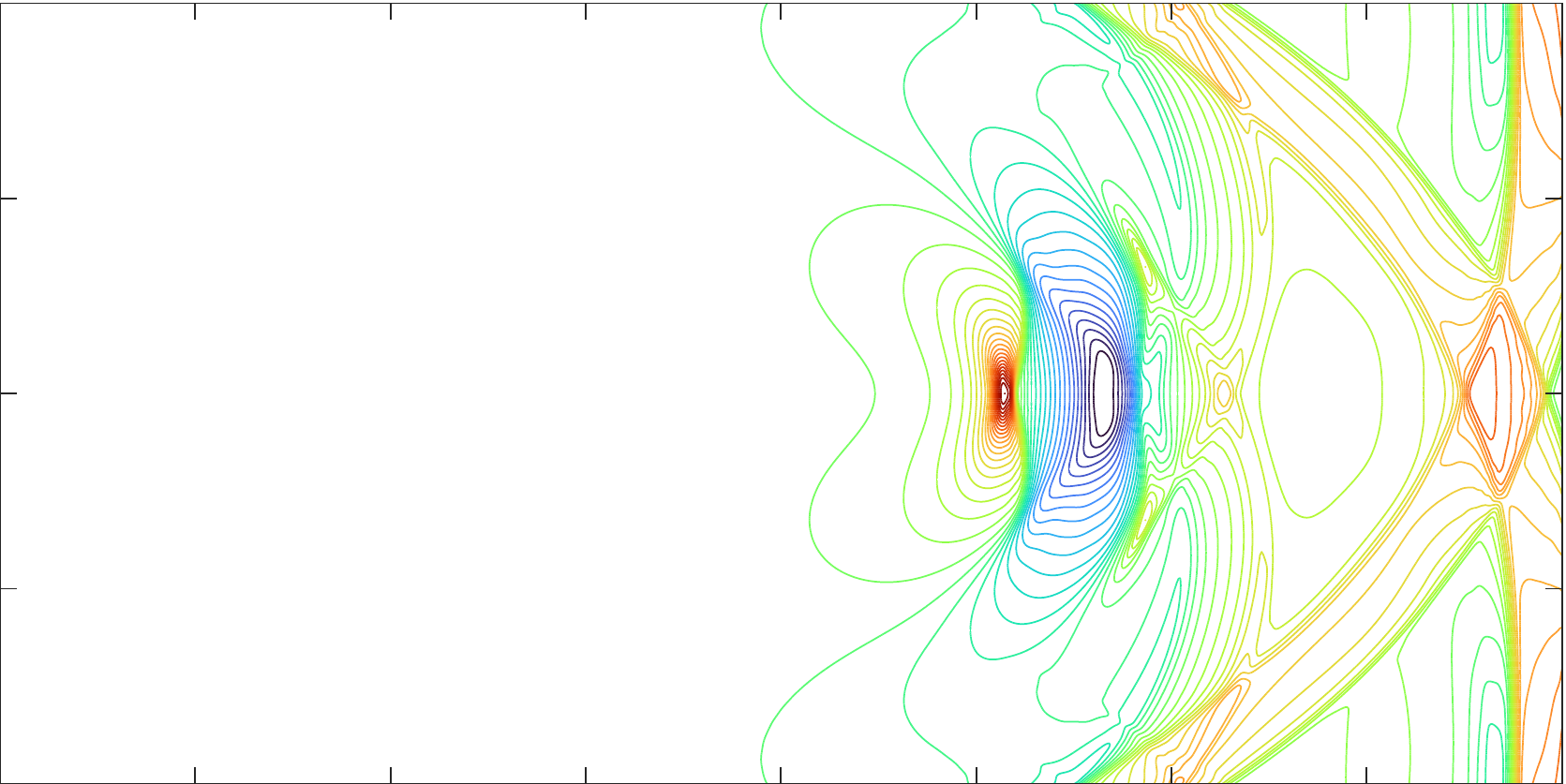}
  \end{subfigure}
  \caption{Example \ref{ex:Pertubation}. The $40$ equally spaced contours of the water surface level $h+b$ obtained by using different schemes. Left: {\tt MM-ES} with $300\times150$ mesh, middle: {\tt UM-ES} with $300\times150$ mesh, right: {\tt UM-ES} with $900\times450$ mesh.
  From top to bottom: $t = 0.12$, $0.24$, $0.36$, $0.48$, $0.6$.}\label{2D_Case_compare}
\end{figure}

\begin{example}[Circular dam break problem]\label{ex:Circle_Dam_flat}\rm
  This example simulates the circular dam break problem with flat bottom \cite{Alcrudo1993High},
  which is used to examine the ability of our schemes to maintain cylindrical symmetry.
  The physical domain is $[0,50]\times[0,50]$ with the outflow boundary conditions,
  and the initial water depth and velocities are
  \begin{equation*}
    \begin{aligned}
      & h= \begin{cases}10, & \text{if} \quad \sqrt{(x_1-25)^2+(x_2-25)^2} \leqslant 11, \\
      1, & \text {otherwise,}\end{cases} \\
      & v_1=v_2=0.
    \end{aligned}
  \end{equation*}
  The output times are $t = 0$, $0.2$, $0.4$, $0.6$, $0.8$, $1.0$, with the gravitational constant $g=9.812$.
  The monitor function is chosen as the same in the last Example.
\end{example}

Figures \ref{fig:Circle_dam_flat}-\ref{fig:Circle_dam_flat_Mesh} present the results obtained by the {\tt MM-ES} scheme with $200\times200$ mesh,
which clearly show that our WB ES adaptive moving mesh scheme can capture the wave structures sharply
and preserve the symmetry well.
\begin{figure}[hbt!]
  \centering
  \begin{subfigure}[b]{0.25\textwidth}
    \centering
    \includegraphics[width=1.0\linewidth]{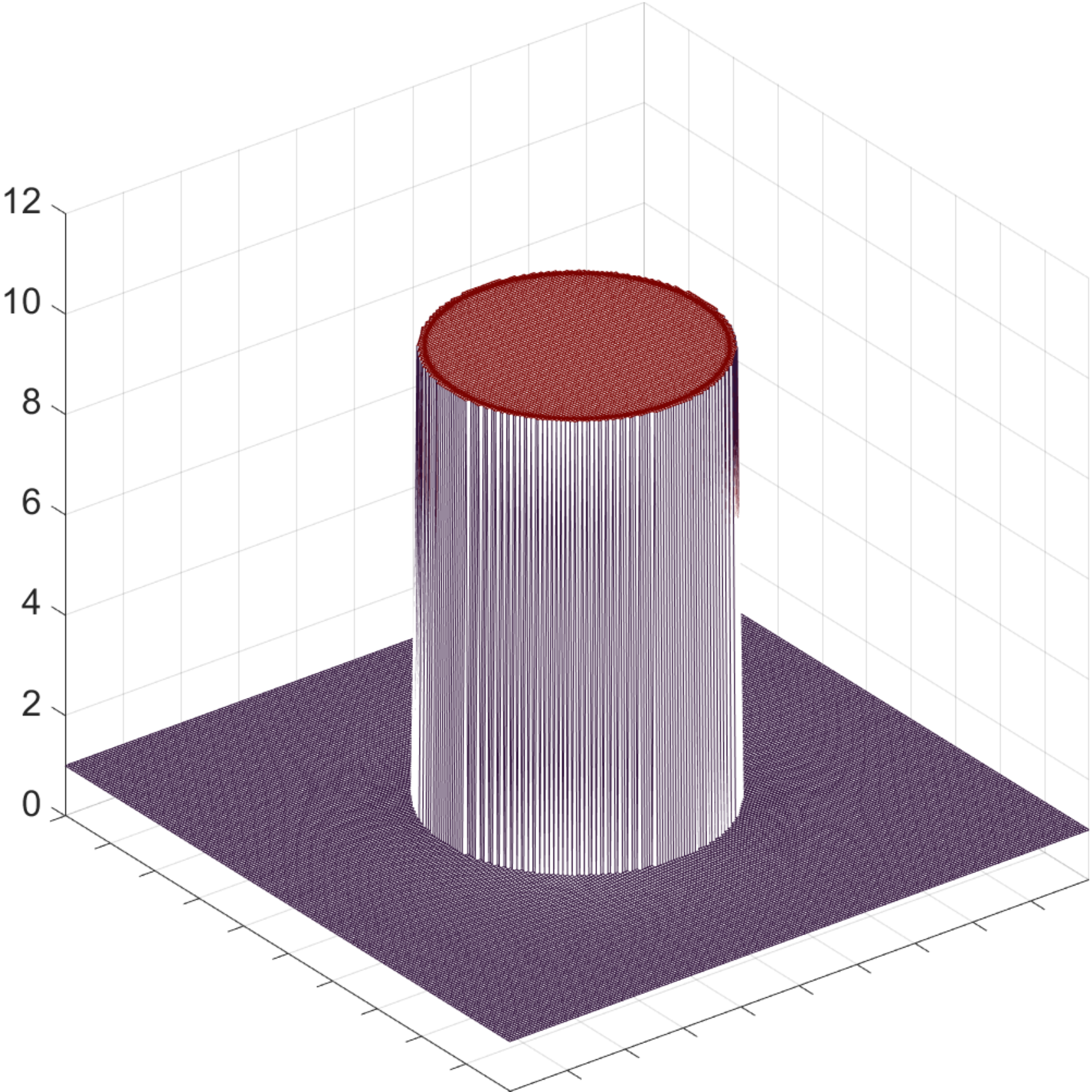}
    \caption{$t=0.0$}
  \end{subfigure}
  \begin{subfigure}[b]{0.25\textwidth}
    \centering
    \includegraphics[width=1.0\linewidth]{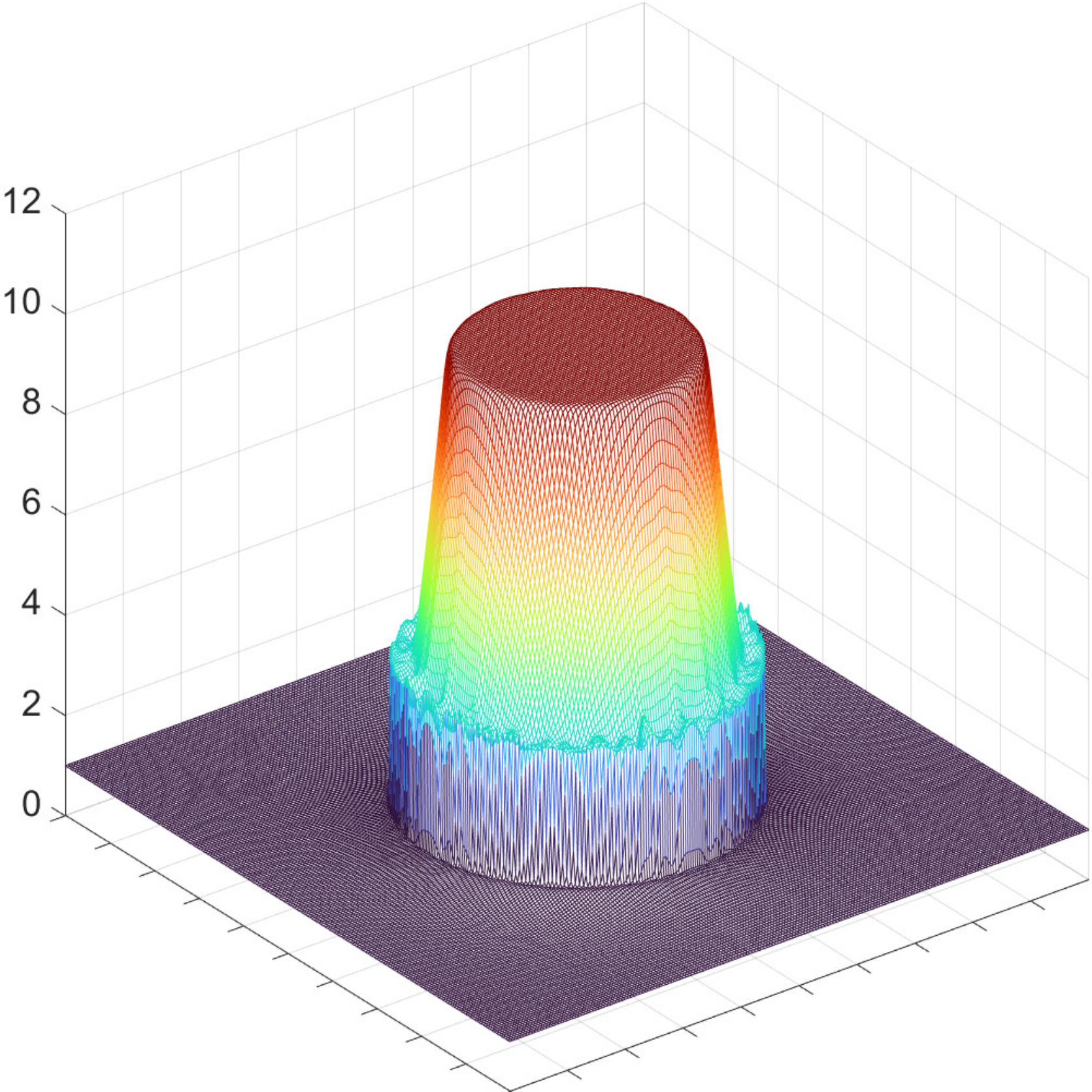}
    \caption{$t=0.2$}
  \end{subfigure}
  \begin{subfigure}[b]{0.25\textwidth}
    \centering
    \includegraphics[width=1.0\linewidth]{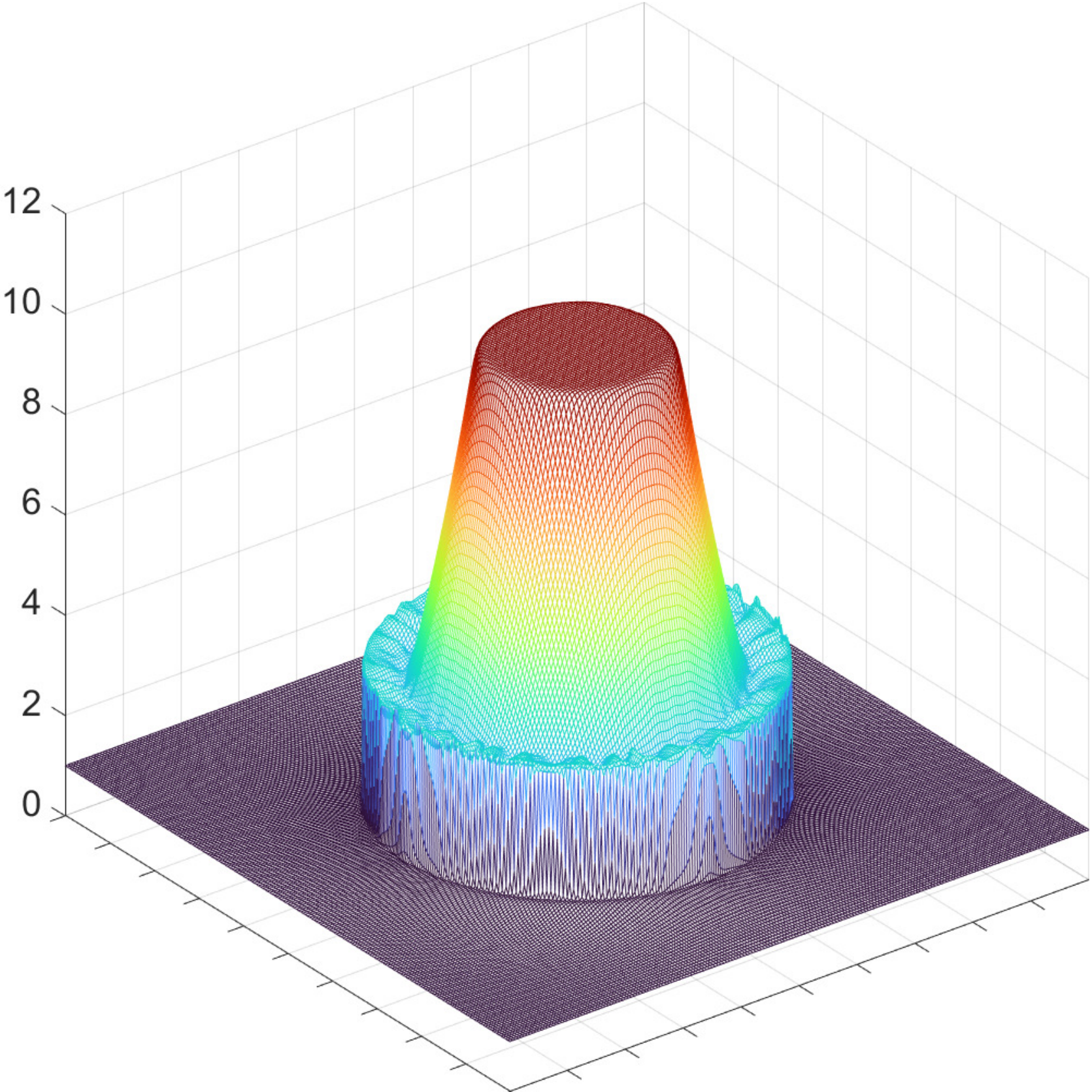}
    \caption{$t=0.4$}
  \end{subfigure}
  \quad\\
  \begin{subfigure}[b]{0.25\textwidth}
    \centering
    \includegraphics[width=1.0\linewidth]{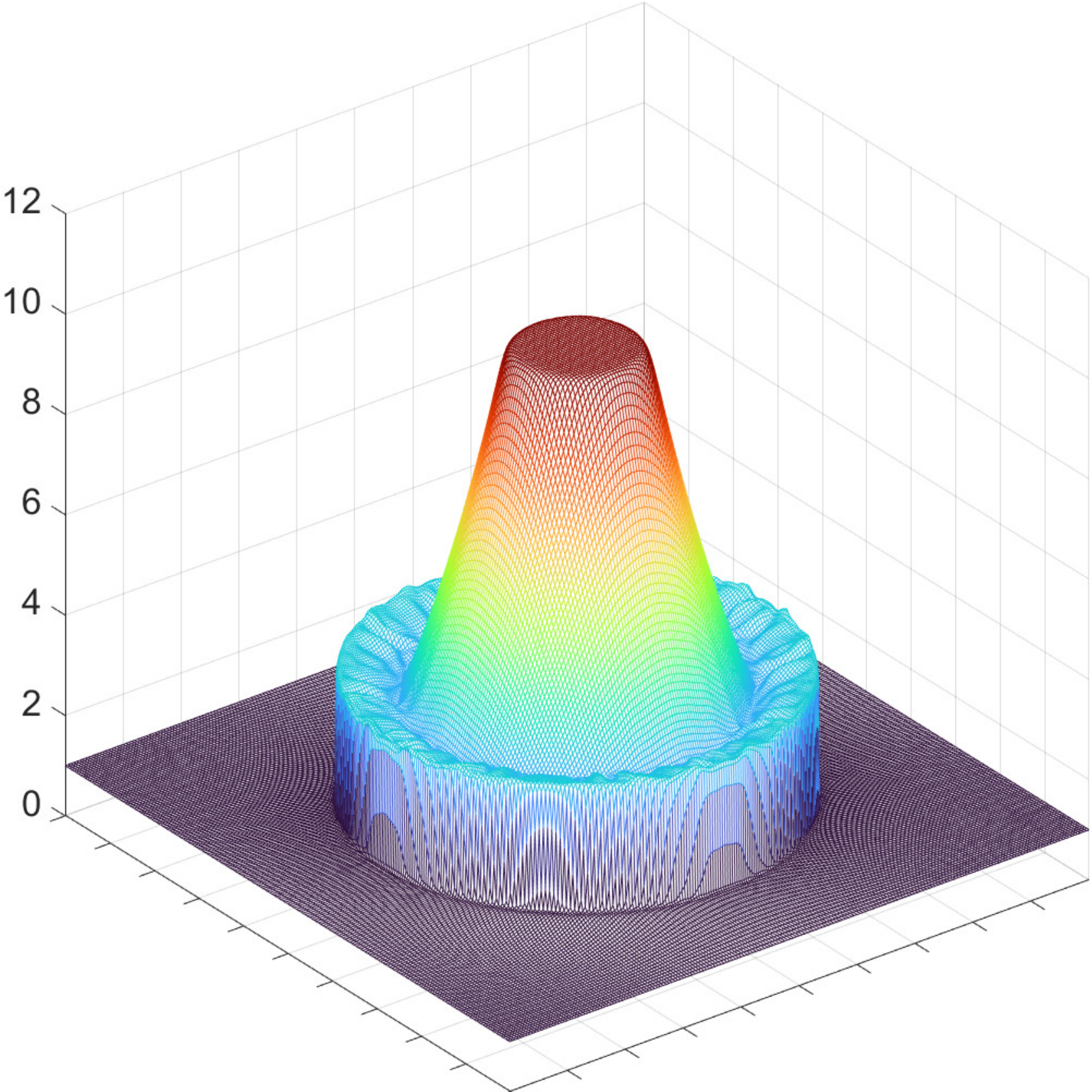}
    \caption{$t=0.6$}
  \end{subfigure}
  \begin{subfigure}[b]{0.25\textwidth}
    \centering
    \includegraphics[width=1.0\linewidth]{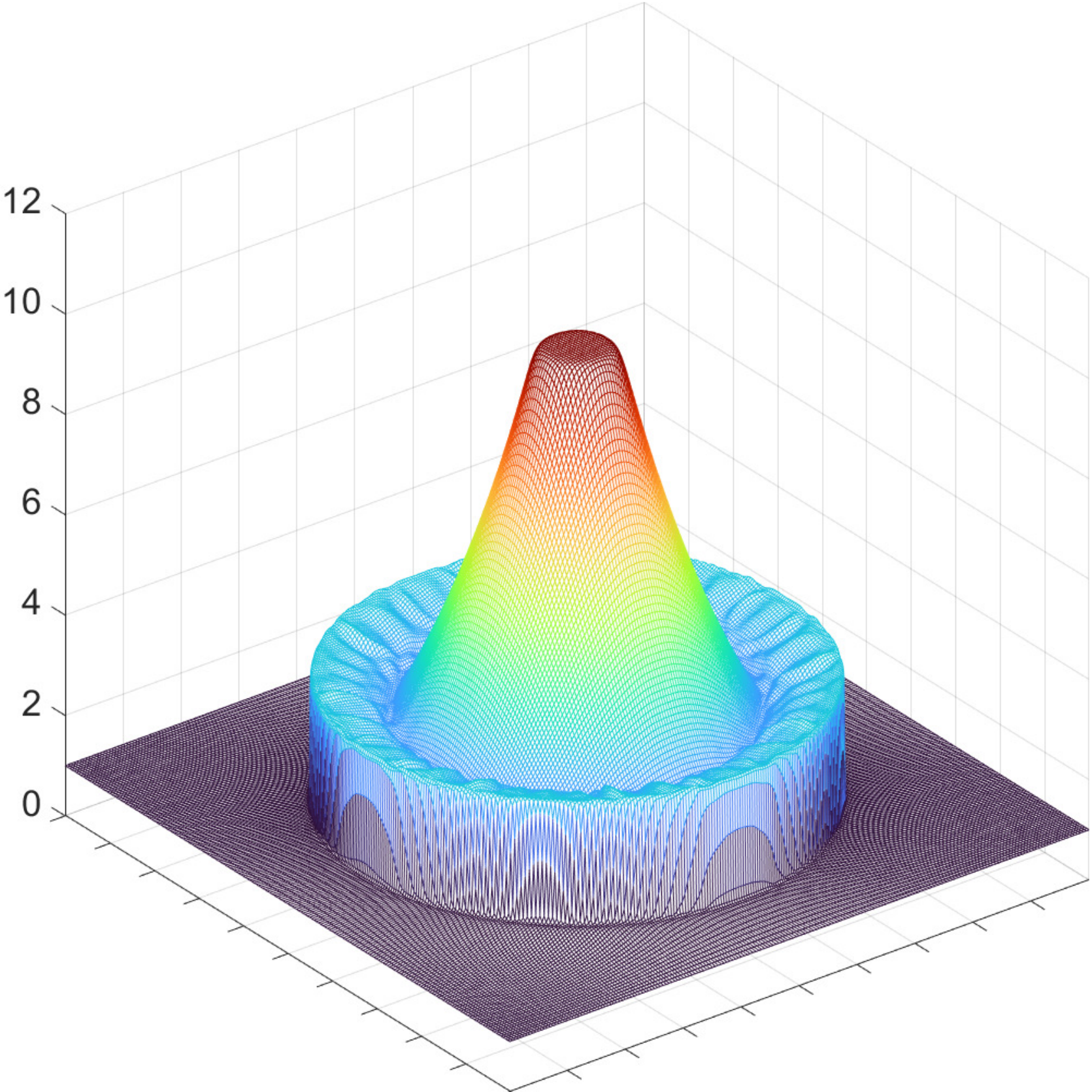}
    \caption{$t=0.8$}
  \end{subfigure}
  \begin{subfigure}[b]{0.25\textwidth}
    \centering
    \includegraphics[width=1.0\linewidth]{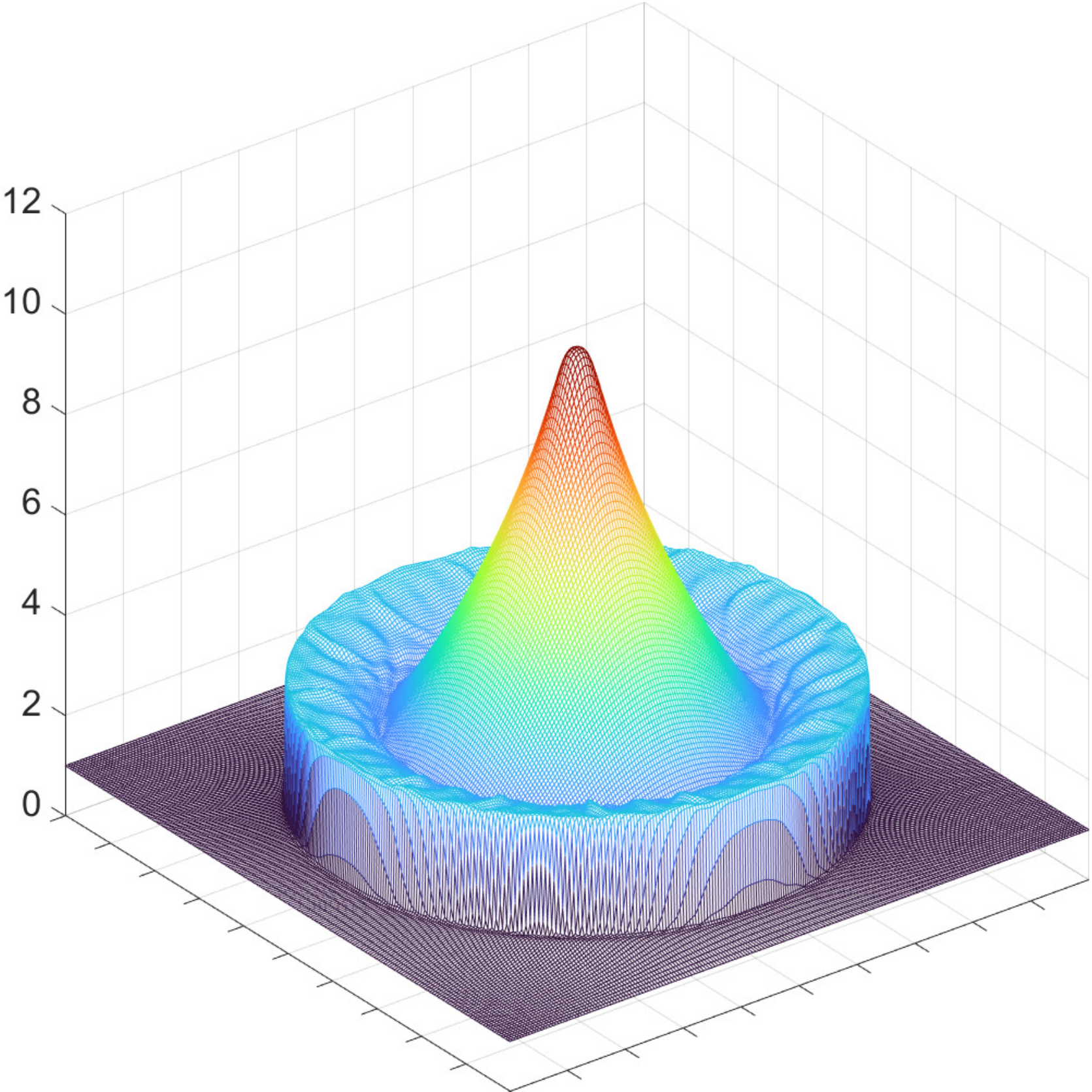}
    \caption{$t=1.0$}
  \end{subfigure}
  \caption{Example \ref{ex:Circle_Dam_flat}. The water surface level $h+b$ obtained by the {\tt MM-ES} scheme with $200\times200$ mesh at different times.}\label{fig:Circle_dam_flat}
\end{figure}

\begin{figure}[hbt!]
  \centering
  \begin{subfigure}[b]{0.25\textwidth}
    \centering
    \includegraphics[width=1.0\linewidth]{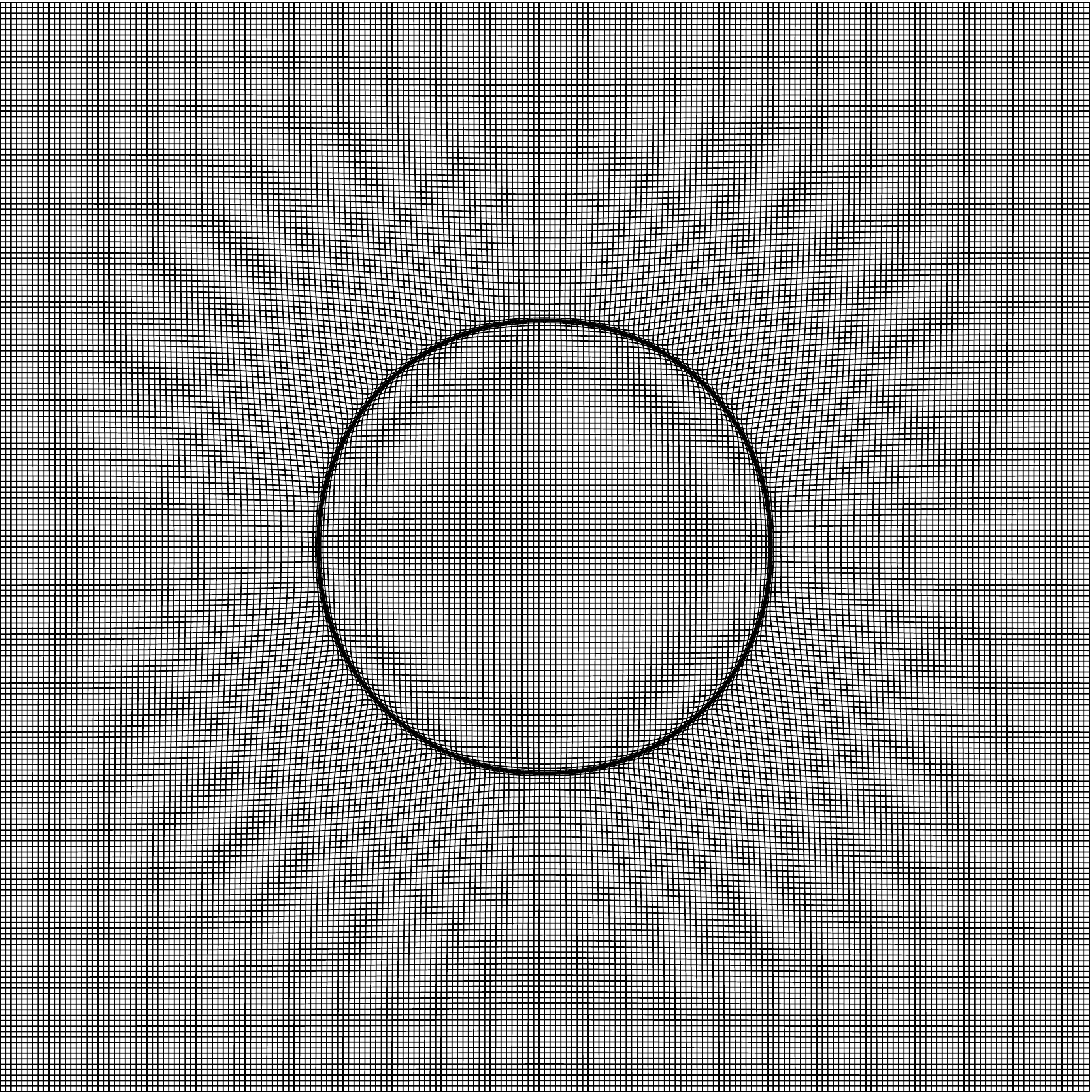}
    \caption{$t=0.0$}
  \end{subfigure}
  \begin{subfigure}[b]{0.25\textwidth}
    \centering
    \includegraphics[width=1.0\linewidth]{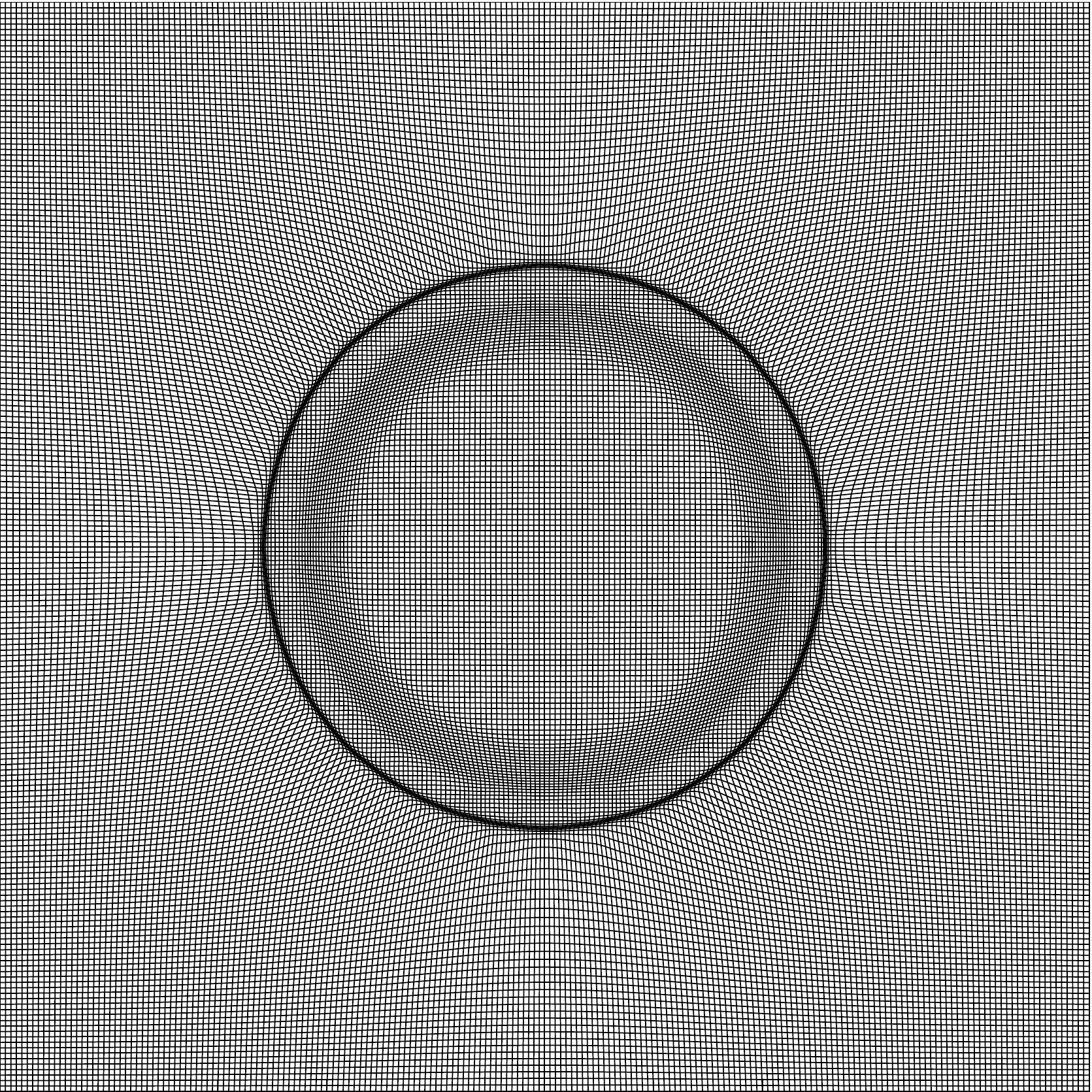}
    \caption{$t=0.2$}
  \end{subfigure}
  \begin{subfigure}[b]{0.25\textwidth}
    \centering
    \includegraphics[width=1.0\linewidth]{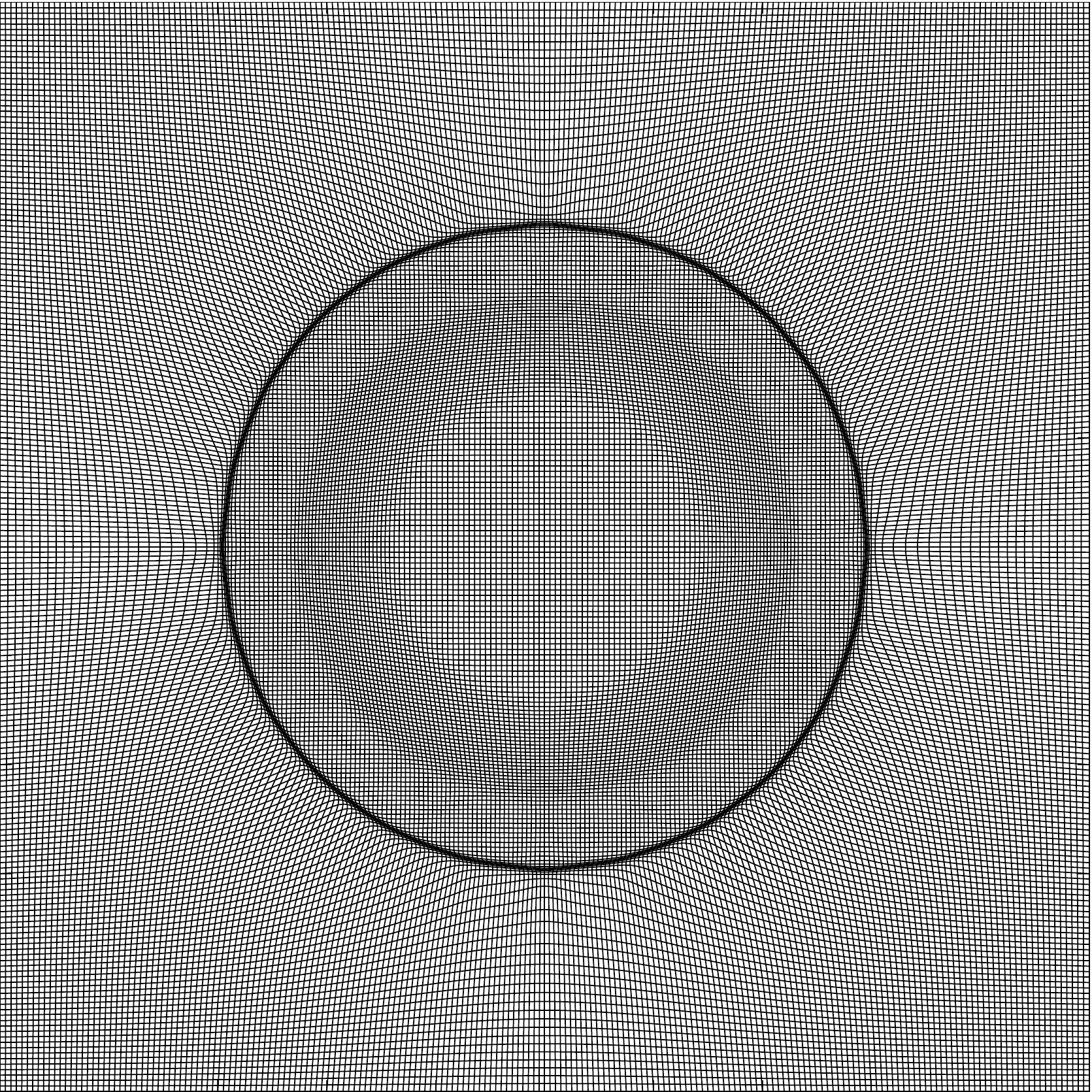}
    \caption{$t=0.4$}
  \end{subfigure}
  \quad\\
  \begin{subfigure}[b]{0.25\textwidth}
    \centering
    \includegraphics[width=1.0\linewidth]{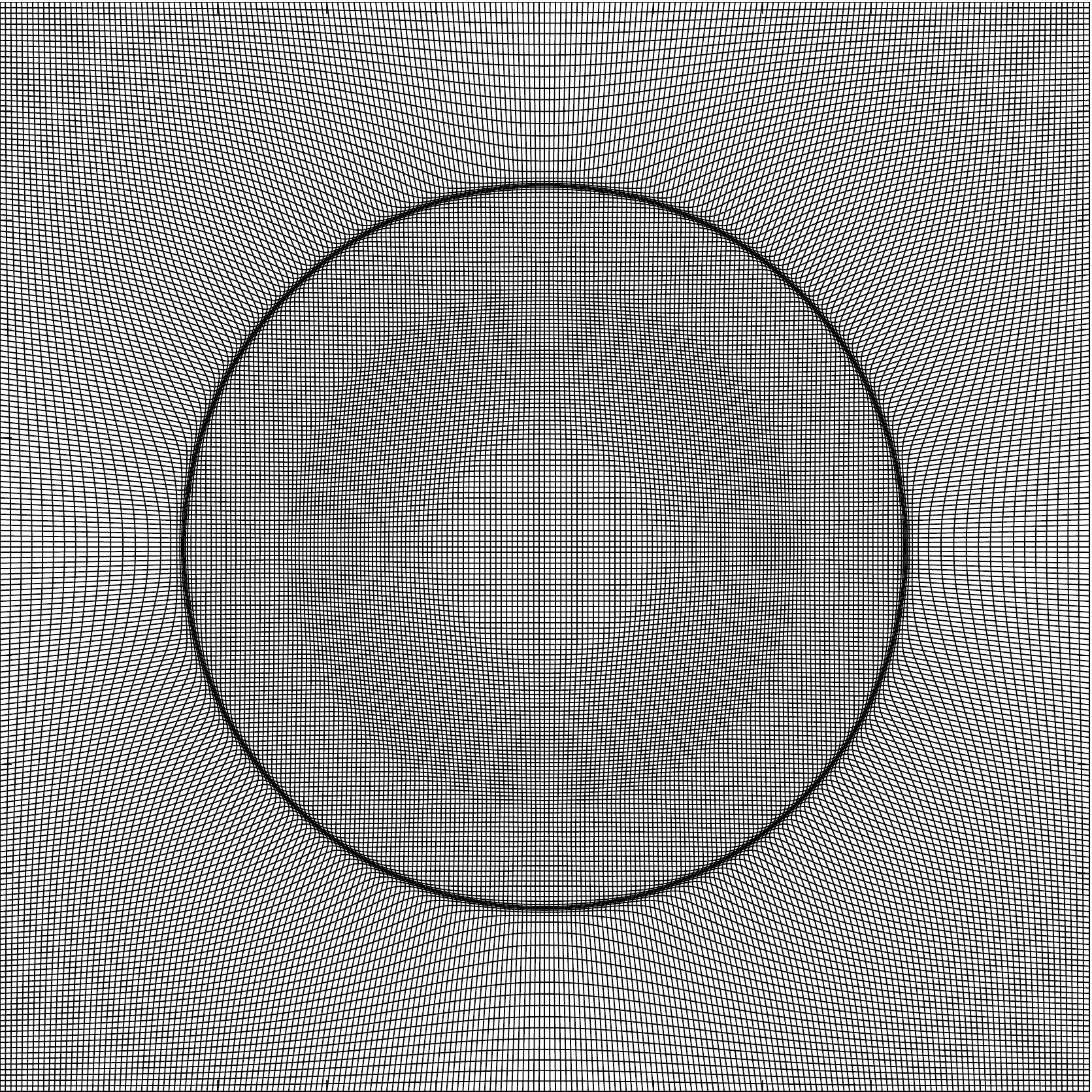}
    \caption{$t=0.6$}
  \end{subfigure}
  \begin{subfigure}[b]{0.25\textwidth}
    \centering
    \includegraphics[width=1.0\linewidth]{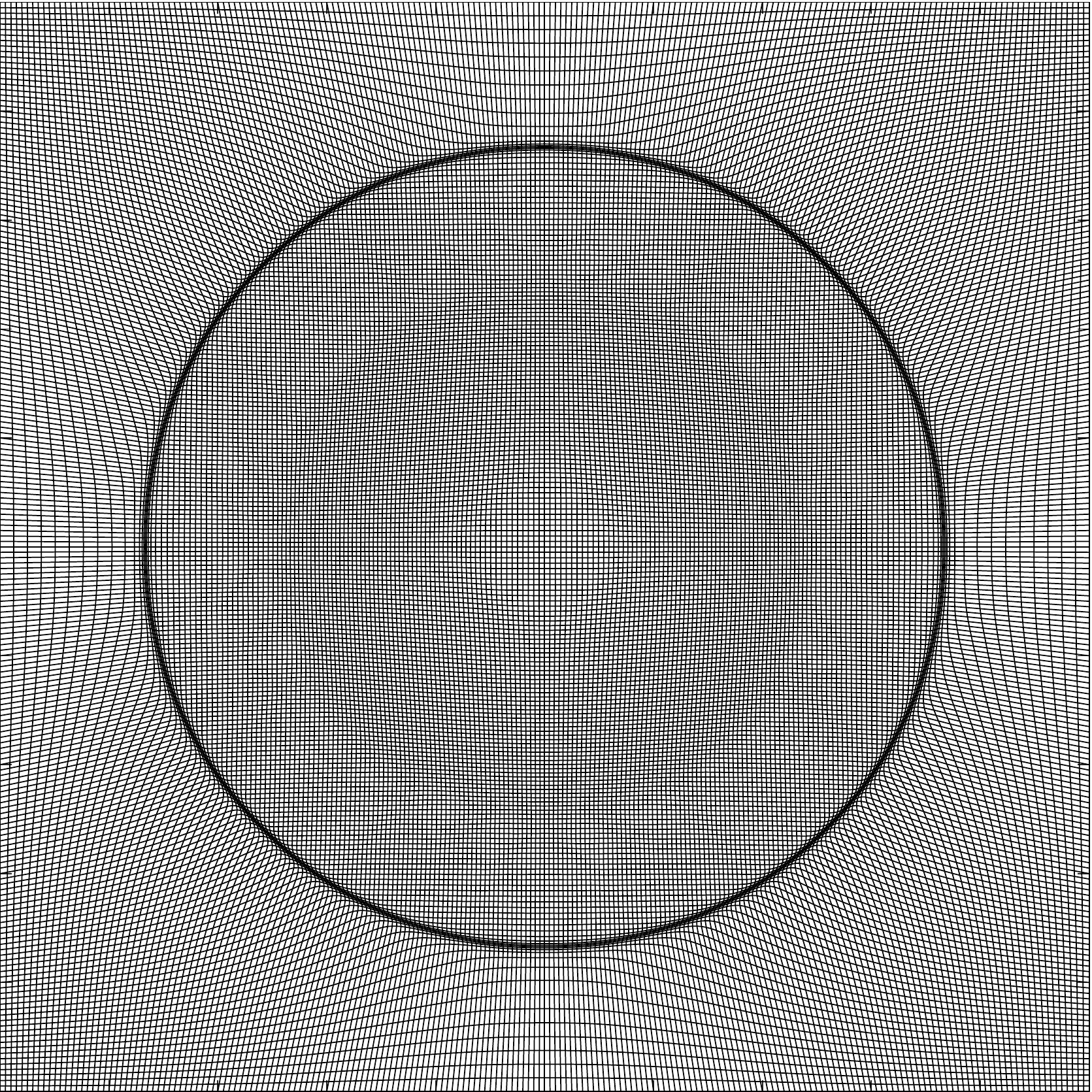}
    \caption{$t=0.8$}
  \end{subfigure}
  \begin{subfigure}[b]{0.25\textwidth}
    \centering
    \includegraphics[width=1.0\linewidth]{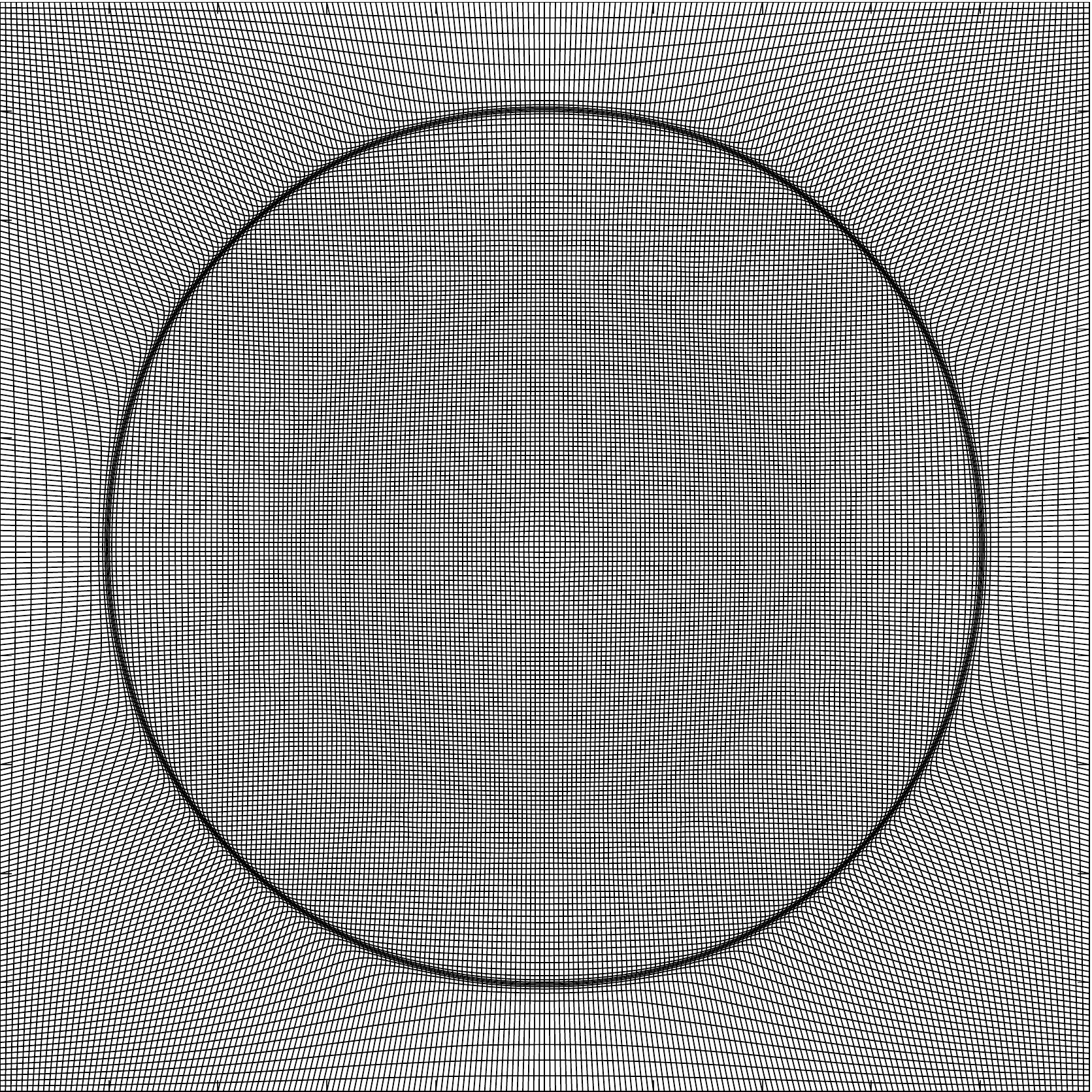}
    \caption{$t=1.0$}
  \end{subfigure}
  \caption{Example \ref{ex:Circle_Dam_flat}. The plots of $200\times200$ adaptive meshes obtained by the {\tt MM-ES} scheme at different times.}\label{fig:Circle_dam_flat_Mesh}
\end{figure}

\begin{example}[Circular dam break on a non-flat river bed]\label{ex:C_D_Break}\rm
  This test solves the circular dam break problem \cite{Capilla2013New,Castro2009High} for the 2D SWEs.
  The physical domain is $[0,2]\times[0,2]$ with the outflow boundary conditions and the bottom topography is defined as
  \begin{equation*}
    b(x_1, x_2) = \begin{cases}\dfrac{1}{8}(\cos (2 \pi(x_1-0.5))+1)(\cos (2 \pi x_2)+1), & \text{if}\quad \sqrt{(x_1-1.5)^2+(x_2-1)^2} \leqslant 0.5, \\ 0, & \text {otherwise}.\end{cases}
  \end{equation*}
  The initial water depth and velocities are
  \begin{equation*}
    \begin{aligned}
      & h = \begin{cases}
        1.1-b, & \text{if}~\sqrt{(x_1-1.25)^2+(x_2-1)^2} \leqslant 0.1, \\
        0.6-b, & \text {otherwise},
      \end{cases} \\
      & v_1=v_2=0.
    \end{aligned}
  \end{equation*}
  The output time is $t=0.15$ with $g = 9.812$.
  The monitor function is the same as that in Example \ref{ex:Pertubation}.
\end{example}

Figure \ref{fig:2D_Circle_Dam} shows the $200\times200$ adaptive mesh,
the water surface level $h+b$ obtained by
using the {\tt MM-ES} scheme,
and the comparison of the cut lines along $x_2 = 1$.
It is clear that the mesh points adaptively concentrate near the shock wave and localized spike in the center to increase the local resolution,
and the {\tt MM-ES} scheme gives sharp results near $x_1\approx 1.26$.

\begin{figure}[hbt!]
  \centering
  \begin{subfigure}[b]{0.3\textwidth}
    \centering
    \includegraphics[width=1.0\linewidth]{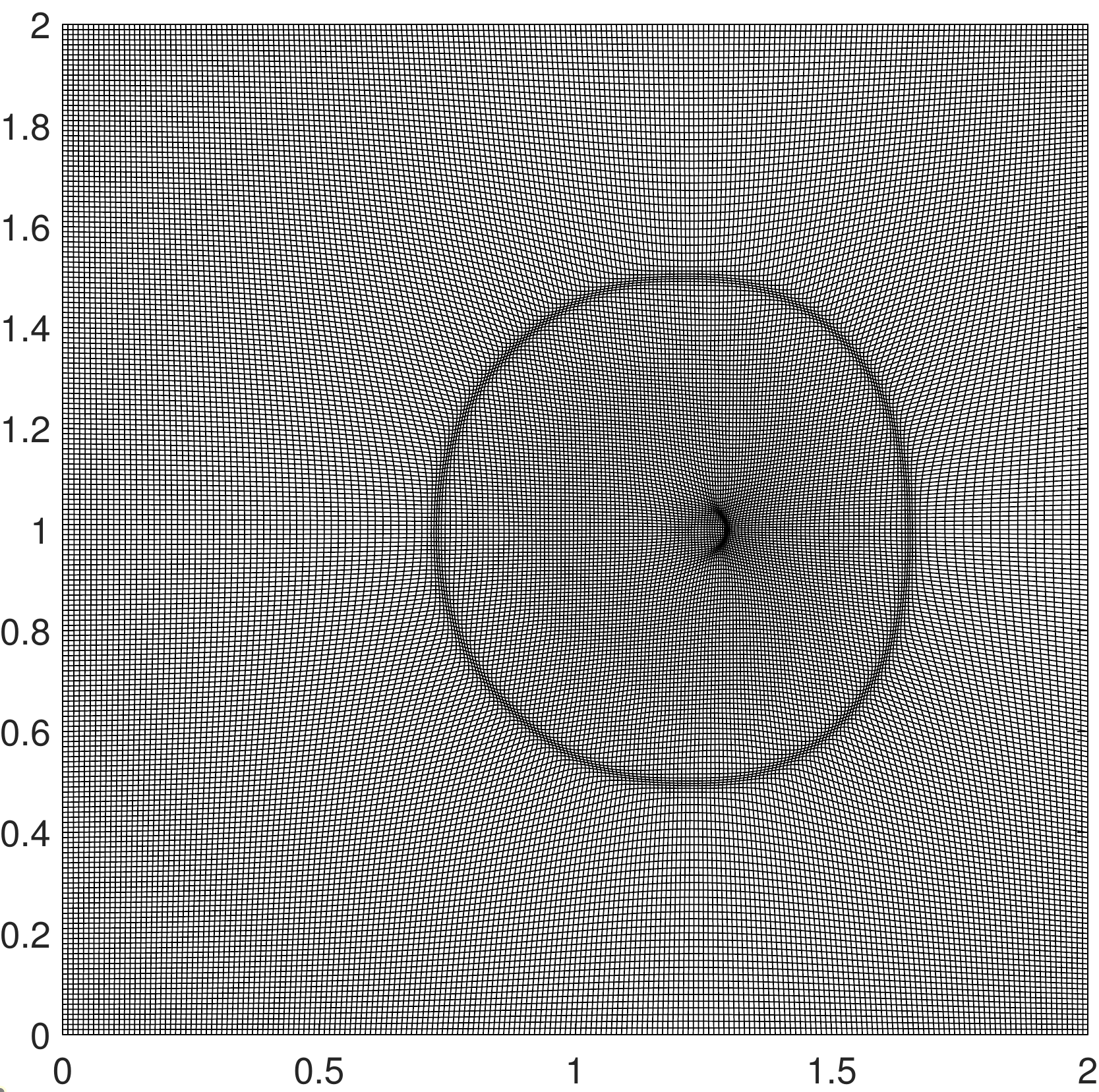}
    \caption{$200\times200$ adaptive mesh, {\tt MM-ES}}
  \end{subfigure}
  \begin{subfigure}[b]{0.3\textwidth}
    \centering
    \includegraphics[width=1.0\linewidth]{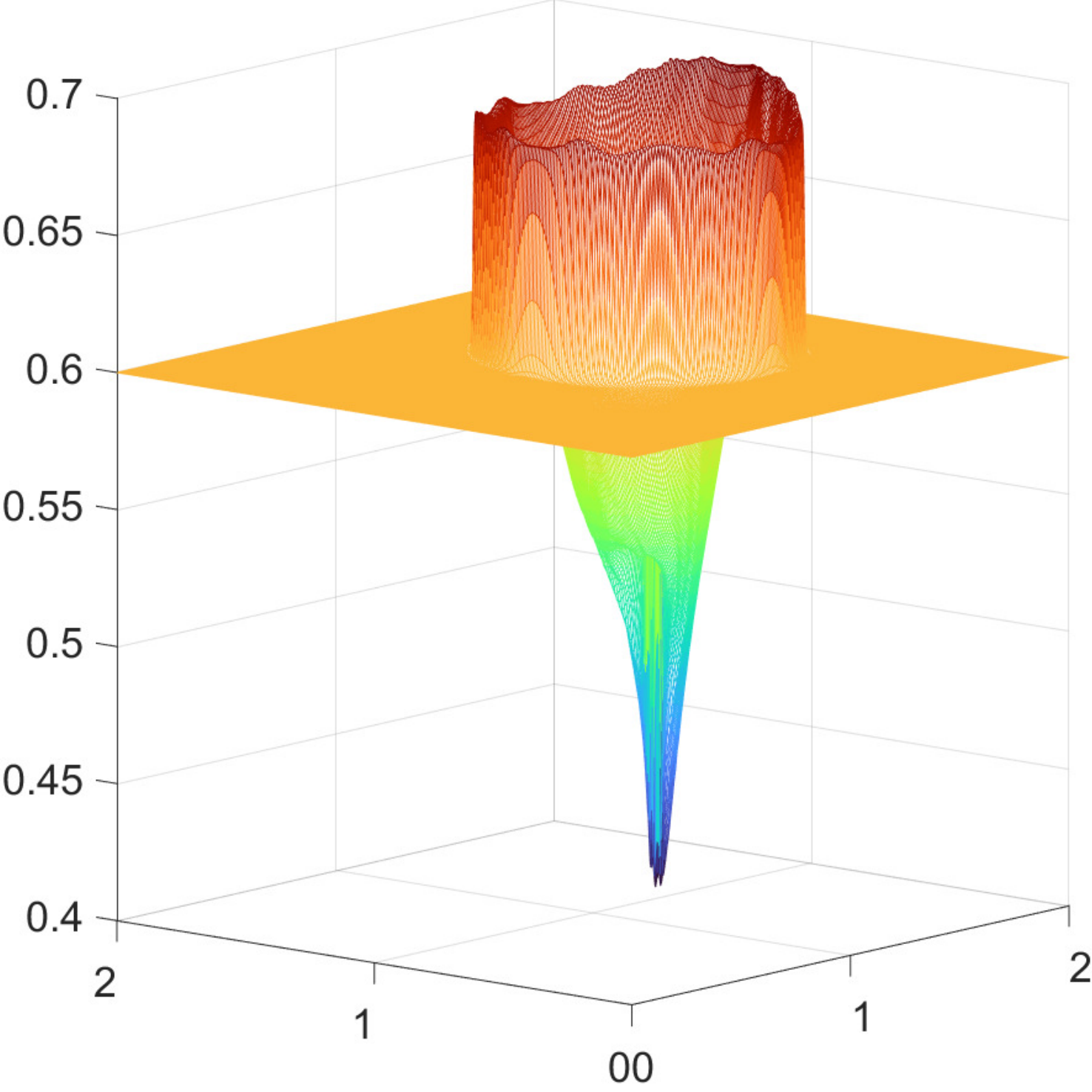}
    \caption{$h+b$, {\tt MM-ES}}
  \end{subfigure}
  \quad
  \begin{subfigure}[b]{0.3\textwidth}
    \centering
    \includegraphics[width=1.0\linewidth]{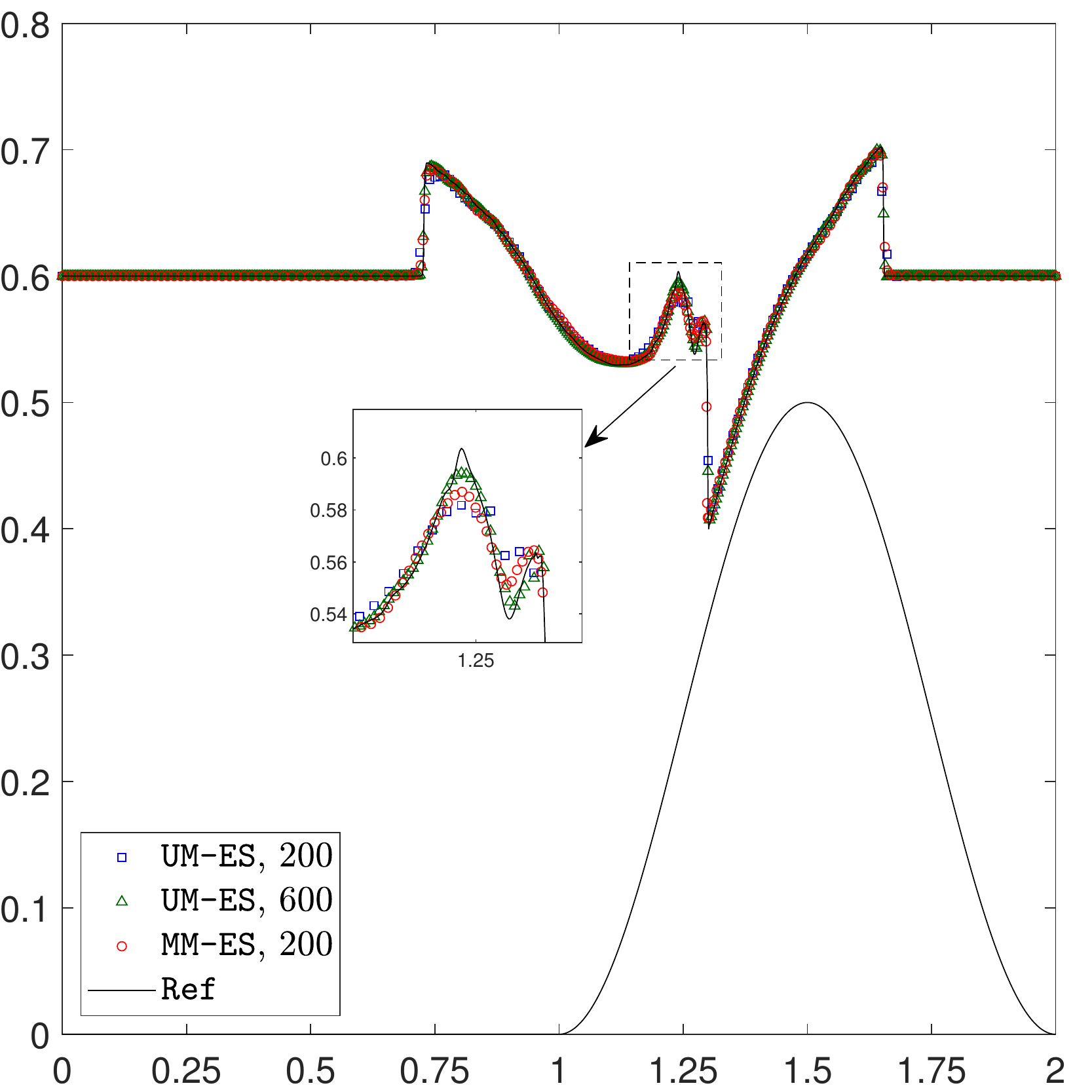}
    \caption{cut lines along $x_2=1$}
  \end{subfigure}
  \caption{Example \ref{ex:C_D_Break}. The results obtained by using the {\tt UM-ES} and {\tt MM-ES} schemes.}\label{fig:2D_Circle_Dam}
\end{figure}

\section{Conclusion}\label{section:Conc}
This paper presented high-order accurate well-balanced (WB) energy stable (ES) adaptive moving mesh finite difference schemes for the SWEs with non-flat bottom topography.
To construct our schemes on moving meshes,
the bottom topography was added as an additional conservative variable that evolved in time to reformulate the SWEs,
and energy inequality based on a modified energy function was derived,
then the reformulated SWEs and energy inequality were transformed into curvilinear coordinates.
The two-point energy conservative (EC) flux was constructed,
and the high-order EC schemes based on such EC flux were proved to preserve the lake at rest.
The newly designed dissipation terms were added to the EC schemes to get the high-order ES schemes,
which were also proved to be WB.
The mesh points were adaptively redistributed by iteratively solving the Euler-Lagrangian equations of a suitable mesh adaptation functional following \cite{Duan2022High,Li2022High}.
The fully-discrete schemes were obtained by using the explicit strong-stability preserving third-order Runge-Kutta method.
The numerical results have validated that our schemes can achieve the designed accuracy, preserve the lake at rest, and capture the wave structures accurately and efficiently.

\section*{Acknowledgments}
The authors were partially supported by
the National Key R\&D Program of China (Project Number 2020YFA0712000),
the National Natural Science Foundation of China (No. 12171227 \& 12288101).

\appendix
\section{1D high-order WB schemes}\label{section:OneDimensionCase}
This appendix presents the 1D semi-discrete WB EC schemes.
The ES schemes can be obtained by adding dissipation terms similar to Section \ref{section:ES} and are omitted here.
The modified 1D SWEs read
\begin{equation}\label{eq:1D_SWEs_Altered}
	\pd{\bU}{t} + \pd{\bF}{x} = - gh \pd{\bm{B}}{x},
\end{equation}
where $\bm{U} = (h,hv_1,b)^{\mathrm{T}}$,
$\bF = (hv_1, hv_1^2+\frac{1}{2}g h^2,0)^{\mathrm{T}}$,
$\bm{B} = (0,b,0)^{\mathrm{T}}$.
The system \eqref{eq:1D_SWEs_Altered} can be rewritten in curvilinear coordinates as
\begin{equation}\label{eq:1D_SWEs_Cur_Coordinate}
    \frac{\partial(J \bm{U})}{\partial \tau}+ \frac{\partial}{\partial \xi}\left(J \frac{\partial \xi}{\partial t} \bm{U}+ \bm{F}\right)=- gh \frac{\partial \bm{B}}{\partial \xi},
\end{equation}
with the VCL
\begin{equation*}
    \frac{\partial J}{\partial \tau}+ \frac{\partial}{\partial \xi}\left(J \frac{\partial \xi}{\partial t}\right)=0.
\end{equation*}
It is worth mentioning that $J \frac{\partial \xi}{\partial x} \equiv 1$,
so that the SCL holds automatically and \eqref{eq:1D_SWEs_Cur_Coordinate} has been simplified.
The two-point EC flux should satisfy the sufficient condition
\begin{align*}
\left(\bm{V}\left(\bU_{R}\right)-\bm{V}\left(\bU_{L}\right)\right)^{\mathrm{T}}\widetilde{\bm{\mathcal{F}}}
    = \ &~\frac{1}{2}\left(\left(J \frac{\partial \xi}{\partial t}\right)_{L}+\left(J \frac{\partial \xi}{\partial t}\right)_{R}\right)\left({\phi}\left(\bU_{R}\right)-{\phi}\left(\bU_{L}\right)\right)\nonumber\\
    &+ \left(\psi\left(\bU_R\right)-\psi\left(\bU_R\right)\right)
    - \frac{g}{2}\left(\left(hv_{1}\right)_{R}-\left(hv_{1}\right)_{L}\right)\left(b_{L}+b_{R}\right),
\end{align*}
with
\begin{equation*}
    \phi = \frac{1}{2}gh^2 + ghb+\gamma gb^2,~
    \psi = \frac{1}{2}gh^2v_1 + ghv_1 b,
\end{equation*}
and one choice is
\begin{equation*}
    \bm{\widetilde{\mathcal{{F}}}}\left(\bU_L,\bU_R,\left(J\frac{\partial \xi}{\partial t}\right)_{L},\left(J\frac{\partial \xi}{\partial t}\right)_{R}\right)
    = \frac12\left(\left(J\frac{\partial \xi}{\partial t}\right)_{L} + \left(J\frac{\partial \xi}{\partial t}\right)_{R}\right)
    \widetilde{\bm{U}} + \bm{\widetilde{F}},
\end{equation*}
with
\begin{equation*}
    \widetilde{\bm{U}}=\begin{pmatrix}
		\mean{h} \\
		\mean{h}\mean{v_1} \\
		\mean{b} \\
	\end{pmatrix},~
      \widetilde{\bm{F}}=\begin{pmatrix}
		\mean{h}\mean{v_1} \\
		\mean{h}\mean{v_1}^2 + \frac{g}{2}\mean{h^2}+g\left(\mean{hb} - \mean{h}\mean{b}\right) \\
		0 \\
	\end{pmatrix}.
\end{equation*}

Assume that a uniform mesh is taken as the computational mesh
$\xi_i = a + i\Delta\xi,~i=0,1,\cdots,N-1,~\Delta\xi = (b-a)/(N-1)$.
The $2p$th-order semi-discrete WB EC schemes can be expressed as
\begin{align*}
    &\frac{\mathrm{d}}{\mathrm{d} t}\bm{\mathcal{U}}_{i}
    =-\frac{1}{\Delta \xi} \left(\bm{\widetilde{\mathcal{{F}}}}_{i+\frac{1}{2}}^{\tt {2pth}} - \bm{\widetilde{\mathcal{{F}}}}_{i-\frac{1}{2}}^{\tt {2pth}}\right) - \dfrac{g{h}_{i}}{\Delta \xi}\left(\widetilde{\bm{B}}_{i+\frac{1}{2}}^{\tt {2pth}}-\widetilde{\bm{B}}_{i-\frac{1}{2}}^{\tt {2pth}}\right),\\
    &\frac{\mathrm{d}}{\mathrm{d} t} J_i=-\frac{1}{\Delta \xi}\left(\left(\widetilde{J\frac{\partial \xi}{\partial t}}\right)_{i+\frac{1}{2}}^{\tt 2pth}-\left(\widetilde{J\frac{\partial \xi}{\partial t}}\right)_{i-\frac{1}{2}}^{\tt 2pth}\right),
\end{align*}
where
\begin{equation*}
    \begin{aligned}
& \widetilde{\bm{\mathcal{F}}}_{i+\frac{1}{2}}^{\tt 2pth}=\sum_{m=1}^p \alpha_{p,m} \sum_{s=0}^{m-1}
\bm{\widetilde{\mathcal{{F}}}}\left(\bU_{i-s},\bU_{i-s+m},\left(J\frac{\partial \xi}{\partial t}\right)_{i-s},\left(J\frac{\partial \xi}{\partial t}\right)_{i-s+m}\right),
\\
& \widetilde{\bm{B}}_{i+\frac{1}{2}}^{\tt 2pth}=\sum_{m=1}^p \alpha_{p,m} \sum_{s=0}^{m-1} \frac{1}{2}\left(\left(\bm{B}\right)_{i-s}+\left(\bm{B}\right)_{i-s+m}\right), \\
& \left(\widetilde{J\frac{\partial \xi}{\partial t}}\right)_{i+\frac{1}{2}}^{\tt 2pth}=\sum_{m=1}^p \alpha_{p,m} \sum_{s=0}^{m-1} \frac{1}{2}\left(\left(J\dfrac{\partial \xi}{\partial t}\right)_{i-s}+\left(J\dfrac{\partial \xi}{\partial t}\right)_{i-s+m}\right).
\end{aligned}
\end{equation*}
Here $J\frac{\partial \xi}{\partial t} = -\dot{x}$,
and $\left(\dot{x}\right)_{i}$ is the mesh velocity at $\xi_{i}$.
The numerical solutions satisfy the semi-discrete energy identities
\begin{equation*}
    \frac{\mathrm{d}}{\mathrm{d} t}{\mathcal{E}}_{i}+\frac{1}{\Delta \xi}\left(\left({\widetilde{{\mathcal{Q}}}}\right)_{i+\frac{1}{2}}^{\tt 2pth}-\left({\widetilde{{\mathcal{Q}}}}\right)_{i-\frac{1}{2}}^{\tt 2pth}\right) = 0,
\end{equation*}
with ${\mathcal{E}}_{i} = J_i\eta(\bU_i)$,
and the numerical energy fluxes
  \begin{align*}
    \left(\widetilde{\mathcal{Q}}\right)_{i+\frac{1}{2}}^{\tt 2pth }=\sum_{m=1}^p \alpha_{p, m} \sum_{s=0}^{m-1} \widetilde{\mathcal{Q}}\left(\bm{U}_{i-s}, \bm{U}_{i-s+m},\left(J \frac{\partial \xi}{\partial t}\right)_{i-s},\left(J \frac{\partial \xi}{\partial t}\right)_{i-s+m}\right),
  \end{align*}
  where
  \begin{align*}
    \widetilde{\mathcal{Q}} =
    \ &\frac{1}{2}\left(\bm{V}(\bU_{L})+\bm{V}(\bU_{R})\right)^{\mathrm{T}}\widetilde{\bm{\mathcal{F}}}
    - \frac{1}{2}\left(\left(J \frac{\partial \xi}{\partial t}\right)_L+\left(J \frac{\partial \xi}{\partial t}\right)_R\right)\left(\phi\left(\bU_{L}\right)+\phi\left(\bU_{R}\right)\right) \nonumber \nonumber\\
    &- \frac{1}{2}\left(\psi\left(\bU_{L}\right)+\psi\left(\bU_{R}\right)\right) + \frac{g}{4} \left(\left(hv_1\right)_{L}+\left(hv_1\right)_{R}\right)\left(b_{L}+b_{R}\right).
  \end{align*}
The time stepsize $\Delta t^{n}$ is given by the CFL condition
\begin{equation*}
    \Delta t^{n} \leqslant \frac{C_{\text{\tiny \tt CFL}}}{\max_{i}\limits\mathcal{\varrho}^{n}_{i}/\Delta \xi},
\end{equation*}
where
\begin{equation*}
    \mathcal{\varrho}_{i}^{n} = \max\left\{\Bigg| J\frac{\partial\xi}{\partial t}+\lambda_{1}(\bm{U})\Bigg|_{i}^n,~
    \Bigg| J\frac{\partial\xi}{\partial t}+\lambda_{2}(\bm{U})\Bigg|_{i}^n\right\}/J_{i}^n,
\end{equation*}
with $\lambda_1= v_1+c,~\lambda_2 = v_1-c,~c = \sqrt{gh}$.
In the accuracy test, the time stepsize is taken as $C_{\text{\tiny \tt CFL}} (\Delta\xi)^{6/3}$ and $C_{\text{\tiny \tt CFL}} (\Delta\xi)^{5/3}$ to make the spatial error dominant for the EC and ES schemes, respectively.



\begin{thebibliography}{10}
	
	\bibitem{Alcrudo1993High}
	F.~Alcrudo and P.~Garcia-Navarro, {A high-resolution Godunov-type scheme in
		finite volumes for the 2D shallow-water equations}, \emph{Int. J. Numer. Meth
		Fluids}, 16 (1993), 489--505.
	
	\bibitem{AuDusse2004fast}
	E.~AuDusse, {Francois Bouchut}, M.O. Bristeau, R.~Klein, and B.~Perthame, {A
		fast and stable well-balanced scheme with hydrostatic reconstruction for
		shallow water flows}, \emph{SIAM J. Sci. Comput.}, 25 (2004), 2050--2065.
	
	\bibitem{Bermudez1994Upwind}
	A.~Bermudez and M.E. Vazquez, Upwind methods for hyperbolic conservation laws
	with source terms, \emph{Computers \& Fluids}, 23 (1994), 1049--1071.
	
	\bibitem{Biswas2018Low}
	B.~Biswas and R.K. Dubey, {Low dissipative entropy stable schemes using third
		order WENO and TVD reconstructions}, \emph{Adv. Comput. Math.}, 44 (2018),
	1153--1181.
	
	\bibitem{Borges2008An}
	R.~Borges, M.~Carmona, B.~Costa, and W.S. Don, {An improved weighted
		essentially non-oscillatory scheme for hyperbolic conservation laws},
	\emph{J. Comput. Phys.}, 227 (2008), 3191--3211.
	
	\bibitem{Brackbill1993An}
	J.U. Brackbill, An adaptive grid with directional control, \emph{J. Comput.
		Phys.}, 108 (1993), 38--50.
	
	\bibitem{Brackbill1982Adaptive}
	J.U. Brackbill and J.S. Saltzman, Adaptive zoning for singular problems in two
	dimensions, \emph{J. Comput. Phys.}, 46 (1982), 342--368.
	
	\bibitem{CAO1999221}
	W.M. Cao, W.Z. Huang, and R.D. Russell, An r-adaptive finite element method
	based upon moving mesh {PDE}s, \emph{J. Comput. Phys.}, 149 (1999), 221--244.
	
	\bibitem{Capilla2013New}
	M.T. Capilla and A.~Balaguer-Beser, A new well-balanced non-oscillatory central
	scheme for the shallow water equations on rectangular meshes, \emph{J.
		Comput. Appl. Math.}, 252 (2013), 62--74.
	
	\bibitem{Castro2009High}
	M.J. Castro, E.D. Fern\'{a}ndez-Nieto, A.M. Ferreiro, J.A.
	Garc\'{\i}a-Rodr\'{\i}guez, and C.~Par\'{e}s, High order extensions of {R}oe
	schemes for two-dimensional nonconservative hyperbolic systems, \emph{J. Sci.
		Comput.}, 39 (2009), 67--114.
	
	\bibitem{CENICEROS2001609}
	H.D. Ceniceros and T.Y. Hou, An efficient dynamically adaptive mesh for
	potentially singular solutions, \emph{J. Comput. Phys.}, 172 (2001),
	609--639.
	
	\bibitem{Davis1982}
	S.F. Davis and J.E. Flaherty, An adaptive finite element method for
	initial-boundary value problems for partial differential equations,
	\emph{SIAM J. Sci. Stat. Comput.}, 3 (1982), 6--27.
	
	\bibitem{Duan2020RMHD}
	J.M. Duan and H.Z. Tang, {High-order accurate entropy stable nodal
		discontinuous Galerkin schemes for the ideal special relativistic
		magnetohydrodynamics}, \emph{J. Comput. Phys.}, 421 (2020), 109731.
	
	\bibitem{DUAN2021109949}
	J.M. Duan and H.Z. Tang, Entropy stable adaptive moving mesh schemes for 2{D}
	and 3{D} special relativistic hydrodynamics, \emph{J. Comput. Phys.}, 426
	(2021), 109949.
	
	\bibitem{Duan2021SWMHD}
	J.M. Duan and H.Z. Tang, High-order accurate entropy stable finite difference
	schemes for the shallow water magnetohydrodynamics, \emph{J. Comput. Phys.},
	431 (2021), 110136.
	
	\bibitem{Duan2022High}
	J.M. Duan and H.Z. Tang, High-order accurate entropy stable adaptive moving
	mesh finite difference schemes for special relativistic
	(magneto)hydrodynamics, \emph{J. Comput. Phys.}, 456 (2022), 111038.
	
	\bibitem{Fjordholm2011Well}
	U.S. Fjordholm, S.~Mishra, and E.~Tadmor, {Well-balanced and energy stable
		schemes for the shallow water equations with discontinuous topography},
	\emph{J. Comput. Phys.}, 230 (2011), 5587--5609.
	
	\bibitem{Godunov1972}
	S.K. Godunov, Symmetric form of the equations of magnetohydrodynamics,
	\emph{Numer. Meth. Mech. Cont. Medium}, 1 (1972), 26--34.
	
	\bibitem{Kuang2017Runge}
	Y.Y. Kuang, K.L. Wu, and H.Z. Tang, Runge-{K}utta discontinuous local evolution
	{G}alerkin methods for the shallow water equations on the cubed-sphere grid,
	\emph{Numer. Math. Theor. Meth. Appl.}, 10 (2017), 373--419.
	
	\bibitem{Kurganov2018Finite}
	A.~Kurganov, {Finite-volume schemes for shallow-water equations}, \emph{Acta
		Numer.}, 27 (2018), 289--351.
	
	\bibitem{Lamby2005Solution}
	P.~Lamby, S.~M\"{u}ller, and Y.~Stiriba, Solution of shallow water equations
	using fully adaptive multiscale schemes, \emph{Internat. J. Numer. Methods
		Fluids}, 49 (2005), 417--437.
	
	\bibitem{Lefloch2002Fully}
	P.G. LeFloch, J.M. Mercier, and C.~Rohde, {Fully discrete entropy conservative
		schemes of arbitraty order}, \emph{SIAM J. Numer. Anal.}, 40 (2002),
	1968--1992.
	
	\bibitem{Leveque1998Balancing}
	R.J. LeVeque, Balancing source terms and flux gradients in high-resolution
	{G}odunov methods: the quasi-steady wave-propagation algorithm, \emph{J.
		Comput. Phys.}, 146 (1998), 346--365.
	
	\bibitem{Li2012Hybrid}
	G.~Li, C.N. Lu, and J.X. Qiu, Hybrid well-balanced {WENO} schemes with
	different indicators for shallow water equations, \emph{J. Sci. Comput.}, 51
	(2012), 527--559.
	
	\bibitem{Li2020High}
	P.~Li, W.S. Don, and Z.~Gao, High order well-balanced finite difference {WENO}
	interpolation-based schemes for shallow water equations, \emph{Comput. \&
		Fluids}, 201 (2020), 104476.
	
	\bibitem{Li2022High}
	S.T. Li, J.M. Duan, and H.Z. Tang, High-order accurate entropy stable adaptive
	moving mesh finite difference schemes for (multi-component) compressible
	{E}uler equations with the stiffened equation of state, \emph{Comput. Methods
		Appl. Mech. Engrg.}, 399 (2022), 115311.
	
	\bibitem{Miller1981}
	K.~Miller, Moving finite elements. {II}, \emph{SIAM J. Numer. Anal.}, 18
	(1981), 1033--1057.
	
	\bibitem{Noelle2006Well}
	S.~Noelle, N.~Pankratz, G.~Puppo, and J.R. Natvig, Well-balanced finite volume
	schemes of arbitrary order of accuracy for shallow water flows, \emph{J.
		Comput. Phys.}, 213 (2006), 474--499.
	
	\bibitem{Noelle2007High}
	S.~Noelle, Y.L. Xing, and C.W. Shu, High-order well-balanced finite volume
	{WENO} schemes for shallow water equation with moving water, \emph{J. Comput.
		Phys.}, 226 (2007), 29--58.
	
	\bibitem{Powell1994}
	K.G. Powell, {An approximate riemann solver for magnetohydrodynamics (that
		works in more than one dimension)}, \emph{ICASE 94-24},  (1994).
	
	\bibitem{Ren2000An}
	W.Q. Ren and X.P. Wang, An iterative grid redistribution method for singular
	problems in multiple dimensions, \emph{J. Comput. Phys.}, 159 (2000),
	246--273.
	
	\bibitem{Stockie2001}
	J.M. Stockie, J.A. Mackenzie, and R.D. Russell, A moving mesh method for
	one-dimensional hyperbolic conservation laws, \emph{SIAM J. Sci. Comput.}, 22
	(2001), 1791--1813.
	
	\bibitem{Tang2004Solution}
	H.Z. Tang, Solution of the shallow-water equations using an adaptive moving
	mesh method, \emph{Internat. J. Numer. Methods Fluids}, 44 (2004), 789--810.
	
	\bibitem{Tang2003Adaptive}
	H.Z. Tang and T.~Tang, Adaptive mesh methods for one- and two-dimensional
	hyperbolic conservation laws, \emph{SIAM J. Numer. Anal.}, 41 (2003),
	487--515.
	
	\bibitem{Tang2004Gas}
	H.Z. Tang, T.~Tang, and K.~Xu, A gas-kinetic scheme for shallow-water equations
	with source terms, \emph{Z. Angew. Math. Phys.}, 55 (2004), 365--382.
	
	\bibitem{Vukovic2002ENO}
	S.~Vukovic and L.~Sopta, E{NO} and {WENO} schemes with the exact conservation
	property for one-dimensional shallow water equations, \emph{J. Comput.
		Phys.}, 179 (2002), 593--621.
	
	\bibitem{Wang2004A}
	D.S. Wang and X.P. Wang, A three-dimensional adaptive method based on the
	iterative grid redistribution, \emph{J. Comput. Phys.}, 199 (2004), 423--436.
	
	\bibitem{Winslow1967Numerical}
	A.M. Winslow, Numerical solution of the quasilinear {P}oisson equation in a
	nonuniform triangle mesh, \emph{J. Comput. Phys.}, 1 (1967), 149--172.
	
	\bibitem{Wu2020Entropy}
	K.L. Wu and C.W. Shu, Entropy symmetrization and high-order accurate entropy
	stable numerical schemes for relativistic {MHD} equations, \emph{SIAM J. Sci.
		Comput.}, 42 (2020), A2230--A2261.
	
	\bibitem{Wu2016Newton}
	K.L. Wu and H.Z. Tang, {A Newton multigrid method for steady-state shallow
		water equations with topography and dry areas}, \emph{Appl. Math. Mech.}, 37
	(2016), 1441--1466.
	
	\bibitem{Xing2014Exactly}
	Y.L. Xing, Exactly well-balanced discontinuous {G}alerkin methods for the
	shallow water equations with moving water equilibrium, \emph{J. Comput.
		Phys.}, 257 (2014), 536--553.
	
	\bibitem{Xing2017Numerical}
	Y.L. Xing, Numerical methods for the nonlinear shallow water equations, in
	\emph{Handbook of numerical methods for hyperbolic problems}, vol.~18 of
	\emph{Handb. Numer. Anal.}, pages 361--384, Elsevier/North-Holland,
	Amsterdam, 2017.
	
	\bibitem{Xing2005High}
	Y.L. Xing and C.W. Shu, {High order finite difference WENO schemes with the
		exact conservation property for the shallow water equations}, \emph{J.
		Comput. Phys.}, 208 (2005), 206--227.
	
	\bibitem{Xing2013Positivity}
	Y.L. Xing and X.X. Zhang, Positivity-preserving well-balanced discontinuous
	{G}alerkin methods for the shallow water equations on unstructured triangular
	meshes, \emph{J. Sci. Comput.}, 57 (2013), 19--41.
	
	\bibitem{Zhang2021High}
	M.~Zhang, W.Z. Huang, and J.X. Qiu, {A high-order well-balanced
		positivity-preserving moving mesh DG method for the shallow water equations
		with non-flat bottom topography}, \emph{J. Sci. Comput.}, 87 (2021), 1--43.
	
	\bibitem{Zhang2022AWell}
	M.~Zhang, W.Z. Huang, and J.X. Qiu, A well-balanced positivity-preserving
	quasi-{L}agrange moving mesh {DG} method for the shallow water equations,
	\emph{Commun. Comput. Phys.}, 31 (2022), 94--130.
	
	\bibitem{Zhao2022Well}
	Z.~Zhao and M.~Zhang, Well-balanced fifth-order finite difference {H}ermite
	{WENO} scheme for the shallow water equations, \emph{J. Comput. Phys.}, 475
	(2023), 111860.
	
\end{thebibliography}


\end{document}